\documentclass[11pt]{article}

\usepackage{amsmath,amssymb,amsthm,mathtools}
\usepackage{bm}
\usepackage{geometry}
\usepackage{enumitem}
\usepackage{hyperref}
\usepackage{booktabs}
\usepackage{tabularx}
\usepackage{placeins}
\usepackage{graphicx}
\usepackage{xcolor}
\usepackage{tikz}
\usetikzlibrary{arrows.meta,decorations.markings,decorations.pathreplacing,calc}
\newtheorem{theorem}{Theorem}[section]
\newtheorem{proposition}[theorem]{Proposition}
\newtheorem{lemma}[theorem]{Lemma}
\newtheorem{corollary}[theorem]{Corollary}
\newtheorem{definition}[theorem]{Definition}
\newtheorem{assumption}[theorem]{Assumption}
\newtheoremstyle{bodyupright}
  {}{}
  {\normalfont}
  {}
  {\bfseries}
  {.}
  {5pt plus 1pt minus 1pt}
  {}
\theoremstyle{bodyupright}
\newtheorem{example}[theorem]{Example}
\newtheorem{remark}[theorem]{Remark}
\theoremstyle{plain}

\newcommand{\RR}{\mathbb{R}}
\newcommand{\CC}{\mathbb{C}}

\newcommand{\HK}{\mathcal{H}_K}

\newcommand{\Int}{\mathcal I}
\newcommand{\MP}{\mathrm{MP}}
\newcommand{\SE}{\mathrm{SE}}
\newcommand{\rhoInt}{\rho^{\Int}}

\newcommand{\HInt}{\check H^{\Int}}

\DeclareMathOperator{\Hess}{Hess}

\newcommand{\bfi}{\bfseries\itshape}

\newif\ifshownotes\shownotestrue   

\title{Learning Hamiltonians from Trajectory Data with Kernel Methods}
\author{%
Jianyu Hu\textsuperscript{1},
David Mart\'in de Diego\textsuperscript{2},
Juan-Pablo Ortega\textsuperscript{1}, and
Miguel Vaquero\textsuperscript{3}\\[0.6em]
\small \textsuperscript{1}Division of Mathematical Sciences, School of Physical and Mathematical Sciences,\\[-0.1em]
\small Nanyang Technological University, Singapore\\
\small \textsuperscript{2}Instituto de Ciencias Matem\'aticas (ICMAT), CSIC--UAM--UC3M--UCM, Madrid, Spain\\
\small \textsuperscript{3}School of Science and Technology, IE University, Madrid, Spain\\[0.4em]
\small \texttt{jianyu.hu@ntu.edu.sg}; \texttt{juan-pablo.ortega@ntu.edu.sg};\\[-0.1em]
\small \texttt{david.martin@icmat.es}; \texttt{miguel.vaquero@ie.edu}
}
\date{\today}

\begin{document}
\maketitle

\begin{abstract}
We study the recovery of a Hamiltonian from discrete trajectory observations using structure-preserving kernel methods. When exact-flow data are fitted through a symplectic one-step scheme, the fixed-step learning target is generally not the physical Hamiltonian itself, but a scheme-dependent inverse modified Hamiltonian. For the implicit midpoint and $p$-implicit symplectic Euler methods, we establish the existence of analytic inverse targets and show that their discrepancies from the physical Hamiltonian are second and first order in the observation step, respectively.
Within an RKHS framework, we derive finite-sample reconstruction bounds for random trajectories, allowing dependence among observations from the same trajectory, and separate the error into estimation, approximation, and integrator-induced components. We further establish stability under noisy trajectory observations and develop a multistride Richardson correction that systematically raises the order of the integrator-induced error without requiring additional trajectories.
Numerical experiments confirm the predicted scheme-dependent orders, the effect of Richardson correction, and the applicability of the method across several Hamiltonian systems.
\end{abstract}

\tableofcontents

\section{Introduction}

Hamiltonian systems provide a fundamental mathematical description of conservative dynamics in mechanics and many other areas of science \cite{Abraham1978,Marsden1994,arnol2013mathematical}.  
In canonical coordinates $z=(q,p)\in\mathbb{R}^{2d}$, the dynamics generated by a Hamiltonian $H$ is governed by
the equation \begin{equation}\label{eq:intro-hamiltonian}
\dot z = J\nabla H(z),
\end{equation}
where $J$ is the canonical symplectic matrix.  The increasing availability of trajectory data has motivated proposals to solve the inverse problem of recovering $H$, or the dynamics it generates, from observations of the system.  
Structure preservation is an important consideration in this context: the learned dynamics should retain its Hamiltonian structure and the treatment of discrete trajectory observations should respect the underlying symplectic geometry.
In this work, {\it we study the recovery of Hamiltonians from discrete trajectory observations using structure-preserving kernel methods}. Our main goal is to {\it understand how the symplectic discretization underlying the learning procedure affects the Hamiltonian identified from the data and to quantify the quality of its recovery from finitely many discrete-time trajectories using a reproducing kernel Hilbert space (RKHS) as the concept class}.

The need to preserve structure is clearest in the classical field of {\bfi geometric numerical integration}, where the analogous question was settled long ago. A symplectic method of order $r$ applied to \eqref{eq:intro-hamiltonian} with step $\delta$ is not the exact flow of $H$, but, for analytic $H$ and after optimal truncation of a generally divergent series, it agrees up to an exponentially small remainder with the time-$\delta$ flow of a {\bfi modified Hamiltonian} $\widetilde H_\delta=H+\mathcal O(\delta^r)$. Consequently, as long as the numerical trajectory remains in a fixed compact set, the energy error stays of size $\mathcal O(\delta^r)$ over times of order $e^{c/\delta}$, and, for variational discretizations, the momentum maps associated with symmetries are conserved exactly through the discrete Noether theorem \cite{benettin1994hamiltonian,HLW06,sanzserna1994numerical,leimkuhler2004simulating}. A non-symplectic method of equal or higher order enjoys no such protection: its energy error grows secularly, typically linearly in $t$, so that a locally more accurate scheme is eventually overtaken by a less accurate structure-preserving one. Over long horizons the decisive quantity is not the local truncation error but the presence or absence of the structure.

Learning a vector field from data is a discretization of a different kind, and the same dichotomy applies. For an unconstrained estimate $\widehat X$ of $J\nabla H$, differentiating $H$ along the induced flow $\widehat z_t$ gives only
\begin{equation*}
\bigl|H(\widehat z_t)-H(\widehat z_0)\bigr|\ \le\ t\,\sup\|\nabla H\|\,\bigl\|\widehat X-J\nabla H\bigr\|_\infty ,
\end{equation*}
and this linear growth is generically attained, exactly as for a non-geometric integrator: a field error of size $\varepsilon$ becomes an energy error of size $\varepsilon t$. If instead the hypothesis space is $\{J\nabla h\mid h\in\mathcal H_K\}$, so that the learned field is Hamiltonian by construction with learned Hamiltonian $\widehat h$, then $\widehat h(\widehat z_t)=\widehat h(\widehat z_0)$ for every $t$ in the maximal interval of existence, and therefore, for every $c\in\mathbb R$,
\begin{equation*}
\bigl|H(\widehat z_t)-H(\widehat z_0)\bigr|\ \le\ 2\sup_{z \in \mathcal L}\bigl|H(z)-\widehat h(z)-c\bigr| ,
\end{equation*}
where $\mathcal L$ is the connected component of $\{\widehat h=\widehat h(\widehat z_0)\}$ containing $\widehat z_0$. The true energy is not conserved, but its deviation is bounded uniformly in $t$ by the approximation error of $\widehat h$ on a single level set, with no secular growth, and symplecticity, Liouville volume and, for invariant $\widehat h$, the corresponding conservation laws hold exactly for the learned dynamics. The guarantee comes from constraining the object that is computed rather than the loss used to compute it, and two estimators with identical training residuals may differ completely in their long-time behavior. The kernel formulation adopted here imposes the constraint at that level; the question addressed in this work is which $\widehat h$ the data then select, once the observations are discrete in time.

A variety of structure-preserving learning methods have been developed for Hamiltonian systems; see \cite{celledoni2021structure} for a survey of the broader program of encoding geometric structure into learning architectures, and \cite{celledoni2023dynamical} for architectures obtained by applying a structure-preserving discretization to a parametrized vector field, so that the property imposed on the field is inherited by the network. A prominent example is {\bfi Hamiltonian neural networks (HNNs)}, whose central idea is to model the Hamiltonian function itself and generate the dynamics through \eqref{eq:intro-hamiltonian}, so that the resulting vector field is Hamiltonian by construction \cite{bertalan2019learning,greydanus2019hamiltonian,toth2019hamiltonian}; Dierkes et al. \cite{dierkes2023symmetry} augmented this construction with a Lie algebra framework that learns symmetries of the system jointly with the energy.
Analogous approaches based on the Lagrangian formulation have also been developed \cite{cranmer2020lagrangian,chen2023data}.
Celledoni et al. \cite{celledoni2023constrained} studied approximations of Hamiltonians of holonomically constrained mechanical systems and emphasized the role of the integrator used in the training loop. In the kernel setting, Hu et al. \cite{hu2024hamiltonian} developed a reproducing kernel Hilbert space (RKHS) formulation based on differential reproducing properties, combining Hamiltonian structure preservation with quantitative finite-sample guarantees. This framework was further extended to Poisson systems and infinite-dimensional geometric settings in \cite{hu2024poisson,hu2025kernel}.
These developments primarily concern how the learned model encodes geometric structures.

For discrete trajectory data, the available information consists of system states separated in time and functionally related by the flow of the Hamiltonian data-generating system. In that case, structure preservation therefore concerns not only the representation of the Hamiltonian but also the relation between consecutive flow observations. One approach is to learn this discrete relation directly.
For example, Jin et al. \cite{jin2020sympnets} introduced SympNets whose architecture produces symplectic maps by construction, while Chen and Tao \cite{chen2021data} used generating functions to represent and learn symplectic maps from trajectory data; related Gaussian-process approaches to symplectic map learning were developed in \cite{rath2020symplectic}.
In the variational formulation, the corresponding discrete object is a {\bfi discrete Lagrangian}. Ober-Bl\"obaum and Offen \cite{oberblobaum2023variational} learned discrete Lagrangians from position observations alone, without observing or approximating velocities or momenta; \cite{lishkova2023discrete} combined this with the identification of symmetry groups and the associated conserved quantities, and \cite{hansen2025forced} extended the discrete variational formulation to forced systems through the Lagrange--d'Alembert principle.
An alternative is to retain the Hamiltonian as the object to be learned and incorporate a symplectic numerical integrator directly into the learning relation between consecutive observations \cite{chen2020symplectic,xiong2021nonseparable,dipietro2020sparse}.
This leads to a fundamental distinction between {\bfi structure preservation} and {\bfi target preservation}. Although a symplectic discretization preserves the symplectic structure of the dynamics, the Hamiltonian identified at a finite time step need not coincide with the physical Hamiltonian generating the observations.

Indeed, the observed trajectory pairs are generated by the exact Hamiltonian flow, whereas the learning relation is imposed through a symplectic numerical integrator. This mismatch gives rise to a scheme-dependent learning target.
The framework of {\bfi backward error analysis} explains this phenomenon. Classical backward error analysis associates a numerical integrator with a nearby modified differential equation; for symplectic integrators, the modified dynamics is again Hamiltonian to arbitrary formal order \cite{benettin1994hamiltonian,HLW06}. 
The learning problem reverses this viewpoint: given observations generated by the exact flow, one seeks a Hamiltonian whose numerical one-step map reproduces the observed evolution.

This inverse viewpoint has recently emerged in data-driven dynamical systems. Zhu et al. \cite{zhu2021imde,zhu2020deep,zhu2022lmm} introduced and further developed {\bfi inverse modified differential equations} to characterize the dynamics implicitly identified when trajectory data are fitted through a prescribed numerical integrator. 
For symplectic discretizations, Offen and Ober-Bl\"obaum \cite{offen2021symplectic} introduced the corresponding notion of an inverse modified Hamiltonian adapted to the geometric integrator, learned it directly from trajectory observations in a Gaussian-process framework, and showed that predictions obtained by applying the same integrator to the learned target compensate the discretization error.
The variational counterpart was developed in \cite{oberblobaum2023variational}, where an inverse version of variational backward error analysis is used to relate the learned discrete Lagrangian to its continuous analogue; the corresponding analysis for variational discretizations of PDEs is carried out in \cite{mclachlan2022backward}, and the field-theoretic learning problem in \cite{offen2024pde}.
David and M\'ehats \cite{david2023symplectic} analyzed the scheme-dependent Hamiltonians associated with symplectic Euler and the midpoint rule and derived higher-order corrections. These works show that the object identified from discrete trajectory data is naturally tied to the numerical scheme used in the learning relation; see also \cite{lu2025numerical} for related observations on numerical integration artifacts in data-driven dynamical-system learning. 
A closely related line of work focuses on the direct approximation of discrete symplectic maps. Herkert et al. \cite{herkert2026symplectic,herkert2026parameter} used generalized kernel interpolation to learn generating functions underlying symplectic-Euler-type discrete updates, with the goal of constructing accurate structure-preserving approximations of the discrete flow.
Related convergence results have also been obtained in the variational setting. Offen \cite{offen2026variational} developed a Gaussian-process framework for identifying continuous and discrete Lagrangians and established convergence as the observation points become dense, together with uncertainty quantification for the learned variational model. Taken together, these works clarify how numerical discretization enters the identification of structure-preserving dynamical models and provide complementary analyses of the resulting approximation and discretization effects.

\paragraph{Summary of the main results.}
The object we wish to recover is the physical Hamiltonian $H$ generating \eqref{eq:intro-hamiltonian}, and the information available about it consists of finitely many states sampled along exact trajectories at a fixed time step. Accordingly, we consider $N$ independent trajectories of the Hamiltonian system \eqref{eq:intro-hamiltonian}, each observed at $L+1$ consecutive time instants with a common sampling interval $\delta>0$. More precisely, the initial states $z_0^{(1)},\ldots,z_0^{(N)}$ are i.i.d. sampled from a probability measure $\mu_0$, and the observed trajectory states are
$$
z_\ell^{(n)}
=
\varphi_{\ell\delta}^{H}(z_0^{(n)}),
\qquad
n=1,\ldots,N,\quad
\ell=0,\ldots,L,
$$
where $\varphi_t^H$ denotes the exact Hamiltonian flow generated by $H$. No derivative of the flow is observed, and $\delta$ is not assumed small relative to the sample size.

The data are related to $H$ through a symplectic one-step integration method. Throughout the paper, $\Int\in\{\mathrm{MP},\mathrm{SE}\}$ denotes the selected symplectic integrator, with $\mathrm{MP}$ and $\mathrm{SE}$ standing for the {\bfi implicit midpoint rule} and the {\bfi $p$-implicit symplectic Euler method}, respectively \cite{HLW06,sanzserna1994numerical,leimkuhler2004simulating}. A superscript $\Int$ indicates dependence on the selected integrator. For the selected integrator $\Int$, each consecutive pair $(z_{\ell-1}^{(n)},z_\ell^{(n)})$ determines a {\bfi scheme-adapted anchor}
$$
x_\ell^{\Int,(n)}:=
\begin{cases}
\dfrac{1}{2}\bigl(z_{\ell-1}^{(n)}+z_\ell^{(n)}\bigr),
& \Int=\mathrm{MP},\\[2mm]
\bigl(q_{\ell-1}^{(n)},p_\ell^{(n)}\bigr),
& \Int=\mathrm{SE},
\end{cases}
$$
and a {\bfi discrete velocity}
$$
v_{\delta,\ell}^{(n)}:=\frac{z_\ell^{(n)}-z_{\ell-1}^{(n)}}{\delta}.
$$
The inverse modified Hamiltonian analysis of Section~\ref{sec:hamiltonian} shows that, for sufficiently small $\delta$, these quantities satisfy the exact identity
\begin{align}\label{eq:scheme-adapted anchor}
v_{\delta,\ell}^{(n)}=J\nabla\check H_\delta^{\Int} \bigl(x_\ell^{\Int,(n)}\bigr),
\qquad n=1,\ldots,N,\quad \ell=1,\ldots,L,
\end{align}
where $\check H_\delta^{\Int}$ is the exact {\bfi inverse modified Hamiltonian} associated with the physical Hamiltonian $H$.
Hence, at a fixed observation step $\delta$, the trajectory data provide exact derivative observations of the scheme-dependent target $\check H_\delta^{\Int}$ rather than of the physical Hamiltonian $H$.

This identity turns the recovery of $H$ into a regression problem for a function whose gradient (or a linear function thereof) is observed. Let $\mathcal H_K$ be the {\bfi reproducing kernel Hilbert space (RKHS)} associated to a fixed chosen kernel map $K: \mathbb{R}^{2d} \times \mathbb{R}^{2d} \rightarrow \mathbb{R}$. Using the scheme-adapted observations \eqref{eq:scheme-adapted anchor}, we reconstruct the Hamiltonian using the {\bfi regularized least-squares estimator}
$$
\widehat h_{\lambda,N,L}^{\Int}:= \mathop{\arg\min}_{h\in\mathcal H_K}\left\{\frac{\delta^2}{NL}\sum_{n=1}^{N}\sum_{\ell=1}^{L}\left\|J\nabla h(x_\ell^{\Int,(n)})-v_{\delta,\ell}^{(n)}\right\|^2+\lambda\|h\|_{\mathcal H_K}^2\right\},
$$
where $\lambda>0$ is a {\bfi Tikhonov regularization parameter}. The recovery of the physical Hamiltonian therefore involves the two levels
$$
\widehat h_{\lambda,N,L}^{\Int} \longrightarrow \check H_\delta^{\Int} \longrightarrow H.
$$
The first involves finite-sample estimation and regularization in the RKHS, whereas the second is governed by the integrator-induced discrepancy between the scheme-dependent target $\check H_\delta^{\Int}$ and the physical Hamiltonian $H$. We quantify these contributions in a common RKHS norm, establish deterministic stability under perturbed trajectory observations, and develop a {\bfi multistride extrapolation} procedure that systematically reduces the integrator-induced target bias.
The main results are summarized as follows.

\begin{enumerate}[leftmargin=*]

\item
\textbf{Quantitative RKHS control of scheme-dependent targets.}
Exact local inverse modified Hamiltonians for the symplectic Euler and implicit midpoint methods have been established in earlier works \cite{david2023symplectic,offen2021symplectic}. Building on this viewpoint, in Section~\ref{sec:hamiltonian} we directly invert the exact-flow midpoint and mixed-endpoint maps to obtain jointly analytic inverse Hamiltonians $\check H_\delta^{\Int}$ on fixed neighborhoods. 
Under certain assumptions (see Assumption~\ref{ass:domain} later on), for sufficiently small $\delta$, this analytic construction yields quantitative control in the RKHS norm
$$
\|\check H_\delta^{\Int}-H\|_{\mathcal H_K} \leq C_{\Int}\delta^{r_{\Int}},
$$
where $r_{\mathrm{MP}}=2$ and $r_{\mathrm{SE}}=1$.
The scheme dependence is reflected in the small-step expansions: midpoint symmetry yields an even analytic family in $\delta$, whereas symplectic Euler generally has a nonvanishing first-order correction. These estimates provide the deterministic target-bias estimate needed for the finite-sample analysis below.

\item
\textbf{Finite-sample recovery and physical consistency.}
In Section~\ref{sec:learning}, we show that the regression problem introduced
above arises naturally from matching the observed time-$\delta$ flow with
the one-step map of the selected symplectic integrator. For the midpoint and
symplectic Euler methods, the resulting one-step discrepancy reduces to the
scheme-adapted regression loss above. Building on the differential RKHS framework of \cite{hu2024hamiltonian,hu2024poisson}, we characterize the corresponding population and empirical problems through the {\bfi input distribution} $\rho^\Int$ and {\bfi covariance operator} $B^\Int$ defined in \eqref{eq:rho-generic} and \eqref{population covariance operator}.

In Section~\ref{sec:main-results}, we establish finite-sample recovery bounds
from $N$ independent trajectories, allowing dependence among observations
within each trajectory; related random-trajectory kernel analyses appear in
\cite{lu2019nonparametric,feng2024learning}. Under the {\bfi source condition}
\eqref{eq:source-condition}, which requires $\check H_\delta^\Int=(B^\Int)^\gamma\psi_\delta^\Int$
for some $\gamma\in(0,1]$ and $\psi_\delta^\Int\in\mathcal H_K$, and with
$\eta=\lambda/\delta^2$, we obtain, on the {\bfi identifiable subspace} and up to
source norms and logarithmic factors,
$$
\|
\widehat h_{\lambda,N,L}^{\Int}-H
\|_{\mathcal H_K}
\lesssim
\eta^{-1}N^{-1/2}
+
\eta^\gamma
+
\delta^{r_{\Int}}.
$$
The three terms quantify, respectively, estimation error, approximation
error, and the integrator-induced discrepancy between $\check H_\delta^\Int$
and $H$. The two exponents have distinct origins and are not comparable: $\gamma$
measures the regularity of the fixed-step target relative to the spectrum of
$B^\Int$, and so governs how fast the approximation error vanishes, whereas
$r_\Int\in\{1,2\}$ is the order of the integrator-induced error fixed by the
scheme, independently of the data. In particular, at a fixed $\delta$,
increasing the number of trajectories can recover the scheme-dependent
target but cannot eliminate its discrepancy from the physical Hamiltonian.

Under uniform control of the source norms, the choices
$$
\eta_N\asymp N^{-1/(2\gamma+2)},
\qquad
\delta_N\asymp
N^{-\gamma/[r_\Int(2\gamma+2)]},
\qquad
\lambda_N
\asymp
N^{-\frac{r_\Int+2\gamma}{r_\Int(2\gamma+2)}}
$$
balance these contributions and yield
$$
\|
\widehat h_{\lambda_N,N,L}^{\Int}-H
\|_{\mathcal H_K}
=
O_{\mathbb P}
\left(
N^{-\gamma/(2\gamma+2)}
\right).
$$

\item
\textbf{Pathwise stability under additive observation noise.}
In Section~\ref{sec:deterministic-perturbations}, we study the stability of the regularized estimator under additive observation noise. The observation noise affects the regression problem through perturbations of both the scheme-adapted anchors and the discrete velocities. We prove
\[
\|
\widehat h_{\lambda,NL}^{U,W}
-
\widehat h_{\lambda,N,L}^{\Int}
\|_{\HK}
\leq
\frac{r_W}{2\sqrt{\eta}}
+
\frac{3\kappa_1\kappa_2}{\eta}
\|\check H_\delta^{\Int}\|_{\HK}\,r_U,
\]
where $\widehat h_{\lambda,NL}^{U,W}$ is the estimator constructed from the
noisy trajectory observations, and $r_U$ and $r_W$ are the realized anchor-
and velocity-noise levels defined in \eqref{eq:pathwise-noise-radii}.
This bound shows that velocity noise is amplified at the scale
$\eta^{-1/2}$, whereas anchor noise enters at the scale $\eta^{-1}$.
Combined with the clean-data estimate above, this yields a reconstruction
bound for the physical Hamiltonian separating estimation error,
approximation error, integrator-induced error, and observation noise.

\item
\textbf{Multistride extrapolation and higher-order physical recovery.}
In Section~\ref{sec:richardson}, we exploit the step-size expansion of the
inverse modified Hamiltonians to reduce the integrator-induced target bias
at the estimator level. From the same uniformly sampled trajectories, we
construct estimators $\widehat h_{s,\delta}^{\Int}$ using observation pairs
separated by $s\delta$, $s=1,\ldots,Q$, and form
$$
\widehat h_\delta^{\Int,(Q)}
=
\sum_{s=1}^Q
\alpha_s^\Int\widehat h_{s,\delta}^{\Int},
$$
where the coefficients $\alpha_s^\Int$ cancel the first $Q-1$ terms in the
step-size expansion. The corresponding extrapolated target satisfies
$$
\|
\check H_\delta^{\Int,(Q)}-H
\|_{\mathcal H_K}
\leq
C_{Q,\Int}\delta^{r_\Int Q},
$$
raising the deterministic target accuracy from
$O(\delta^{r_\Int})$ to $O(\delta^{r_\Int Q})$.

We further quantify the statistical amplification introduced by the
extrapolation. Although the stride estimators are constructed from the same
trajectories and are therefore dependent, no independence across strides is
required. Under uniform source conditions, the extrapolated estimator
retains the statistical recovery rate above, while the step-size requirement
is relaxed from
$$
\delta_N
\asymp
N^{-\gamma/[r_\Int(2\gamma+2)]}
\qquad\text{to}\qquad
\delta_N
\asymp
N^{-\gamma/[r_\Int Q(2\gamma+2)]}.
$$
Thus the same physical recovery accuracy can be achieved with a larger
observation step.

\item
\textbf{Numerical validation and trajectory effects.}
In Section~\ref{sec:experiments}, we use the H\'enon--Heiles system to
separate and validate the three components of the reconstruction error.
The midpoint and symplectic Euler estimators exhibit the predicted
second- and first-order integrator-induced errors, respectively, and the
learned discrepancies reproduce the leading scheme-dependent
inverse-target corrections. We further show that multistride Richardson
extrapolation reduces the integrator-induced error at the predicted
orders until the estimation and approximation errors become dominant.

We also examine robustness under noisy trajectory-state observations and
the allocation of a fixed observation budget between independent
trajectories and repeated observations along each trajectory. The latter
experiments illustrate the distinct roles of $N$ and $L$: independent
trajectories provide statistical averaging and broader sampling of the
initial distribution, whereas observations within a trajectory remain
dynamically dependent. Finally, reconstructions of a double pendulum,
a Frenkel--Kontorova system, and a strongly non-convex Hamiltonian
demonstrate the applicability of the same trajectory-based midpoint
estimator beyond the benchmark problem.
\end{enumerate}

The remainder of the paper is organized as follows.
Section~\ref{sec:hamiltonian} develops the inverse modified Hamiltonian
analysis for the implicit midpoint and $p$-implicit symplectic Euler
methods and establishes the corresponding small-step expansions and RKHS
bias estimates.
Section~\ref{sec:learning} formulates the trajectory-based kernel
regression problem, introduces the associated population and empirical
operators, and derives the representer formulation.
Section~\ref{sec:main-results} establishes the finite-sample recovery and
physical-consistency bounds and also studies deterministic stability under
perturbations of the observed trajectories.
Section~\ref{sec:richardson} develops the multistride extrapolation
procedure and the resulting higher-order recovery estimates.
The numerical experiments in Section~\ref{sec:experiments} validate the
scheme-dependent integrator-induced errors, the finite-sample behavior,
the multistride correction and its noise sensitivity, and the distinct
roles of independent trajectories and within-trajectory observations.
They also demonstrate the reconstruction on several additional
Hamiltonian systems. The code and numerical experiments accompanying this paper are
publicly available at
\url{https://github.com/jianyuhu/Learning-Hamiltonians-from-Trajectory-Data-with-Kernel-Methods}.

\section{Hamiltonian systems, symplectic integrators, and inverse modified Hamiltonians}
\label{sec:hamiltonian}

The main goal of this section is to introduce the two symplectic one-step methods used throughout the paper and the inverse modified Hamiltonians associated with them. The key fact established in this section is that, for sufficiently small step size, the exact time-$\delta$ flow can be represented exactly as one step of the selected integrator applied to an integrator-dependent Hamiltonian. 

We begin by fixing the Hamiltonian setup; see \cite{Abraham1978,Marsden1994} for the general theory. We work on the canonical phase space $\RR^{2d}$ with coordinates $z=(q,p)$, $q,p\in\RR^d$, and canonical symplectic matrix
\[
J=\begin{pmatrix}0&I_d\\-I_d&0\end{pmatrix},
\qquad J^\top=-J,\qquad J^2=-I_{2d}.
\]
Tangent vectors are identified with elements of $\RR^{2d}$ through the canonical isomorphism $T_z\RR^{2d}\cong\RR^{2d}$, and are split as $u=(u_q,u_p)$ conformally with $z=(q,p)$, so that $u_q,u_p\in\RR^d$ are the position and momentum components of $u$ and $u_q^i,u_p^i$ denote their $i$-th entries, $i=1,\ldots,d$; the same convention applies to $v \in \mathbb{R} ^{2d}$. The associated {\bfi symplectic form} is the constant differential two-form
\[
\omega(u,v):=u^\top Jv=u_q^\top v_p-u_p^\top v_q=\sum_{i=1}^{d}\bigl(u_{q}^{i}v_{p}^{i}-u_{p}^{i}v_{q}^{i}\bigr),
\qquad u,v\in\RR^{2d},
\]
that is, $\omega=\sum_{i=1}^{d}\mathrm dq^i\wedge\mathrm dp^i$.
It is skew-symmetric, closed, and nondegenerate, so $(\RR^{2d},\omega)$ is a symplectic vector space and $(q,p)$ are {\bfi Darboux coordinates} for it. By the Darboux theorem, every symplectic manifold of dimension $2d$ admits such coordinates in a neighborhood of each point, in which the symplectic form takes exactly the expression above \cite[Thm.~3.2.2]{Abraham1978}, \cite[Thm.~5.1.1]{Marsden1994}. The analysis below is therefore carried out in the canonical model without loss of generality at the local level at which we work.

Nondegeneracy makes $v\mapsto\omega(v,\cdot)$ an isomorphism of $\RR^{2d}$ onto its dual, so every sufficiently smooth {\bfi Hamiltonian function} $H:\RR^{2d}\to\RR$ determines a unique {\bfi Hamiltonian vector field} $X_H$ by
\begin{equation}\label{eq:hamiltonian-vector-field}
\omega\bigl(X_H(z),v\bigr)=\mathrm dH(z)v
\quad\text{for all }v\in\RR^{2d},
\qquad\text{equivalently}\qquad
X_H(z)=J\nabla H(z),
\end{equation}
the two forms being equivalent because $J^\top J=I_{2d}$. The dynamics are governed by
\begin{equation*}
\dot z(t)=J\nabla H(z(t)), \qquad z(0)=z_0,
\end{equation*}
and we denote the corresponding local flow by $\varphi_t^H$. Writing $\{F,G\}:=\nabla F^\top J\nabla G=\omega(X_F,X_G)$ for the canonical Poisson bracket, any observable evolves along the flow according to $\tfrac{\mathrm d}{\mathrm dt}\bigl(F\circ\varphi_t^H\bigr)=\{F,H\}\circ\varphi_t^H$. Taking $F=H$ and using the skew-symmetry of $J$ gives
\begin{equation}\label{eq:energy-conservation}
\frac{\mathrm d}{\mathrm dt}H\bigl(\varphi_t^H(z_0)\bigr)=\nabla H^\top J\nabla H=0 ,
\end{equation}
so the Hamiltonian is a first integral: energy is conserved exactly along the exact flow. Note that $H$ enters \eqref{eq:hamiltonian-vector-field} only through its differential, so $H$ and $H+c$ generate the same dynamics for every $c\in\RR$; a Hamiltonian reconstructed from its vector field is therefore determined only up to an additive constant.

A differentiable map $\Phi:U\to\RR^{2d}$, defined on an open set $U\subseteq\RR^{2d}$, is called {\bfi symplectic} if $\Phi^*\omega=\omega$, where $\Phi^*$ denotes the pullback of differential forms by $\Phi$. Equivalently,
\begin{equation}\label{eq:symplecticity}
(D\Phi(z))^\top J\,D\Phi(z)=J \qquad\text{for all } z\in U.
\end{equation}
For a vector field $X$ on $U$, write $\iota_X$ for the {\bfi interior product} with $X$, so that $\iota_X\omega$ is the one-form $v\mapsto\omega(X,v)$, and $\mathcal L_X$ for the {\bfi Lie derivative} along $X$, defined on a differential form $\beta$ by $\mathcal L_X\beta:=\tfrac{\mathrm d}{\mathrm dt}\big|_{t=0}(\varphi_t^X)^*\beta$, with $\varphi_t^X$ the flow of $X$. In this notation, \eqref{eq:hamiltonian-vector-field} reads $\iota_{X_H}\omega=\mathrm dH$. Since $\omega$ is closed, Cartan's magic formula $\mathcal L_X=\iota_X\circ\mathrm d+\mathrm d\circ\iota_X$ gives
\[
\mathcal L_{X_H}\omega
=\iota_{X_H}\mathrm d\omega+\mathrm d\,\iota_{X_H}\omega
=\mathrm d(\mathrm dH)=0 ,
\]
so $(\varphi_t^H)^*\omega$ is independent of $t$ and hence equal to its value $\omega$ at $t=0$: the flow $\varphi_t^H$ is symplectic for every $t$ for which it is defined. Hamiltonian flows thus preserve both the Hamiltonian, by \eqref{eq:energy-conservation}, and the canonical symplectic form; taking $d$-th exterior powers in $\Phi^*\omega=\omega$ shows that any symplectic map also preserves the Liouville volume $\omega^d/d!$.

Condition \eqref{eq:symplecticity} admits a converse, which is the structural mechanism behind the construction of the inverse modified Hamiltonians below. Since $\iota_X\omega$ is the one-form $v\mapsto X^\top Jv$, its coefficient vector is $-JX$, and closedness $\mathrm d(\iota_X\omega)=0$ amounts to the symmetry of $D(-JX)$, that is,
\[
J\,DX+(DX)^\top J=0 .
\]
If in addition $U$ is simply connected, the Poincar\'e lemma yields $h\in C^\infty(U)$, unique up to an additive constant, with $\iota_X\omega=\mathrm dh$, that is, $X=J\nabla h$. Symplecticity of a one-step map is therefore exactly the structural information needed to recognize its associated discrete velocity field as the Hamiltonian vector field of some function, which is how the inverse modified Hamiltonians are obtained below; see also \cite{HLW06,sanzserna1994numerical,leimkuhler2004simulating}.

We assume throughout that the trajectories considered below exist for all $|t|\le T$ and remain in a fixed compact set. Whenever an implicit one-step map is used, we work on the locally unique branch continuous at $\delta=0$ and on a domain where all relevant maps are well defined.

\subsection{Symplectic one-step methods}

We now specify the two symplectic integrators used throughout the paper. 

\begin{definition}[Symplectic integrator]
A {\bfi one-step method} $I_\delta$ for Hamiltonian systems assigns to each sufficiently smooth Hamiltonian $h:\RR^{2d}\to\RR$ and each step size $\delta$ a map $I_\delta^h:U_\delta^h\subseteq \RR^{2d}\to\RR^{2d}$, defined on an open set $U_\delta^h$, intended to approximate the time-$\delta$ flow $\varphi_\delta^h$ of the Hamiltonian vector field $X_h=J\nabla h$. It is a {\bfi symplectic integrator of order $r\ge1$} if, for every such $h$, the one-step map $I_\delta^h$ is symplectic and, on every compact set $\mathcal C$ on which both $I_\delta^h$ and $\varphi_\delta^h$ are defined,
\begin{equation*}
\sup_{z\in\mathcal C}
\bigl\|I_\delta^h(z)-\varphi_\delta^h(z)\bigr\|
\le C_{h,\mathcal C}|\delta|^{r+1}
\end{equation*}
for all sufficiently small $|\delta|$, where $C_{h,\mathcal C}$ is independent of $\delta$. For an implicit method, $I_\delta^h$ denotes the locally
unique branch continuous at $\delta=0$.
\end{definition} 

Two symplectic integrators are used throughout this work: the implicit midpoint rule and the $p$-implicit symplectic Euler rule. Both are classical, and we recall them only to fix the notation and the two structural properties that determine the behavior of their inverse modified Hamiltonians: the time symmetry of the former and its absence for the latter. For their derivation, symplecticity, and order, we refer to the standard monographs \cite{HLW06,sanzserna1994numerical,leimkuhler2004simulating}; the symplecticity of the symplectic Euler variants goes back to \cite{devogelaere1956}, and that of the midpoint rule, as the one-stage Gauss collocation method, to \cite{lasagni1988canonical,sanzserna1988rk,suris1988conservation}.

\paragraph{The implicit midpoint rule.}

For a Hamiltonian $h:\RR^{2d}\to\RR$, the implicit midpoint rule is defined on the locally unique branch continuous at $\delta=0$ by
\begin{equation*}
I_\delta^h(z) = z' \quad\text{such that}\quad z' = z + \delta\, J\nabla h\!\left(\frac{z+z'}{2}\right).
\end{equation*}
Writing $z_{\mathrm{mid}}:=(z+z')/2$ and $v_\delta:=(z'-z)/\delta$, the rule becomes
\begin{equation*}
v_\delta = J\nabla h(z_{\mathrm{mid}}).
\end{equation*}

The midpoint rule is a second-order symplectic Runge--Kutta method and is symmetric in time,
$I_{-\delta}^h\circ I_\delta^h = \mathrm{Id}$; see \cite{HLW06}. Consequently, its normalized inverse modified Hamiltonian expansions contain only even powers of \(\delta\).

\paragraph{The \texorpdfstring{$p$}{p}-implicit symplectic Euler rule.}

Throughout the main
analysis we use the $p$-implicit variant; its $q$-implicit adjoint is invoked only for the
sign-reversal and averaging observations below. For
\(h:\mathbb R^{2d}\to\mathbb R\), the \(p\)-implicit rule maps \(z=(q,p)\) to
\(z'=(Q,P)\) through
\begin{equation}\label{eq:symplectic-euler}
\begin{aligned}
 P&=p-\delta\,\partial_qh(q,P),\\
 Q&=q+\delta\,\partial_ph(q,P).
\end{aligned}
\end{equation}
We write \(I_\delta^h(q,p)=(Q,P)\) for its locally unique branch continuous at
\(\delta=0\).  Introducing the mixed point and discrete velocity $x_{\mathrm{SE}}:=(q,P)$ and $v_\delta:=(Q-q,P-p)/\delta$,  equation~\eqref{eq:symplectic-euler} becomes
\begin{equation*}
 v_\delta=J\nabla h(x_{\mathrm{SE}}).
\end{equation*}

The $p$-implicit symplectic Euler rule is a first-order symplectic
integrator, but is not symmetric in time \cite{HLW06}.  Its adjoint is the \(q\)-implicit variant
\[
 Q=q+\delta\,\partial_ph(Q,p),
 \qquad
 P=p-\delta\,\partial_qh(Q,p),
\]
whose observable mixed input is \((Q,p)\).  The two variants have opposite first-order corrections in their inverse modified Hamiltonians.

When we consider the implicit midpoint and symplectic Euler methods simultaneously, we use the index $\Int\in\{\MP,\SE\}$ and write $I_\delta^{h,\Int}$ for the corresponding one-step map. 
Although both methods are implicit when used for time stepping, their evaluation points and discrete velocities are determined directly once the two endpoints of a step are available. Thus, {\it no implicit equation needs to be solved to form these quantities from a given endpoint pair}. This leads to the following common endpoint-based representation.

\begin{definition}[Scheme-adapted  anchor]
\label{def:regression-anchor}
Let \(z^-=(q,p)\) and \(z^+=(Q,P)\) be an endpoint pair separated by a step \(\delta\). The associated {\bfi discrete velocity} is $v_\delta(z^-,z^+):=(z^+-z^-)/\delta$.
For the selected integrator \(\Int\in\{\MP,\SE\}\), its
{\bfi scheme-adapted  anchor} is
\begin{equation*}
 x^{\Int}(z^-,z^+):=
 \begin{cases}
 \frac12(z^-+z^+),&\Int=\MP,\\[2pt]
 (q,P),&\Int=\SE.
 \end{cases}
\end{equation*}
The anchor is \emph{observable} in the sense that it is computed entirely from the two endpoints;
it need not itself be a state on the observed trajectory. It is \emph{scheme-adapted} because it
is precisely the point at which the selected integrator evaluates the  Hamiltonian
vector field.  Consequently, whenever $z^+=I_\delta^{h,{\Int}}(z^-)$, the endpoints satisfy
\begin{equation*}
v_\delta(z^-,z^+)=J\nabla h\!\left(x^{\Int}(z^-,z^+)\right).
\end{equation*}
\end{definition}
\subsection{Inverse modified Hamiltonians}

Backward error analysis relates a numerical one-step map to a nearby continuous-time system through formal power series in the step size; see \cite{HLW06,benettin1994hamiltonian}. In the present inverse problem, the relevant relation goes in the opposite direction: given the exact time-$\delta$ flow $\varphi_\delta^H$, we seek a Hamiltonian whose numerical one-step map reproduces this flow. This motivates the following notion.

\begin{definition}[Formal inverse modified Hamiltonian]\label{def:inv-modified-ham}
Let $H\in C^\infty(\RR^{2d})$, and let $I_\delta$ be a consistent, sufficiently regular symplectic method admitting a formal modified equation. The {\bfi inverse modified Hamiltonian} is a formal power series
\[
\hypertarget{sym-checkH-formal}{\check{H}^{\mathrm{form}}}(z;\delta)
= H(z) + \delta\check{H}_1(z) + \delta^2\check{H}_2(z) + \delta^3\check{H}_3(z) + \cdots
\]
such that, again coefficientwise,
\begin{equation*}
I_\delta^{\check{H}^{\mathrm{form}}}(z)
\mathrel{\overset{\mathrm{formal}}{=}}\varphi^H_\delta(z).
\end{equation*}
The series is determined only up to a formal $\delta$-dependent
additive constant.
\end{definition}

Definition~\ref{def:inv-modified-ham} does not assert convergence of
the formal series or exact equality at a fixed nonzero step size.
For the statistical problem studied below, however, we require a
genuine fixed-step target satisfying the inverse relation exactly.
We therefore impose the following analytic setting.


\begin{assumption}[Analytic setting]
\label{ass:analytic-setting}
Let $\mathcal K\subset\RR^{2d}$ be a compact convex set with nonempty interior. Assume that, for some $r_0>0$, the Hamiltonian $H$ extends holomorphically to the complex neighborhood
\begin{equation*}
\mathcal K_{r_0}:=
\left\{
z\in\CC^{2d}:\operatorname{dist}(z,\mathcal K)<r_0
\right\},
\end{equation*}
and that this extension is real-valued on real arguments.
\end{assumption}

Under Assumption~\ref{ass:analytic-setting}, the Hamiltonian vector field $f:=J\nabla H$ is holomorphic on $\mathcal K_{r_0}$ and bounded on every smaller tube
$\mathcal K_r$, $0<r<r_0$. 
Exact local inverse modified Hamiltonians for the schemes considered have been constructed in related settings in \cite{david2023symplectic,offen2021symplectic}. For the present analysis, we use versions defined on a common analytic neighborhood and with the quantitative coefficient bounds required later in the RKHS estimates. The corresponding constructions are given in Appendix~\ref{sec:bae}.

\begin{proposition}[Exact inverse modified Hamiltonians]
\label{prop:exact-inverse-targets}
Suppose Assumption~\ref{ass:analytic-setting} holds. For each $\Int\in\{\MP,\SE\}$, there exist $\delta_*^\Int>0$ and a normalized analytic family $\check H_\delta^\Int$, unique up to a $\delta$-dependent additive constant, such that for a fixed $z_\star\in\mathcal K$, $\check H_\delta^\Int(z_\star)=H(z_\star),$
and
\begin{equation*}
I_\delta^{\check H_\delta^\Int,\Int}(z)= \varphi_\delta^H(z), \qquad z\in\mathcal K,\quad |\delta|<\delta_*^\Int.
\end{equation*}

Moreover, the midpoint family is even in $\delta$, $\check H_{-\delta}^{\MP}=\check H_\delta^{\MP},$
and, with $f:=J\nabla H$ and $G(z):=f(z)^\top\Hess H(z)f(z)$,
\begin{align*}
\check H_\delta^{\MP}(z)
&=
H(z)
+
\frac{\delta^2}{24}
\bigl(G(z)-G(z_\star)\bigr)
+
\mathcal O_{C^0(\mathcal K)}(\delta^4),
\\
\check H_\delta^{\SE}(z)
&=
H(z)
+
\frac{\delta}{2}
\Bigl(
\langle H_q(z),H_p(z)\rangle
-
\langle H_q(z_\star),H_p(z_\star)\rangle
\Bigr)
+
\mathcal O_{C^0(\mathcal K)}(\delta^2).
\end{align*}
\end{proposition}

\begin{proof}
The midpoint statements follow from
Proposition~\ref{prop:imh-analytic} and Corollary \ref{cor:imh-native-bound} in Appendix~\ref{sec:bae}, while the symplectic Euler
statements follow from Proposition~\ref{prop:imh-analytic-se} and Corollary \ref{cor:imh-native-bound-se}.
\end{proof}

Combining Proposition~\ref{prop:exact-inverse-targets} with Definition~\ref{def:regression-anchor}, every exact-flow pair $z^+=\varphi_\delta^H(z^-)$ satisfies
\begin{equation}\label{eq:exact-regression-target}
v_\delta(z^-,z^+)=J\nabla\check H_\delta^\Int \bigl(x^\Int(z^-,z^+)\bigr).
\end{equation}
Thus, although the endpoint pair is generated by the exact flow of $H$, it satisfies the discrete relation of the selected integrator exactly with $\check H_\delta^\Int$ in place of $H$. In particular, the fixed-step inverse target is integrator-dependent and differs from the physical Hamiltonian by
\[
\check H_\delta^{\MP}-H=\mathcal O(\delta^2),
\qquad
\check H_\delta^{\SE}-H=\mathcal O(\delta),
\]
up to additive constants.
Since only the Hamiltonian vector field enters \eqref{eq:exact-regression-target}, $\check H_\delta^\Int$ is identifiable only up to an additive constant. The normalization at $z_\star$ selects a representative for comparison with $H$; the corresponding identifiability issue in the RKHS reconstruction is addressed in Section~\ref{sec:learning}.

The midpoint coefficient follows from the classical forward coefficient
\cite{HLW06,offen2021symplectic} by reversing the formal identity, while the Euler coefficient is derived in Proposition~\ref{prop:imh-analytic-se}. The adjoint $q$-implicit Euler rule reverses the sign of the linear correction. Appendix~\ref{sec:bae} contains the analytic constructions and the quantitative coefficient and target-bias bounds.

\begin{example}[Harmonic oscillator: comparison of exact targets]
Let
\[
 H(q,p)=\frac12(p^2+\omega^2q^2),
 \qquad \omega>0,\qquad z_\star=(0,0).
\]
On the branches continuous at \(\delta=0\), the two normalized exact targets are
\begin{align*}
 \check H_\delta^{\MP}
 &=\frac{2\tan(\omega\delta/2)}{\omega\delta}\,H,
 &&|\omega\delta|<\pi,\notag\\
 \check H_\delta^{\SE}(q,p)
 &=\frac{\tan(\omega\delta)}{\omega\delta}\,H(q,p)
   +\frac{\sec(\omega\delta)-1}{\delta}\,qp,
 &&|\omega\delta|<\frac{\pi}{2}.
\end{align*}
Their expansions are
\begin{equation*}
 \check H_\delta^{\MP}
 =\left(1+\frac{\omega^2\delta^2}{12}\right)H+\mathcal O(\delta^4), \quad \mbox{and} \quad
 \check H_\delta^{\SE}
 =H+\frac{\omega^2\delta}{2}\,qp
   +\frac{\omega^2\delta^2}{3}H+\mathcal O(\delta^3).
\end{equation*}
Thus midpoint produces a quadratic rescaling of the Hamiltonian, whereas symplectic Euler introduces a first-order cross-term. At $|\omega\delta|=\pi$, the midpoint construction becomes singular, while at $|\omega\delta|=\pi/2$ the corresponding Euler mixed-point construction becomes singular. These failures illustrate the local small-step character of the exact inverse modified Hamiltonians.
\end{example}

\section{RKHS reconstruction from random trajectories}
\label{sec:learning}

We now formulate the inverse problem studied in this paper. The goal is to infer an unknown Hamiltonian function $H:\RR^{2d}\to\RR$ from trajectory data of the Hamiltonian system \eqref{eq:intro-hamiltonian} sampled at a fixed time step $\delta>0$.
In the analysis, we denote the time-$\delta$ data-generating map by the exact flow $\varphi_\delta^H$, but we do not assume that the estimator knows or can evaluate this map. From the sampled trajectory states alone, we seek to construct an estimator of the dynamical system \eqref{eq:intro-hamiltonian} that preserves its Hamiltonian structure and is compatible with the symplectic discretization used in the model.

The reconstruction is built from one of the two symplectic one-step methods introduced in Section~\ref{sec:hamiltonian}. Given consecutive exact trajectory states, the corresponding scheme-adapted anchor and discrete velocity can be formed directly from the endpoints. By Proposition~\ref{prop:exact-inverse-targets} and \eqref{eq:exact-regression-target}, these quantities satisfy the selected one-step relation exactly with the inverse modified Hamiltonian $\check H_\delta^\Int$. The learning problem below therefore first reconstructs this fixed-step target in an RKHS and subsequently relates it to the physical Hamiltonian $H$ as $\delta\to0$.

\subsection{Random trajectories and scheme-adapted observations}

We first formulate the learning problem in terms of latent exact trajectories.
Let $z_0^{(1)},\ldots,z_0^{(N)} \stackrel{\mathrm{i.i.d.}}{\sim}\mu_0,$ where $\mu_0$ is a probability measure on phase space $\RR^{2d}$, and define
\begin{equation*}
z_l^{(n)}:=\varphi_\delta^H\bigl(z_{l-1}^{(n)}\bigr),\qquad n=1,\ldots,N,\quad l=1,\ldots,L.
\end{equation*}
Thus, the $N$ trajectories are independent, whereas the states within each trajectory are generally dependent. Later on in Section~\ref{sec:deterministic-perturbations} we shall also allow the recorded trajectory states to be corrupted by additive observational noise of the form
\begin{equation}\label{eq:noisy-observation}
\widetilde z_l^{(n)}=z_l^{(n)}+e_l^{(n)},\qquad n=1,\ldots,N,\quad l=0,\ldots,L,
\end{equation}
where $e_l^{(n)}$ are $\RR^{2d}$-valued random variables.
In this section, we focus on the reconstruction problem for the exact states $z_l^{(n)}$. This preserves the scheme-adapted relation with the inverse modified Hamiltonian and isolates the randomness to that arising exclusively from sampling the initial conditions. Accordingly, the clean trajectory dataset is
\begin{align}\label{eq:data-generation}
\mathcal D_{N,L,\delta}
=
\left\{
z_l^{(n)}
:
n=1,\ldots,N,\quad l=0,\ldots,L
\right\},
\end{align}
and the estimator does not require knowledge or evaluation of the flow map
$\varphi_\delta^H$ itself. Fix $\Int\in\{\MP,\SE\}$. From each consecutive pair $(z_{l-1}^{(n)},z_l^{(n)})$, we form the scheme-adapted anchor
\begin{equation*}
x_l^{\Int,(n)}:=x^{\Int}\bigl(z_{l-1}^{(n)},z_l^{(n)}\bigr),\qquad  v_{\delta,l}^{(n)}:=\frac{z_l^{(n)}-z_{l-1}^{(n)}}{\delta}.
\end{equation*}
Hence both the regression inputs and the discrete velocities are obtained directly from the observed endpoints. In particular, no implicit equation needs to be solved when forming these quantities from the trajectory data.

By Proposition~\ref{prop:exact-inverse-targets} and \eqref{eq:exact-regression-target}, whenever $|\delta|<\delta_*^\Int$, the scheme-adapted observations satisfy
\begin{equation}\label{eq:exact-scheme-adapted-regression}
v_{\delta,l}^{(n)}= J\nabla \HInt_\delta   \bigl(x_l^{\Int,(n)}\bigr), \qquad    n=1,\ldots,N,\quad l=1,\ldots,L.
\end{equation}
Thus, at fixed $\delta$, the exact-flow trajectory pairs provide exact derivative observations of the integrator-dependent inverse target $\HInt_\delta$. In general, they do not provide exact derivative observations of the physical Hamiltonian $H$ itself.

To describe the distribution of the regression inputs, define the anchor maps
\begin{equation}\label{eq:population-anchor-map}
X_l^{\Int}(z_0):= x^{\Int}\left(        \varphi_{(l-1)\delta}^H(z_0), \varphi_{l\delta}^H(z_0)\right), \qquad l=1,\ldots,L.
\end{equation}
The corresponding averaged one-point distribution is
\begin{equation}\label{eq:rho-generic}
\rho^\Int:=\frac{1}{L}\sum_{l=1}^L(X_l^\Int)_\#\mu_0.
\end{equation}
The measure $\rho^{\Int}$ will serve as the population input distribution in Section~\ref{sec:main-results}. Although the $N$ trajectories are independent, the $L$ anchors generated by a single trajectory are generally dependent. Accordingly, the error analysis treats each trajectory, rather than each individual anchor, as one independent observation block.

\subsection{RKHS formulation and regularized problems}
\label{sec:rkhs-formulation}

We now formulate the reconstruction problem in a reproducing kernel Hilbert space (RKHS). Since the scheme-adapted observations in \eqref{eq:exact-scheme-adapted-regression} involve first derivatives of the Hamiltonian, we choose a {\bfi native space} in which derivative evaluation is continuous and into which the analytic Hamiltonian and inverse modified targets can be embedded.

\begin{assumption}[Learning domain and native space]
\label{ass:domain}
Let $\Omega\subset\RR^{2d}$ be a bounded, connected Lipschitz domain such that $\overline{\Omega}\subset\operatorname{int}\mathcal K$, where $\mathcal K$ is the compact analytic set of Assumption~\ref{ass:analytic-setting}.
We assume that $\operatorname{supp}\rho^\Int\subset\Omega,$ for $\Int\in\{\MP,\SE\}.$
Let $K:\overline{\Omega}\times\overline{\Omega}\to\RR$ be a $C^2$ positive-semidefinite kernel whose RKHS $\HK=\HK(\Omega)$ is norm-equivalent to $H^s(\Omega)$ for some $s>d+1$.
\end{assumption}

\begin{remark}[Derivative bounds]\label{rem:derivative-bounds}
Suppose that the Sobolev embedding implied by Assumption~\ref{ass:domain} holds with $s>d+2$. Then $\HK\hookrightarrow C^2(\overline{\Omega}).$
Consequently, there exist constants $\kappa_1,\kappa_2>0$ such that
\begin{equation}\label{eq:gradient-perturbation-bounds}
\|J\nabla h(x)\|\le\kappa_1\|h\|_{\HK},\qquad\|J\nabla h(x)-J\nabla h(y)\|\le\kappa_2 |x-y|\,\|h\|_{\HK},
\end{equation}
for all $h\in\HK$ and all $x,y\in\Omega$ for which the line segment joining $x$ and $y$ is contained in $\Omega$.
The second estimate follows from the uniform bound on the Hessian provided by the Sobolev embedding and the mean value theorem.
\end{remark}

\begin{example}[Mat\'ern kernels]\label{Matern kernels}
A standard class satisfying Assumption~\ref{ass:domain} is given by Mat\'ern kernels
\[
K_\nu(x,y)=\frac{2^{1-\nu}}{\Gamma(\nu)}\left(\frac{\sqrt{2\nu}}{\ell}\|x-y\|\right)^\nu\mathsf K_\nu\left(\frac{\sqrt{2\nu}}{\ell}\|x-y\|\right), \qquad \nu,\ell>0,
\]
where $\mathsf K_\nu$ denotes the modified Bessel function of the second kind and $\ell$ is the length-scale parameter. On a bounded Lipschitz domain $\Omega\subset\RR^{2d}$, the RKHS of the restricted Mat\'ern kernel is norm-equivalent to $\HK(\Omega)\cong H^{\nu+d}(\Omega);$
see \cite[Cor.~10.48]{wendland2005scattered} and \cite[Sec.~2]{fischer2020sobolev}. Hence choosing $\nu>1$ gives $s=\nu+d>d+1$, as required in Assumption~\ref{ass:domain}.
\end{example}

Together, Assumptions~\ref{ass:analytic-setting} and~\ref{ass:domain} ensure that, for all sufficiently small $|\delta|$, both the physical Hamiltonian $H$ and the exact inverse targets $\HInt_\delta$ belong to the hypothesis space $\HK$. Indeed, their analyticity on a neighborhood of $\overline{\Omega}$ implies membership in $H^s(\Omega)$, while $\HK$ is norm-equivalent to $H^s(\Omega)$. 
Moreover, Corollaries~\ref{cor:imh-native-bound} and~\ref{cor:imh-native-bound-se} yield \begin{equation}\label{eq:generic-native-bias}
\|\HInt_\delta-H\|_{\HK}=\mathcal O(|\delta|^{r_\Int}),\quad \mbox{with} \quad r_{\MP}=2,\qquad r_{\SE}=1.
\end{equation}
Since $s>d+1$ and the phase-space dimension is $2d$, the Sobolev embedding $\HK\hookrightarrow C^1(\overline{\Omega})$ guarantees that the derivative observations below are well defined and continuous on $\HK$.

Given the random trajectories  \eqref{eq:data-generation} or equivalently the scheme-adapted observations \eqref{eq:exact-scheme-adapted-regression}, we consider the following empirical regularized minimization problem 
\begin{align}
\widehat h_{\lambda,N,L}^{\Int}&:=\arg\min_{h\in\HK}\widehat R_{\lambda,N,L}^{\Int}(h), \label{eq:emp-pro}\\
\widehat R_{\lambda,N,L}^{\Int}(h)&:= \frac{\delta^2}{NL} \sum_{n=1}^N\sum_{l=1}^L \left\|A^\Int h(x_l^{\Int,(n)})-v_{\delta,l}^{(n)}\right\|^2+\lambda\|h\|_{\HK}^2, \label{eq:emp-fun}    
\end{align}
where $A^\Int h:=J\nabla h$ denotes the Hamiltonian operator and $\lambda>0$ is the Tikhonov regularization parameter. The corresponding population regularized problem is
\begin{align}
h_{\lambda,L}^{*,\Int}&:=\arg\min_{h\in\HK}R_{\lambda,L}^{\Int}(h), \label{eq:exp-pro}\\
R_{\lambda,L}^{\Int}(h)&:=\delta^2\|A^\Int h-A^\Int \HInt_\delta\|_{L^2(\rhoInt)}^2+\lambda\|h\|_{\HK}^2\label{eq:exp-fun}.
\end{align}
The factor $\delta^2$ in \eqref{eq:emp-fun} and \eqref{eq:exp-fun} makes the data-fit term equal to the squared one-step endpoint defect,
\[
\delta^2\|A^\Int h(x_l)-v_{\delta,l}\|^2 = \|z_l-z_{l-1}-\delta A^\Int h(x_l)\|^2,
\]
and leads to the effective regularization parameter $\eta=\lambda/\delta^2$.

\subsection{Operator and kernel representations}
\label{sec:estimator-representations}

We now derive two complementary representations of the regularized estimators of the minimization problems \eqref{eq:emp-pro} and \eqref{eq:exp-pro}.
An operator formulation is convenient for the subsequent statistical error analysis, while the kernel representation reduces the empirical problem to a finite-dimensional linear system suitable for computation. Both rely on the same differential RKHS structure induced by the Hamiltonian observation operator $A^\Int h=J\nabla h$.

\paragraph{Operator framework.}
The basic RKHS operator machinery follows \cite{hu2024hamiltonian}; we adapt it here to the scheme-dependent input distributions and trajectory-based observations.
For each $\Int\in\{\MP,\SE\}$, under Assumption~\ref{ass:domain}, $A^\Int:\HK\longrightarrow L^2(\rho^\Int;\RR^{2d})$ is a bounded linear operator, with adjoint
\begin{equation*}
(A^\Int)^*g=\int_\Omega g(x)^\top J\nabla_1K(x,\cdot)\, \mathrm d\rho^\Int(x), \qquad g\in L^2(\rho^\Int;\RR^{2d}).
\end{equation*}
Accordingly, the {\bfi population covariance operator} is given by 
\begin{equation}
\label{population covariance operator}
B^\Int:=(A^\Int)^*A^\Int:\HK\longrightarrow\HK, \qquad B^\Int h = \int_\Omega \nabla h(x)^\top\nabla_1K(x,\cdot)\, \mathrm d\rho^\Int(x).
\end{equation}
Moreover, $B^\Int$ is positive semidefinite, self-adjoint, and trace class, and hence compact.

For the empirical problem, define
\begin{equation}\label{eq:operator-AN}
A_{NL}^{\Int}:\HK\longrightarrow\RR^{2dNL}, \qquad A_{NL}^{\Int}h := \frac{1}{\sqrt{NL}} \operatorname{Vec}\left(J\nabla h(x_1^{\Int,(1)}),\ldots, J\nabla h(x_L^{\Int,(N)})\right),
\end{equation}
and $B_{NL}^{\Int}:=(A_{NL}^{\Int})^*A_{NL}^{\Int}.$
By the same differential reproducing argument, $A_{NL}^{\Int}$ is bounded, while $B_{NL}^{\Int}$ is positive semidefinite, self-adjoint, and finite rank. Set
\begin{align}\label{eq:observation-vector}
v_{NL}:=\operatorname{Vec}\bigl(v_{\delta,1}^{(1)},\ldots,v_{\delta,L}^{(1)},\ldots, v_{\delta,1}^{(N)},\ldots,v_{\delta,L}^{(N)}\bigr).    
\end{align}

\begin{proposition}[Operator representations]
\label{prop:operator-representations}
Let $\Int\in\{\MP,\SE\}$ and set $\eta:=\frac{\lambda}{\delta^2}$.
Then the empirical and population minimization problems \eqref{eq:emp-pro} and \eqref{eq:exp-pro} admit unique solutions, given by
\begin{equation}\label{eq:empirical-operator-representation}
\begin{aligned}
\widehat h_{\lambda,N,L}^{\Int} &=\frac{1}{\sqrt{NL}}\left(B_{NL}^{\Int}+\eta I\right)^{-1}(A_{NL}^{\Int})^*v_{NL},\\
h_{\lambda,L}^{*,\Int}&=\left(B^\Int+\eta I\right)^{-1} B^\Int\HInt_\delta.
\end{aligned}
\end{equation}
\end{proposition}

\begin{proof}
Both objective functionals are strictly convex and Fr\'echet differentiable. Setting their first variations equal to zero gives the corresponding normal equations. More precisely,
\[
(B_{NL}^{\Int}+\eta I)\widehat h_{\lambda,N,L}^{\Int}=\frac{1}{\sqrt{NL}}(A_{NL}^{\Int})^*v_{NL},\qquad (B^\Int+\eta I)h_{\lambda,L}^{*,\Int}=B^\Int\HInt_\delta.
\]
Since $B_{NL}^{\Int}$ and $B^\Int$ are positive semidefinite and $\eta>0$, both shifted operators are boundedly invertible.
\end{proof}

\paragraph{Kernel representation.}
We next convert the expression for the empirical estimator in \eqref{eq:empirical-operator-representation} into the solution of a finite-dimensional differential Gram-style kernel regression. Applying the differential representer theorem of \cite{hu2024hamiltonian} to the $NL$ scheme-adapted observations yields the following representation.
Define the block differential Gram matrix
\begin{equation*}
\mathbf G_{NL}^\Int :=\left[\nabla_{1,2}K\left(x_l^{\Int,(n)},x_{l'}^{\Int,(n')}\right)\right]_{(n,l),(n',l')}\in\RR^{2dNL\times2dNL},
\end{equation*}
where $\nabla_{1,2}K(x,y)=
\left[ \partial_{1,i}\partial_{2,j}K(x,y) \right]_{i,j=1}^{2d}.$
The matrix $\mathbf G_{NL}^\Int$ is positive semidefinite. We also set $\mathbb{J}_{NL}=\operatorname{diag}\{J,\ldots,J\}$ (with $NL$ diagonal copies).

\begin{proposition}[Kernel representation]
For every $\lambda>0$, the empirical estimator admits the representation
\begin{equation*}
\widehat h_{\lambda,N,L}^{\Int}=\sum_{n=1}^N\sum_{l=1}^L(c_l^{(n)})^\top J\nabla_1K \left(x_l^{\Int,(n)},\cdot\right),
\end{equation*}
where $c_l^{(n)}\in\RR^{2d}$. Writing $\mathbf c_{NL}:=\operatorname{Vec}\left(c_1^{(1)},\ldots,c_L^{(1)},\ldots, c_1^{(N)},\ldots,c_L^{(N)}\right),$ the coefficient vector is given by
\begin{equation*}
\mathbf c_{NL}=\left(\mathbb J_{NL}\mathbf G_{NL}^\Int\mathbb J_{NL}^\top + \eta NLI \right)^{-1}v_{NL}.
\end{equation*}
\end{proposition}

\begin{proof}
This follows directly from the differential representer theorem in \cite{hu2024hamiltonian}, applied to the $NL$ scheme-adapted anchors with effective regularization parameter $\eta=\lambda/\delta^2$.
\end{proof}

\subsection{Identifiability and trajectory coverage}
\label{sec:identifiability}

The regularization parameter $\lambda>0$ guarantees uniqueness of the regularized estimators in \eqref{eq:emp-fun} and \eqref{eq:exp-fun}, but this should be distinguished from identifiability of the Hamiltonian from the trajectory observations. 
For each $\Int\in\{\MP,\SE\}$, define the {\bfi unobservable subspace}
\begin{equation*}
\mathcal H_{\mathrm{null}}^\Int:=\ker A^\Int=\ker B^\Int=\left\{h\in\HK\mid \nabla h=0\quad\text{on }\operatorname{supp}\rho^\Int\right\}.
\end{equation*}
Indeed, the identity $\langle B^\Int h,h\rangle_{\HK}=\|A^\Int h\|_{L^2(\rho^\Int;\mathbb R^{2d})}^2$ yields $\ker B^\Int=\ker A^\Int$.  Moreover, by Assumption~\ref{ass:domain}, $\HK\hookrightarrow C^1(\overline\Omega)$, so that $\nabla h=0$ in $L^2(\rho^\Int;\mathbb{R}^{2d})$ is equivalent to $J\nabla h=0$ on $\operatorname{supp}\rho^\Int$ since $J$ is nondegenerate.
Thus two Hamiltonians whose difference belongs to $\mathcal H_{\mathrm{null}}^\Int$ are indistinguishable from the population trajectory observations.

The Tikhonov term does not make these unobservable directions identifiable; rather, it selects the representative with zero component along them. Indeed, for any $g\in\mathcal H_{\mathrm{null}}^\Int$, the scheme-adapted anchors lie in $\operatorname{supp}\rho^\Int$, and hence $A_{NL}^\Int g=0.$
Taking the $\HK$-inner product of the estimator \eqref{eq:empirical-operator-representation} with $g$ therefore gives $\langle \widehat h_{\lambda,N,L}^\Int,g\rangle_{\HK}=\frac{1}{\eta\sqrt{NL}}\langle v_{NL},A_{NL}^\Int g\rangle_{\mathbb R^{2dNL}}=0.$
Consequently, $\widehat h_{\lambda,N,L}^\Int\in\bigl(\mathcal H_{\mathrm{null}}^\Int\bigr)^\perp.$

Unlike standard regression, the population input law is generated by the dynamics:
\[
(\mu_0,\varphi_t^H,\delta,L,\Int)\longmapsto\rho^\Int\longmapsto B^\Int.
\]
This makes trajectory coverage an intrinsic part of the learning problem. In particular, Hamiltonian trajectories satisfy $H(\varphi_t^H(z))=H(z)$, and may be further constrained by additional first integrals.
Consequently, the initial distribution determines which invariant regions the data represent, while the subsequent dynamics determine how these regions are explored. 
Accordingly, $N$ and $L$ play distinct roles: increasing $N$ improves the statistical sampling of $\mu_0$ without changing the population operator $B^\Int$, whereas increasing $L$ extends the dynamical observation from each initial state.

This distinction suggests separating the information already contained in the initial distribution from that generated by the subsequent trajectory evolution. To this end, define the initial-design operator $B_0:\HK\to\HK$ by
\begin{equation}\label{eq:initial-information-operator}
\langle B_0h,g\rangle_{\HK}:=\int\langle J\nabla h(z),J\nabla g(z)\rangle_{\mathbb R^{2d}}\,\mathrm d\mu_0(z).
\end{equation}
Thus $B_0$ describes the population information available from the initial states alone, whereas $B^\Int$ incorporates the scheme-adapted observations collected along the trajectories. In particular,
\begin{equation}\label{eq:trajectory-information-form}
\langle B^\Int h,g\rangle_{\HK}=\frac1L\sum_{l=1}^L\int\left\langle J\nabla h(X_l^\Int(z)), J\nabla g(X_l^\Int(z))\right\rangle_{\mathbb R^{2d}}\,\mathrm d\mu_0(z).
\end{equation}
The difference $B^\Int-B_0$ therefore measures how much the population information operator changes under dynamical exploration of the initial distribution. This comparison also reveals a trade-off in choosing the sampling interval $\delta$. The inverse-modified target satisfies
\[
\|\check H_\delta^\Int-H\|_{\HK}=\mathcal O(\delta^{r_\Int}),\qquad r_{\MP}=2,\qquad r_{\SE}=1,
\]
whereas, for fixed $L$, the physical observation horizon is $T=L\delta$. Decreasing $\delta$ therefore reduces the integrator-induced error, but also leaves less time for the dynamics to modify the information contained in the initial distribution. The following result quantifies this short-horizon effect.

\begin{proposition}[Short-horizon trajectory information]
\label{prop:short-horizon-information}
Suppose that Assumption~\ref{ass:domain} holds with $s>d+2$, and that the scheme-adapted anchors remain in $\Omega$. Then, for fixed $L$ and sufficiently small $\delta$,
\begin{equation*}
\|B^\Int-B_0\|_{\mathrm{op}} \le CL\delta,
\end{equation*}
where $C$ is independent of $\delta$. 
\end{proposition}

\begin{proof}
By the definitions of $B^\Int$ in \eqref{eq:trajectory-information-form} and $B_0$ in \eqref{eq:initial-information-operator}, for any $h,g\in\HK$,
\begin{align*}
\langle (B^\Int-B_0)h,g\rangle_{\HK}=\frac1L\sum_{l=1}^L\int\Big[\langle J\nabla h(X_l^\Int(z)), J\nabla g(X_l^\Int(z))\rangle-\langle J\nabla h(z), J\nabla g(z)\rangle\Big]\,\mathrm d\mu_0(z).
\end{align*}
The Sobolev embedding bound \eqref{eq:gradient-perturbation-bounds} in  Remark~\ref{rem:derivative-bounds} implies that
\begin{align*}
\left|
\langle J\nabla h(x),J\nabla g(x)\rangle-\langle J\nabla h(y),J\nabla g(y)\rangle\right|\le 2\kappa_1\kappa_2|x-y|\,\|h\|_{\HK}\|g\|_{\HK}.
\end{align*}
Moreover, for fixed $L$ and sufficiently small $\delta$, the scheme-adapted anchors satisfy
\begin{equation*}
\sup_{z\in\operatorname{supp}\mu_0}|X_l^\Int(z)-z| \le C_X l\delta,\qquad 1\le l\le L. \end{equation*}
Therefore,
\[
\left|\langle (B^\Int-B_0)h,g\rangle_{\HK}\right|\le \frac{2\kappa_1\kappa_2C_X}{L}\sum_{l=1}^Ll\delta\,\|h\|_{\HK}\|g\|_{\HK}.
\]
The result follows by taking the supremum over $\|h\|_{\HK},\|g\|_{\HK}\le1$ and setting $C=2\kappa_1\kappa_2C_X$.
\end{proof}

For fixed $L$, Proposition~\ref{prop:short-horizon-information} shows that, over a short physical observation horizon, the trajectory-generated population operator remains close to that induced by the initial distribution:
\[
B^\Int = B_0+\mathcal O(T), \qquad T=L\delta\to0.
\]
Thus dense temporal sampling over a vanishing time interval does not substantially alter the population information contained in the initial distribution. In particular, directions that are weakly represented by $B_0$ cannot be substantially strengthened merely by decreasing $\delta$ while keeping $L$ fixed. This reveals two competing scales in the choice of the sampling interval: a smaller $\delta$ reduces the deterministic target bias at order $\delta^{r_\Int}$, whereas the departure of the trajectory information from the initial design is only $\mathcal O(L\delta)$. Hence, a smaller $\delta$ improves target accuracy but, for fixed $L$, reduces the time available for dynamical exploration.

This trade-off also motivates the Richardson extrapolation developed later on in Section~\ref{sec:richardson}. By reducing the integrator-induced error at a fixed base step, Richardson extrapolation improves target accuracy without necessarily shortening the physical time scale available for dynamical exploration.

\begin{remark}[Full support]
If $\operatorname{supp}\rho^\Int=\overline\Omega$ and $\Omega$ is connected, then the continuity of $\nabla h$ implies $\mathcal H_{\mathrm{null}}^\Int=\operatorname{span}\{1\},$ provided that constants belong to $\HK$. Thus full support is sufficient for identifiability up to an additive constant. This condition is only qualitative, however; the strength of the representation of the observable directions is determined by the positive spectrum of $B^\Int$.
\end{remark}

\section{Error bounds and stability}\label{sec:main-results}

We now analyze the reconstruction error committed by the estimator introduced in Section~\ref{sec:rkhs-formulation}. We first consider the case using the exact trajectory observations \eqref{eq:data-generation}. In this setting, randomness comes only from the sampled initial conditions $z_0^{(n)}\sim\mu_0$, and thus from fluctuations in the empirical trajectory-anchor operator around its population counterpart $B^\Int$. We consider the effect of additive observation noise separately in Section~\ref{sec:deterministic-perturbations}.

At a fixed time step $\delta$, the scheme-adapted data provide exact derivative observations of the inverse modified Hamiltonian $\HInt_\delta$, rather than of the physical Hamiltonian $H$. Moreover, as discussed in Section~\ref{sec:identifiability}, Hamiltonians can only be recovered on the {\bfi identifiable subspace} $(\mathcal H_{\mathrm{null}}^\Int)^\perp$; the source condition introduced below ensures that the fixed-step target belongs to this subspace.

Accordingly, the reconstruction error naturally separates into three contributions: a finite-sample estimation error, an approximation error, and an integrator-induced error.  More precisely, with $h_{\lambda,L}^{*,\Int}$ denoting the population regularized minimizer,
\begin{equation}\label{eq:error-decomposition}
\begin{aligned}\widehat h_{\lambda,N,L}^{\Int}-H\quad= \quad \underbrace{\widehat h_{\lambda,N,L}^{\Int}-h_{\lambda,L}^{*,\Int} }_{\mathrm{estimation\ error}}\quad + \underbrace{h_{\lambda,L}^{*,\Int}-\HInt_\delta}_{\mathrm{approximation\ error}}+\underbrace{\HInt_\delta-H}_{\mathrm{integrator-induced\ error}}.
\end{aligned}
\end{equation}
The first two terms concern the reconstruction of the fixed-step target and are controlled below by the RKHS operator framework of Section~\ref{sec:estimator-representations}. The third term is specific to the chosen symplectic integrator and, by the backward error analysis leading to
\eqref{eq:generic-native-bias}, satisfies
\[
\|\HInt_\delta-H\|_{\HK}\le C_\Int\delta^{r_\Int}, \qquad r_{\MP}=2,\qquad r_{\SE}=1.
\]
After establishing the corresponding reconstruction and consistency bounds, we turn to noisy trajectory observations in Section~\ref{sec:deterministic-perturbations} and study the effect of observation noise pathwise.

\subsection{Estimation and approximation errors}
\label{sec:estimation-regularization}

We first control the estimation and approximation errors in \eqref{eq:error-decomposition} associated with the reconstruction of the fixed-step target $\HInt_\delta$. The $L$ observations along a trajectory are generally dependent, whereas the $N$ trajectories are independent.
We therefore treat each trajectory as one observation block and apply Hilbert-space concentration across the $N$ blocks.

\paragraph{Trajectory-block concentration.}
For a fixed $h\in\HK$, associate with the $n$th trajectory the $\HK$-valued random variable
\[
\xi_{\Int}^{(n)}(h):= \frac1L\sum_{l=1}^L\sum_{i=1}^{2d}\partial_i h(x_l^{\Int,(n)})(\nabla_1K)_i(x_l^{\Int,(n)},\cdot),\quad n=1,\ldots, N.
\]
Although the $L$ terms entering each $\xi_{\Int}^{(n)}(h)$ are generally dependent, the random variables $\{\xi_{\Int}^{(n)}(h)\}_{n=1}^N$ are independent and identically distributed in $\HK$.
Moreover, $B_{NL}^{\Int}h=\frac1N\sum_{n=1}^N\xi_{\Int}^{(n)}(h), $ and $ \mathbb E[\xi_{\Int}^{(n)}(h)] = B^\Int h,$ where the second identity follows from the definition of $\rho^\Int$ in \eqref{eq:rho-generic}. 
Consequently,
\[
B_{NL}^{\Int}h-B^\Int h=\frac1N\sum_{n=1}^N\left(\xi_{\Int}^{(n)}(h)-\mathbb E[\xi_{\Int}^{(n)}(h)]\right).
\]
The following result applies Hilbert-space concentration across the independent trajectory blocks, without requiring independence among the anchors within each trajectory.

\begin{lemma}[Trajectory-block concentration]
\label{lem:trajectory-concentration}
For any $h\in\HK$ and $0<\varepsilon<1$, with probability at least $1-\varepsilon$,
\begin{equation}\label{eq:trajectory-concentration}
\|(B_{NL}^{\Int}-B^\Int)h\|_{\HK} \le 2d C\kappa^2 \sqrt{\frac{\log(2/\varepsilon)}{N}} \|h\|_{\HK},
\end{equation}
where $C>0$ is a universal constant.
\end{lemma}

\begin{proof}
By the differential reproducing property \cite[Theorem 2.7]{hu2024hamiltonian}, $|\partial_i h(x)| \le \kappa\|h\|_{\HK},$ where  $\kappa^2=\|K\|_{C_b^{2}(\Omega \times \Omega)}$. Hence
\[
\|\xi_{\Int}^{(n)}(h)\|_{\HK}\le \frac1L\sum_{l=1}^L\sum_{i=1}^{2d}|\partial_i h(x_l^{\Int,(n)})|\|(\nabla_1K)_i(x_l^{\Int,(n)},\cdot)\|_{\HK}\le 2d\kappa^2\|h\|_{\HK}.
\]
Set $Y_n(h):= \xi_{\Int}^{(n)}(h)-\mathbb E[\xi_{\Int}^{(n)}(h)].$
Then $\{Y_n(h)\}_{n=1}^N$ are independent, centered $\HK$-valued random variables satisfying $\|Y_n(h)\|_{\HK}\le 4d\kappa^2\|h\|_{\HK}.$
Since every Hilbert space is $(2,1)$-smooth, Theorem~3.5 of \cite{pinelis1994}, applied to the martingale $N^{-1}\sum_{n=1}^jY_n(h)$, $j=1,\ldots,N$, implies \eqref{eq:trajectory-concentration}.
\end{proof}

\paragraph{Estimation error.}

We next use Lemma~\ref{lem:trajectory-concentration} together with a resolvent argument to bound the estimation error.

\begin{proposition}[Estimation error]
\label{thm:estimation}
Let $\eta=\lambda/\delta^2$. For $0<\varepsilon<1$, with probability at least $1-\varepsilon$,
\begin{equation}\label{eq:estimation-bound}
\|\widehat h_{\lambda,N,L}^{\Int}-h_{\lambda,L}^{*,\Int}\|_{\HK}\le C\,\frac{2d\kappa^2}{\eta} \sqrt{\frac{\log(4/\varepsilon)}{N}}\,\|\HInt_\delta\|_{\HK},
\end{equation}
where $C>0$ is a universal constant.
\end{proposition}

\begin{proof}
By \eqref{eq:exact-scheme-adapted-regression} and Proposition~\ref{prop:operator-representations}, we have $\widehat h_{\lambda,N,L}^{\Int}=(B_{NL}^{\Int}+\eta I)^{-1}B_{NL}^{\Int}\HInt_\delta, $ and $ h_{\lambda,L}^{*,\Int} =(B^\Int+\eta I)^{-1}B^\Int\HInt_\delta.$
Set $R_{NL}:=(B_{NL}^{\Int}+\eta I)^{-1}$ and $R:=(B^\Int+\eta I)^{-1}.$ The resolvent identity gives
\[
\widehat h_{\lambda,N,L}^{\Int}-h_{\lambda,L}^{*,\Int}=R_{NL}(B_{NL}^{\Int}-B^\Int)\HInt_\delta+R_{NL}(B^\Int-B_{NL}^{\Int})RB^\Int\HInt_\delta.
\]
Since $B_{NL}^{\Int}$ and $B^\Int$ are positive semidefinite, $\|R_{NL}\|_{\mathrm{op}}\le\eta^{-1}, $ and $ \|RB^\Int\|_{\mathrm{op}}\le1.$
Therefore,
\begin{align*}
\|\widehat h_{\lambda,N,L}^{\Int}-h_{\lambda,L}^{*,\Int}\|_{\HK} \le \eta^{-1}\Big(\|(B_{NL}^{\Int}-B^\Int)\HInt_\delta\|_{\HK}+ \|(B_{NL}^{\Int}-B^\Int)RB^\Int\HInt_\delta\|_{\HK}\Big).
\end{align*}
Applying Lemma~\ref{lem:trajectory-concentration} to $\HInt_\delta$ and $RB^\Int\HInt_\delta$, each with failure probability $\varepsilon/2$, and using a union bound yields
\eqref{eq:estimation-bound}.
\end{proof}

\paragraph{Approximation error.}
It remains to quantify the approximation error induced by the regularized population problem. We impose the standard source condition on the fixed-step target; see, e.g., \cite{engl1996regularization}.

\begin{assumption}[Source condition]
\label{ass:source}
For the fixed step size $\delta$, there exist $\gamma\in(0,1]$ and $\psi_\delta^\Int\in\HK$ such that
\begin{equation}\label{eq:source-condition}
\HInt_\delta=(B^\Int)^\gamma\psi_\delta^\Int.
\end{equation}
\end{assumption}

Assumption~\ref{ass:source} implies that the fixed-step target $\HInt_\delta$ belongs to the identifiable subspace. 
Moreover, $\overline{\operatorname{Ran}(B^\Int)^\gamma}= (\mathcal H_{\mathrm{null}}^\Int)^\perp,$ so the source spaces are dense in the identifiable subspace; see \cite{hu2026kernel}. Hence the source condition is primarily a quantitative regularity assumption.

By Proposition~\ref{prop:operator-representations}, we have $h_{\lambda,L}^{*,\Int}-\HInt_\delta = -\eta(B^\Int+\eta I)^{-1}\HInt_\delta.$
Using \eqref{eq:source-condition} and the spectral calculus for the
positive self-adjoint operator $B^\Int$,
\[
\left\|
\eta(B^\Int+\eta I)^{-1}(B^\Int)^\gamma \right\|_{\mathrm{op}} \le \sup_{\sigma\ge0}\frac{\eta\sigma^\gamma}{\sigma+\eta} \le \eta^\gamma,
\qquad 0<\gamma\le1.
\]
Consequently,
\begin{equation}\label{eq:regularization-bound}
\|h_{\lambda,L}^{*,\Int}-\HInt_\delta\|_{\HK} \le \eta^\gamma\|\psi_\delta^\Int\|_{\HK} =\left(\frac{\lambda}{\delta^2}\right)^\gamma \|\psi_\delta^\Int\|_{\HK}.
\end{equation}
Thus, for fixed $\delta$, the approximation error vanishes as $\lambda\to0$ (equivalently, $\eta=\lambda/\delta^2\to0$) at the rate $\eta^\gamma$, with a constant determined by the source norm $\|\psi_\delta^\Int\|_{\HK}$.

\subsection{Total error and consistency}

Combining the estimation, approximation, and integrator-induced errors yields the main reconstruction bound. Recall that $\eta=\lambda/\delta^2$, $r_{\MP}=2$, and $r_{\SE}=1$. 

\begin{theorem}[Reconstruction error] \label{thm:global-error}
Let $\Int\in\{\MP,\SE\}$ and fix an admissible step size $\delta>0$.
Suppose that Assumptions~\ref{ass:analytic-setting}, \ref{ass:domain} and~\ref{ass:source} hold.
Then, for every $\lambda>0$ and $0<\varepsilon<1$, with probability at least $1-\varepsilon$,
\begin{equation}\label{eq:global-error}
\begin{aligned}
\|\widehat h_{\lambda,N,L}^{\Int}-H\|_{\HK}\le  C\frac{2d\kappa^2}{\eta} \sqrt{\frac{\log(4/\varepsilon)}{N}} \|\HInt_\delta\|_{\HK}+\eta^\gamma\|\psi_\delta^\Int\|_{\HK} + C_\Int\delta^{r_\Int}.
\end{aligned}
\end{equation}
Here $C>0$ is universal, while $C_\Int$ is the constant in the integrator-induced error estimate~\eqref{eq:generic-native-bias}.
\end{theorem}

\begin{proof}
The result follows directly from the error decomposition \eqref{eq:error-decomposition}, Proposition~\ref{thm:estimation}, the regularization estimate~\eqref{eq:regularization-bound}, and the integrator-induced error bound~\eqref{eq:generic-native-bias}.
\end{proof}

\begin{remark}
The three terms in \eqref{eq:global-error} have different origins.
The first is the finite-sample estimation error. The second is the approximation
error associated with the regularized population problem and decreases as
$\eta\to0$. Decreasing $\eta$, however, amplifies finite-sample fluctuations.
The third is the integrator-induced error associated with reconstructing the
fixed-step inverse target $\HInt_\delta$ rather than the physical Hamiltonian $H$.
It cannot be reduced by increasing the number of trajectories or by decreasing
the regularization parameter. Hence, once the estimation and approximation
errors are sufficiently small, a fixed-step error floor of order
$\delta^{r_\Int}$ remains.
\end{remark}

The fixed-step error in Theorem~\ref{thm:global-error} motivates a simultaneous refinement of the sampling size, regularization parameter, and time step. We next balance the three terms in \eqref{eq:global-error} to obtain convergence toward the physical Hamiltonian. To quantify the possible dependence of the source condition on the step size, define the step-dependent source norm
\begin{equation*}
S_\delta^\Int:=\inf\left\{\|\psi\|_{\HK}\mid\psi\in\HK,\quad \HInt_\delta=(B^\Int)^\gamma\psi \right\}.
\end{equation*}
Thus $S_\delta^\Int$ is the minimum source norm associated with the fixed-step target. By Assumption~\ref{ass:source}, $S_\delta^\Int<\infty$ for each fixed admissible $\delta$. For the joint limit below, we assume that these source norms remain uniformly bounded as $\delta\to0$.

\begin{corollary}[Balanced joint convergence] \label{cor:balanced-joint-rate}
Let $\Int\in\{\MP,\SE\}$ and suppose that Assumptions~\ref{ass:analytic-setting} and~\ref{ass:domain} hold.
Assume further that the source condition in
Assumption~\ref{ass:source} holds with a common exponent $\gamma\in(0,1]$ for all $0<\delta\le\delta_0$ and that
\begin{equation}\label{eq:uniform-source-norm}
\sup_{0<\delta\le\delta_0}S_\delta^\Int<\infty.
\end{equation}
Choose $\eta_N\asymp N^{-1/(2\gamma+2)}, $ $ \delta_N\asymp N^{-\gamma/[r_\Int(2\gamma+2)]},$ and set $\lambda_N:=\eta_N\delta_N^2\asymp N^{-(r_\Int+2\gamma)/[2r_\Int(\gamma+1)]}$. Then
\begin{equation*}
\|\widehat h_{\lambda_N,N,L}^{\Int}-H\|_{\HK} = \mathcal O_{\mathbb P} \left(N^{-\gamma/(2\gamma+2)}\right).
\end{equation*}
\end{corollary}

\begin{proof}
By \eqref{eq:generic-native-bias}, we have $\|\HInt_\delta\|_{\HK} \le \|H\|_{\HK}+C_\Int\delta^{r_\Int},$ and hence $\|\HInt_\delta\|_{\HK}$ is uniformly bounded for $0<\delta\le\delta_0$. Moreover, by the definition of $S_\delta^\Int$ and \eqref{eq:regularization-bound}, the approximation error is bounded by $\eta^\gamma S_\delta^\Int$.
Therefore, \eqref{eq:uniform-source-norm} and Theorem~\ref{thm:global-error} give
\[
\|\widehat h_{\lambda,N,L}^{\Int}-H\|_{\HK} = \mathcal O_{\mathbb P} \left( \eta^{-1}N^{-1/2} +\eta^\gamma +\delta^{r_\Int}\right).
\]
Balancing the estimation and approximation terms yields  $\eta_N\asymp N^{-1/(2\gamma+2)}, $ and $ \eta_N^\gamma \asymp N^{-\gamma/(2\gamma+2)}.$
Choosing $\delta_N$ such that $\delta_N^{r_\Int} \asymp N^{-\gamma/(2\gamma+2)}$ gives the stated scaling of $\delta_N$ and makes all three error terms of order $N^{-\gamma/(2\gamma+2)}$. Finally, $\lambda_N=\eta_N\delta_N^2$ gives the claimed scaling of $\lambda_N$.
\end{proof}

In particular, the two integrators attain the same balanced reconstruction rate, but the required step-size refinements depend on the order of the integrator-induced error. More precisely,
\[
\delta_N^{\MP}\asymp N^{-\gamma/[4(\gamma+1)]}, \qquad \delta_N^{\SE} \asymp N^{-\gamma/[2(\gamma+1)]}.
\]
Thus, the second-order inverse-target accuracy of the implicit midpoint rule allows a slower refinement of the observation step than symplectic Euler while retaining the same overall reconstruction rate.

\begin{remark}[Relaxing the uniform source condition]
The uniform source bound \eqref{eq:uniform-source-norm} can be relaxed.
Suppose, more generally, that for some $\beta\ge0$, $S_\delta^\Int\lesssim\delta^{-\beta},$ as $\delta\to0.$
Then the three contributions in Theorem~\ref{thm:global-error} scale as $\eta^{-1}N^{-1/2}, \eta^\gamma\delta^{-\beta}, $ and $ \delta^{r_\Int}.$
Balancing these terms gives $\eta_N \asymp N^{-(r_\Int+\beta)/[2(r_\Int(\gamma+1)+\beta)]}, $ and $ \delta_N \asymp N^{-\gamma/[2(r_\Int(\gamma+1)+\beta)]},$ and consequently
\[
\|\widehat h_{\lambda_N,N,L}^{\Int}-H\|_{\HK} = \mathcal O_{\mathbb P} \left( N^{-r_\Int\gamma/[2(r_\Int(\gamma+1)+\beta)]}\right).
\]
Corollary~\ref{cor:balanced-joint-rate} corresponds to the case $\beta=0$. Thus the convergence rate reflects a competition between the integrator order $r_\Int$ and the deterioration of the source norm as $\delta\to0$.
\end{remark}

\subsection{Pathwise robustness under additive observation noise}
\label{sec:deterministic-perturbations}

The analysis above concerns the exact trajectory observations. We now consider the noisy observations introduced in \eqref{eq:noisy-observation}.
The stability analysis is pathwise with respect to the observation noise: for each realization of the noise variables, we compare the estimator constructed from the noisy trajectory states with its clean-data counterpart. 

\paragraph{Noisy scheme-adapted observations.}

Fix $\Int\in\{\MP,\SE\}$ and recall that the noisy trajectory states are
\begin{align}\label{eq:perturbed-data}
\widetilde z_l^{(n)}=z_l^{(n)}+e_l^{(n)},    
\end{align}
where $e_l^{(n)}$ are $\RR^{2d}$-valued random variables. For the corresponding latent exact-flow pair $(z_{l-1}^{(n)},z_l^{(n)})$, let $x_l^{\Int,(n)}:=x^\Int(z_{l-1}^{(n)},z_l^{(n)})$ and $v_{\delta,l}^{(n)}:=\frac{z_l^{(n)}-z_{l-1}^{(n)}}{\delta}.$
By the exact inverse-target identity,
\begin{equation}\label{eq:exact-inverse-target-identity}
v_{\delta,l}^{(n)}
=
J\nabla\HInt_\delta
\bigl(x_l^{\Int,(n)}\bigr).
\end{equation}

Writing
$e_l^{(n)}=(e_{q,l}^{(n)},e_{p,l}^{(n)})$,
the noisy scheme-adapted anchor takes the form $\widetilde x_l^{\Int,(n)}=x_l^{\Int,(n)}+U_l^{\Int,(n)},$
where the induced anchor-noise random variable is
\[
U_l^{\MP,(n)}=\frac12\left(e_{l-1}^{(n)}+e_l^{(n)}\right),
\qquad
U_l^{\SE,(n)}=\left(e_{q,l-1}^{(n)},e_{p,l}^{(n)}\right).
\]
The noisy discrete velocity satisfies
\begin{equation}\label{eq:perturbed-velocity}
\widetilde v_{\delta,l}^{(n)}=v_{\delta,l}^{(n)}+W_l^{(n)},
\qquad
W_l^{(n)}:=\frac{e_l^{(n)}-e_{l-1}^{(n)}}{\delta},
\end{equation}
where $W_l^{(n)}$ is the induced velocity-noise random variable.
Thus observation noise enters the regression problem through two channels: an integrator-dependent perturbation of the regression input and a perturbation of the regression output.

\paragraph{Stability of the estimator.}
Given the noisy trajectory observations \eqref{eq:perturbed-data}, we consider the minimization problem
\begin{equation*}
\widehat h_{\lambda,NL}^{U,W}:=\arg\min_{h\in\HK}\left\{\frac{\delta^2}{NL} \sum_{n=1}^N\sum_{l=1}^L \left|J\nabla h\bigl(x_l^{\Int,(n)}+U_l^{\Int,(n)}\bigr)-\bigl(v_{\delta,l}^{(n)}+W_l^{(n)}\bigr)\right|^2+\lambda\|h\|_{\HK}^2  \right\}.
\end{equation*}

To derive an operator representation, define the {\bfi noisy empirical sampling operator} $A_{NL}^{U,\Int}:\HK\longrightarrow\RR^{2dNL}$ by
\begin{equation*}
A_{NL}^{U,\Int}h:=\frac{1}{\sqrt{NL}}\operatorname{Vec}\left(J\nabla h(x_1^{\Int,(1)}+U_1^{\Int,(1)}),\ldots,J\nabla h(x_L^{\Int,(N)}+U_L^{\Int,(N)})\right),
\end{equation*}
with associated covariance operator $B_{NL}^{U,\Int}:=(A_{NL}^{U,\Int})^*A_{NL}^{U,\Int}$.
Set $\eta:=\frac{\lambda}{\delta^2}$. We also set $W_{NL}:=\operatorname{Vec}\bigl(W_1^{(1)},\ldots,W_L^{(1)},\ldots, W_1^{(N)},\ldots,W_L^{(N)}\bigr).$
Arguing as in Proposition~\ref{prop:operator-representations}, we obtain
\begin{equation}\label{eq:perturbed-operator-representation}
\widehat h_{\lambda,NL}^{U,W}=\frac{1}{\sqrt{NL}}\left(B_{NL}^{U,\Int}+\eta I\right)^{-1}(A_{NL}^{U,\Int})^*\left(v_{NL}+W_{NL}\right).
\end{equation}

The representation \eqref{eq:perturbed-operator-representation} shows that the observation noise affects the estimator in two different ways.
The velocity-noise term $W_{NL}$ changes the right-hand side of the normal equation, whereas the anchor-noise term $U$ changes the sampling operator itself.
To separate these two effects,
introduce the auxiliary estimator
\begin{align}\label{eq:h-U-0}
\widehat h_{\lambda,NL}^{U,0}
:=
\frac{1}{\sqrt{NL}}
\left(B_{NL}^{U,\Int}+\eta I\right)^{-1}
(A_{NL}^{U,\Int})^*v_{NL},
\end{align}
which uses the noisy anchors but the clean velocities. Then
\[
\widehat h_{\lambda,NL}^{U,W}
-\widehat h_{\lambda,N,L}^{\Int}
=
\left(
\widehat h_{\lambda,NL}^{U,W}
-\widehat h_{\lambda,NL}^{U,0}
\right)
+
\left(
\widehat h_{\lambda,NL}^{U,0}
-\widehat h_{\lambda,N,L}^{\Int}
\right).
\]
The first difference isolates the perturbation of the regression outputs, while the second isolates the perturbation of the regression inputs.
For a fixed realization of the observation noise, we therefore define the pathwise noise levels
\begin{equation}\label{eq:pathwise-noise-radii}
r_U:=\frac{1}{NL}\sum_{n=1}^N\sum_{l=1}^L|U_l^{\Int,(n)}|,
\qquad
r_W:=\left(\frac{1}{NL}\sum_{n=1}^N\sum_{l=1}^L|W_l^{(n)}|^2\right)^{1/2}.
\end{equation}

\begin{proposition}[Pathwise stability]
\label{prop:deterministic-noisy-sensitivity}
Let $\Int\in\{\MP,\SE\}$ and suppose that Assumption~\ref{ass:domain} holds with $s>d+2$. 
Assume that the line segment joining each clean anchor $x_l^{\Int,(n)}$ and its noisy counterpart $x_l^{\Int,(n)}+U_l^{\Int,(n)}$ is contained in $\Omega$.
Then, for every $\eta=\lambda/\delta^2>0$,
\begin{equation*}
\left\|
\widehat h_{\lambda,NL}^{U,W}-\widehat h_{\lambda,N,L}^{\Int}\right\|_{\HK}\le\frac{r_W}{2\sqrt{\eta}}+\frac{3\kappa_1\kappa_2}{\eta} \|\HInt_\delta\|_{\HK}\,r_U,
\end{equation*}
where $\kappa_1,\kappa_2>0$ are constants determined by the continuous embedding $\HK\hookrightarrow C^2(\overline\Omega)$.
\end{proposition}

\begin{proof}
We first control the output perturbation. By \eqref{eq:perturbed-operator-representation} and the definition of $\widehat h_{\lambda,NL}^{U,0}$ in \eqref{eq:h-U-0},
\[
\widehat h_{\lambda,NL}^{U,W}- \widehat h_{\lambda,NL}^{U,0} = \frac{1}{\sqrt{NL}} \left(B_{NL}^{U,\Int}+\eta I\right)^{-1} (A_{NL}^{U,\Int})^*W_{NL}.
\]
Since $B_{NL}^{U,\Int} = (A_{NL}^{U,\Int})^*A_{NL}^{U,\Int},$
the action of $\left(B_{NL}^{U,\Int}+\eta I\right)^{-1} (A_{NL}^{U,\Int})^*$ along a singular direction of $A_{NL}^{U,\Int}$ with singular value $s$ is multiplication by $\frac{s}{s^2+\eta}.$
Therefore,
\[
\left\|\left(B_{NL}^{U,\Int}+\eta I\right)^{-1}(A_{NL}^{U,\Int})^*\right\|_{\mathrm{op}}\le\sup_{s\ge0}\frac{s}{s^2+\eta}= \frac{1}{2\sqrt{\eta}},
\]
where the maximum is attained at $s=\sqrt{\eta}$. Since $\frac{1}{\sqrt{NL}}\|W_{NL}\|=r_W$, we obtain
\begin{equation}\label{eq:output-perturbation-bound}
\left\|
\widehat h_{\lambda,NL}^{U,W}-\widehat h_{\lambda,NL}^{U,0}\right\|_{\HK}\le \frac{1}{\sqrt{NL}}\left\|\left(B_{NL}^{U,\Int}+\eta I\right)^{-1}(A_{NL}^{U,\Int})^*\right\|_{\mathrm{op}}\|W_{NL}\|\le \frac{r_W}{2\sqrt{\eta}}.
\end{equation}

We next control the perturbation of the regression inputs. By the exact inverse-target identity in \eqref{eq:exact-inverse-target-identity}, we have
\begin{align}\label{eq:identity1}
\left(B_{NL}^{U,\Int}+\eta I\right)\widehat h_{\lambda,NL}^{U,0}=(A_{NL}^{U,\Int})^*A_{NL}^{\Int}\HInt_\delta,
\end{align}
whereas Proposition~\ref{prop:operator-representations} gives
\begin{align}\label{eq:identity2}
\left(B_{NL}^{\Int}+\eta I\right)\widehat h_{\lambda,N,L}^{\Int}=B_{NL}^{\Int}\HInt_\delta.
\end{align}
Subtracting the identities \eqref{eq:identity1} and \eqref{eq:identity2} yields
\[
\left(B_{NL}^{U,\Int}+\eta I\right)\left(\widehat h_{\lambda,NL}^{U,0}-\widehat h_{\lambda,N,L}^{\Int}\right)=\left((A_{NL}^{U,\Int})^*A_{NL}^{\Int}-B_{NL}^{\Int}\right)\HInt_\delta-\left(B_{NL}^{U,\Int}-B_{NL}^{\Int}\right)\widehat h_{\lambda,N,L}^{\Int}.
\]

We now estimate the two operator perturbations.
For any $f,g\in\HK$ with $\|f\|_{\HK},\|g\|_{\HK}\le1$, the bounds \eqref{eq:gradient-perturbation-bounds} and definition \eqref{eq:pathwise-noise-radii} give
\begin{align*}
&\left|\left\langle\left((A_{NL}^{U,\Int})^*A_{NL}^{\Int}-B_{NL}^{\Int}\right)f,g\right\rangle_{\HK}\right|=\left|\left\langle A_{NL}^{\Int}f,\,\left(A_{NL}^{U,\Int}-A_{NL}^{\Int}\right)g\right\rangle_{\RR^{2dNL}}\right|\\
&\qquad=\frac1{NL}\left|\sum_{n=1}^N\sum_{l=1}^L\left\langle J\nabla f(x_l^{\Int,(n)}),\,J\nabla g(x_l^{\Int,(n)}+U_l^{\Int,(n)})-J\nabla g(x_l^{\Int,(n)})\right\rangle\right|\\
&\qquad\le\frac1{NL}\sum_{n=1}^N\sum_{l=1}^L\kappa_1\kappa_2|U_l^{\Int,(n)}|=\kappa_1\kappa_2 r_U.
\end{align*}
Therefore,
\begin{equation}\label{eq:cross-operator-perturbation}
\left\|
(A_{NL}^{U,\Int})^*A_{NL}^{\Int}
-
B_{NL}^{\Int}
\right\|_{\mathrm{op}}
\le
\kappa_1\kappa_2 r_U.
\end{equation}

Similarly, using 
\[
B_{NL}^{U,\Int}-B_{NL}^{\Int}=\bigl(A_{NL}^{U,\Int}-A_{NL}^{\Int}\bigr)^*A_{NL}^{\Int}+(A_{NL}^{U,\Int})^*\bigl(A_{NL}^{U,\Int}-A_{NL}^{\Int}\bigr),
\]
and applying the same argument to each term, we obtain
\begin{equation}\label{eq:covariance-operator-perturbation}
\left\|B_{NL}^{U,\Int}-B_{NL}^{\Int}\right\|_{\mathrm{op}}\le 2\kappa_1\kappa_2 r_U.
\end{equation}

Since $B_{NL}^{\Int}$ and $B_{NL}^{U,\Int}$ are positive semidefinite, we obtain $\left\|(B_{NL}^{\Int}+\eta I)^{-1}B_{NL}^{\Int}\right\|_{\mathrm{op}}\le1, $ and $\left\|(B_{NL}^{U,\Int}+\eta I)^{-1}\right\|_{\mathrm{op}}\le\eta^{-1}.$
In particular,
$\|\widehat h_{\lambda,N,L}^{\Int}\|_{\HK}\le\|\HInt_\delta\|_{\HK}$.
Together with \eqref{eq:cross-operator-perturbation}--\eqref{eq:covariance-operator-perturbation}, the preceding identity therefore yields
\begin{equation}\label{eq:input-perturbation-bound}
\left\|\widehat h_{\lambda,NL}^{U,0}-\widehat h_{\lambda,N,L}^{\Int}\right\|_{\HK}\le\frac{3\kappa_1\kappa_2}{\eta}\|\HInt_\delta\|_{\HK}r_U.
\end{equation}
The conclusion follows from \eqref{eq:output-perturbation-bound} and \eqref{eq:input-perturbation-bound}.
\end{proof}

Combining Proposition~\ref{prop:deterministic-noisy-sensitivity} with Theorem~\ref{thm:global-error} gives the overall reconstruction bound
\begin{align*}
\|\widehat h_{\lambda,NL}^{U,W}-H\|_{\HK} \le\ & C\frac{2d\kappa^2}{\eta} \sqrt{\frac{\log(4/\varepsilon)}{N}} \|\HInt_\delta\|_{\HK} + \eta^\gamma\|\psi_\delta^\Int\|_{\HK} + C_\Int\delta^{r_\Int}\\
&+ \frac{r_W}{2\sqrt{\eta}}+\frac{3\kappa_1\kappa_2}{\eta}\|\HInt_\delta\|_{\HK} r_U,
\end{align*}
with probability at least $1-\varepsilon$.

Thus the total reconstruction error separates into four contributions:
finite-sample estimation error, approximation error, integrator-induced error,
and observation noise. The observation noise acts through two different channels: velocity noise is amplified by $\eta^{-1/2}$, whereas anchor noise is amplified by $\eta^{-1}$.
Since endpoint errors are differenced in the discrete velocity, \eqref{eq:perturbed-velocity} further reveals an explicit $\delta^{-1}$ sensitivity. Hence decreasing the observation step reduces the integrator-induced error
but may simultaneously amplify observation noise.

\section{Richardson extrapolation and integrator-error reduction}
\label{sec:richardson}

The reconstruction bound in Theorem \ref{thm:global-error}  contains an integrator-induced error
\[
\|\HInt_\delta-H\|_{\mathcal H_K}
=\mathcal O(\delta^{r_\Int}),
\qquad r_{\MP}=2,\qquad r_{\SE}=1,
\]
which cannot be reduced by increasing the sample size or decreasing the
regularization parameter. Once the estimation and approximation errors fall
below this level, the discrepancy between the inverse-modified target and the
physical Hamiltonian becomes the dominant source of reconstruction error.

A useful feature of trajectory data is that this integrator-induced error
can be reduced without collecting additional trajectories. From the same uniformly sampled trajectory $z_0,z_1,\ldots,z_L, $ with $ z_l=\varphi_\delta^H(z_{l-1}),$
we may form not only adjacent pairs $(z_{l-1},z_l)$, but also pairs $(z_{l-s},z_l)$ separated by $s$ time steps. Such a pair represents the exact flow over the interval $s\delta$ and, with the corresponding scheme-adapted anchor, provides an exact derivative observation of the inverse modified Hamiltonian $\check H_{s\delta}^{\Int}$.

The resulting stride-dependent targets have systematically related deterministic biases. For implicit midpoint, $\check H_{s\delta}^{\MP}= H+\sum_{j\ge1}s^{2j}\delta^{2j}\check H_{2j}^{\MP},$ whereas for symplectic Euler, $\check H_{s\delta}^{\SE}= H+\sum_{j\ge1}s^j\delta^j\check H_j^{\SE}.$
Consequently, suitable linear combinations of the stride estimators cancel the leading inverse-target corrections. With $Q$ strides, the deterministic bias improves from $\mathcal O(\delta^2)$ to $\mathcal O(\delta^{2Q})$ for implicit midpoint and from $\mathcal O(\delta)$ to $\mathcal O(\delta^Q)$ for symplectic Euler.

Figure~\ref{fig:richardson-data-reuse} illustrates the basic two-level construction. The same observed trajectory is reused at two temporal separations: stride~$1$ yields a fit targeting $\check H_\delta^\Int$, whereas stride~$2$ yields a second fit targeting $\check H_{2\delta}^\Int$. Their leading biases have known relative scalings and can therefore be canceled by an appropriate linear combination. 
For example, the combinations $\frac43\widehat h_{1,\delta}^{\MP}-\frac13\widehat h_{2,\delta}^{\MP}$ and $2\widehat h_{1,\delta}^{\SE} -\widehat h_{2,\delta}^{\SE}$ raise the deterministic orders from $2$ to $4$ and from $1$ to $2$, respectively.

\begin{figure}[ht]
\centering
\begin{tikzpicture}[x=1cm,y=1cm,>=Stealth,line cap=round,line join=round]
  \colorlet{strideone}{blue!68!black}
  \colorlet{stridetwo}{orange!88!black}

  \draw[gray!65,line width=0.8pt] (0,1.45) -- (5,1.45);
  \foreach \a/\b in {0/1.25,1.25/2.5,2.5/3.75,3.75/5}
    \draw[strideone,line width=1pt] (\a,1.43) to[bend right=27] (\b,1.43);
  \foreach \a/\b in {0/2.5,1.25/3.75,2.5/5}
    \draw[stridetwo,line width=1pt] (\a,1.47) to[bend left=24] (\b,1.47);
  \foreach \x/\i in {0/0,1.25/1,2.5/2,3.75/3,5/4}{
    \fill[black] (\x,1.45) circle (2pt);
    \node[below=2pt,fill=white,inner sep=0.5pt] at (\x,1.45) {$z_{\i}$};
  }
  \node[font=\footnotesize] at (2.5,2.85) {one observed trajectory};
  \node[stridetwo,font=\scriptsize,fill=white,inner sep=1pt] at (2.5,2.23)
    {stride $2$: lag $2\delta$};
  \node[strideone,font=\scriptsize,fill=white,inner sep=1pt] at (2.5,0.70)
    {stride $1$: lag $\delta$};

  \node[draw=stridetwo,fill=orange!6,rounded corners=1pt,align=center,
        text width=4.25cm,inner sep=4pt,font=\footnotesize] (fittwo) at (8.05,2.25)
    {\textcolor{stridetwo}{\textbf{stride 2 fit}}\\[-1pt]
     $\widehat h^{\Int}_{2,\delta}\ \leadsto\ \check H^{\Int}_{2\delta}$\\[-1pt]
     same trajectory, lag $2\delta$};
  \node[draw=strideone,fill=blue!5,rounded corners=1pt,align=center,
        text width=4.25cm,inner sep=4pt,font=\footnotesize] (fitone) at (8.05,0.42)
    {\textcolor{strideone}{\textbf{stride 1 fit}}\\[-1pt]
     $\widehat h^{\Int}_{1,\delta}\ \leadsto\ \check H^{\Int}_\delta$\\[-1pt]
     same trajectory, lag $\delta$};
  \draw[->,stridetwo,line width=0.8pt] (5.12,1.72) to[bend left=8] (fittwo.west);
  \draw[->,strideone,line width=0.8pt] (5.12,1.18) to[bend right=8] (fitone.west);

  \node[draw=black!70,fill=green!7,rounded corners=1pt,align=center,
        text width=4.65cm,inner sep=5pt,font=\footnotesize] (combined) at (13.45,1.34)
    {\textbf{Richardson combination}\\[-1pt]
     {\scriptsize MP: $\tfrac43\widehat h^{\MP}_{1,\delta}-\tfrac13\widehat h^{\MP}_{2,\delta}$,
      $H+\mathcal O(\delta^4)$}\\[-1pt]
     {\scriptsize SE: $2\widehat h^{\SE}_{1,\delta}-\widehat h^{\SE}_{2,\delta}$,
      $H+\mathcal O(\delta^2)$}};
  \draw[->,gray!75,line width=0.8pt] (fittwo.east) to[bend left=10] (combined.north west);
  \draw[->,gray!75,line width=0.8pt] (fitone.east) to[bend right=10] (combined.south west);
\end{tikzpicture}
\caption{
Two-level Richardson extrapolation using multistride pairs from the same observed trajectory. Strides~$1$ and~$2$ produce estimators targeting $\check H_\delta^\Int$ and $\check H_{2\delta}^\Int$, respectively. Their linear combination cancels the leading integrator-dependent correction.
}
\label{fig:richardson-data-reuse}
\end{figure}
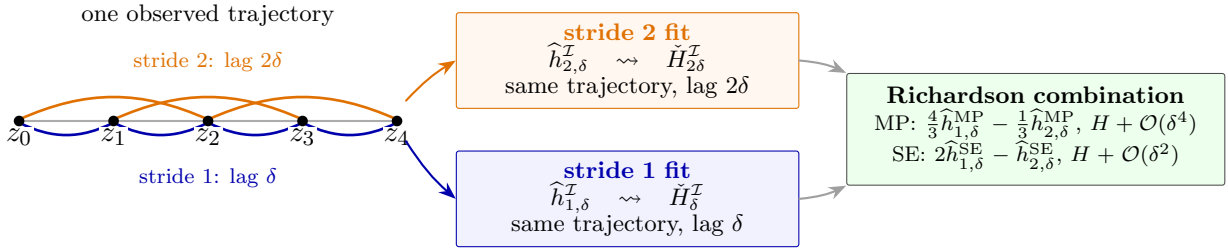

The resulting procedure is therefore a post-processing construction: the stride-specific estimators are trained separately and subsequently combined with Richardson weights. No additional trajectories or higher-order numerical integrators are required. 
In particular, Richardson extrapolation is applied here to the learned inverse modified Hamiltonians, rather than to numerical approximations of the trajectories themselves.

\subsection{Common multistride data and estimators}
\label{sec:rich-stride}

We first formulate the regression problem associated with each stride. The key observation is that a stride-$s$ pair is an exact-flow pair over the effective time step $s\delta$. Hence the scheme-adapted observation identity derived for the base step remains valid, with $\delta$ replaced by $s\delta$.

\paragraph{Stride-$s$ data and exact target.}
Consider the trajectories
\[
z_l^{(n)}=(q_l^{(n)},p_l^{(n)}), \qquad n=1,\ldots,N,\qquad l=0,\ldots,L,
\]
and fix an integer stride $s\in\{1,\ldots,Q\}$, with $Q\le L$.
For each $l=s,\ldots,L$, the pair $(z_{l-s}^{(n)},z_l^{(n)})$ spans the time interval $s\delta$. For $\Int\in\{\MP,\SE\}$, define the corresponding scheme-adapted anchor and discrete velocity by
\begin{align*}
x_{s,l}^{\Int,(n)}:=
\begin{cases}
\displaystyle
\frac12\bigl(z_{l-s}^{(n)}+z_l^{(n)}\bigr),
& \Int=\MP,\\[1mm]
\bigl(q_{l-s}^{(n)},p_l^{(n)}\bigr),
& \Int=\SE,
\end{cases}\qquad
v_{s,l}^{(n)}:=\frac{z_l^{(n)}-z_{l-s}^{(n)}}{s\delta}.
\end{align*}
Each trajectory contributes $L_s:=L-s+1$ stride-$s$ pairs, for a total of $NL_s$ observations.
Since $z_l^{(n)}=\varphi_{s\delta}^H(z_{l-s}^{(n)}),$ the exact inverse-target identities established for the base step apply with $\delta$ replaced by $s\delta$. In particular,
\begin{equation}\label{eq:stride-exact-target}
v_{s,l}^{(n)}=J\nabla \check H_{s\delta}^{\Int}
\bigl(x_{s,l}^{\Int,(n)}\bigr).
\end{equation}
Thus the stride-$s$ observations form an exact regression problem for
the inverse modified Hamiltonian $\check H_{s\delta}^{\Int}$, rather
than for the physical Hamiltonian $H$.

\paragraph{Stride-dependent population law.}
At the population level, define
\[
X_{s,l}^{\Int}(z_0):=
\begin{cases}
\displaystyle \frac12\left(\varphi_{(l-s)\delta}^H(z_0)+\varphi_{l\delta}^H(z_0)\right),
& \Int=\MP,\\[1mm]
\left(\pi_q\varphi_{(l-s)\delta}^H(z_0),\pi_p\varphi_{l\delta}^H(z_0)\right),
& \Int=\SE,
\end{cases}
\]
and let $\rho_s^{\Int}:=\frac1{L_s}\sum_{l=s}^L (X_{s,l}^{\Int})_\#\mu_0$ be the stride-averaged input law. For $s=1$, $X_{1,l}^{\Int}$ and $\rho_{1}^{\Int}$ reduce to the corresponding base-step anchor map and averaged input law introduced in \eqref{eq:population-anchor-map} and \eqref{eq:rho-generic}. Associated with $\rho_s^\Int$, define $A_s^\Int:\HK\to L^2(\rho_s^\Int;\mathbb R^{2d}),$ with $A_s^\Int h=J\nabla h,$ and denote $B_s^\Int:=(A_s^\Int)^*A_s^\Int.$

For the remainder of this section, we assume that $\operatorname{supp}\rho_s^\Int\subset\Omega, $ for all $s=1,\ldots,Q$.
This is the natural multistride counterpart of the support condition in Assumption~\ref{ass:domain}.
The restriction to integer strides is only used to reuse the same uniformly sampled trajectories; Richardson extrapolation itself applies to arbitrary step sizes.

\paragraph{Stride-$s$ regression problem.}
Given a stride-dependent regularization parameter $\lambda_s>0$, define
\begin{equation*}
\widehat h_{s,\lambda_s,NL_s}^{\Int}:= \arg\min_{h\in\HK}\left\{\frac{(s\delta)^2}{NL_s}\sum_{n=1}^N\sum_{l=s}^L\left|J\nabla h(x_{s,l}^{\Int,(n)})-v_{s,l}^{(n)}\right|^2+\lambda_s\|h\|_{\HK}^2\right\}.
\end{equation*}
By \eqref{eq:stride-exact-target}, this is precisely the same regularized regression problem as in the base-step analysis, with $\delta$ replaced by $s\delta$ and target
$\check H_{s\delta}^{\Int}$.

Let $A_{s,NL_s}^{\Int}$ and $v_{s,NL_s}^{\Int}$ denote the empirical sampling operator and observation vector associated with the stride-$s$ data, defined analogously to their base-step counterparts in \eqref{eq:operator-AN} and \eqref{eq:observation-vector}, and set $B_{s,NL_s}^{\Int}:=(A_{s,NL_s}^{\Int})^*A_{s,NL_s}^{\Int}.$
Then, arguing as in Proposition~\ref{prop:operator-representations}, the stride-$s$ estimator admits the operator representation
\begin{equation*}
\widehat h_{s,\lambda_s,NL_s}^{\Int}=\frac{1}{\sqrt{NL_s}}\left( B_{s,NL_s}^{\Int}+\eta_s I \right)^{-1}(A_{s,NL_s}^{\Int})^*v_{s,NL_s}^{\Int},
\end{equation*}
where $\eta_s:=\frac{\lambda_s}{(s\delta)^2}.$ Thus $\eta_s$ is the effective regularization parameter associated with stride $s$. 

For the remainder of this section, we use the stride-balanced regularization
\begin{equation}\label{eq:stride-balanced-regularization}
\lambda_s=(s\delta)^2\eta, \qquad s=1,\ldots,Q,
\end{equation}
so that all stride levels have the same effective regularization parameter $\eta$. Since the focus of this section is the dependence on the effective step size $s\delta$, we suppress the fixed statistical parameters and write
\begin{equation}\label{eq:stride-est-short}
\widehat h_{s,\delta}^{\Int}:=\widehat h_{s,\lambda_s,NL_s}^{\Int}.
\end{equation}
Thus, for each stride $s$, the estimator $\widehat h_{s,\delta}^{\Int}$ targets the inverse modified Hamiltonian $\check H_{s\delta}^{\Int}$.

\subsection{Richardson extrapolation of the stride estimators}

We now combine the stride estimators introduced in Section~\ref{sec:rich-stride}. The construction exploits the fact that the inverse modified Hamiltonians at different strides share the same small-step expansion, while the dependence of each correction term on the stride is known explicitly.
Recall that
\[
\check H_{s\delta}^{\Int}=H+\sum_{j\ge 1}s^{r_{\Int}j}\delta^{r_{\Int}j}\check H_{r_{\Int}j}^{\Int},\qquad r_{\MP}=2,\qquad r_{\SE}=1,
\]
where for implicit midpoint the odd coefficients vanish by symmetry.
This suggests choosing weights that preserve the physical Hamiltonian while cancelling the first $Q-1$ inverse-modified corrections.

\begin{definition}[Richardson coefficients]
\label{def:rich-coeff}
Fix $\Int\in\{\MP,\SE\}$ and $Q\ge2$. The {\bfi Richardson coefficients} $\alpha_1^{\Int},\ldots,\alpha_Q^{\Int}\in\RR$ are the unique solution of the Vandermonde system
\begin{equation}\label{eq:vandermonde}
\sum_{s=1}^{Q}\alpha_s^{\Int}=1,\qquad \sum_{s=1}^{Q}\alpha_s^{\Int}s^{r_{\Int}j}=0, \qquad j=1,\ldots,Q-1.
\end{equation}
\end{definition}
We also define the {\bfi Richardson amplification factor} $\mathcal A_Q^\Int:=\sum_{s=1}^Q |\alpha_s^\Int|.$

\begin{example}[Two-level Richardson coefficients]
For $Q=2$, the Richardson coefficients are 
\[
(\alpha_1^{\MP},\alpha_2^{\MP})=\left(\frac43,-\frac13\right),
\qquad
(\alpha_1^{\SE},\alpha_2^{\SE})=(2,-1).
\]
Hence $\mathcal A_2^{\MP}=\frac53, \mathcal A_2^{\SE}=3,$ and the deterministic bias is raised from order $2$ to $4$ for implicit midpoint and from order $1$ to $2$ for symplectic Euler.
\end{example}

For $\Int\in\{\MP,\SE\}$, define the $Q$-level {\bfi Richardson estimator} by
\begin{equation}\label{eq:richardson-estimator}
\widehat h_{\delta}^{\Int,(Q)}:= \sum_{s=1}^{Q}\alpha_s^{\Int}\widehat h_{s,\delta}^{\Int},
\end{equation}
where $\widehat h_{s,\delta}^{\Int}$ is the stride-$s$ estimator defined in \eqref{eq:stride-est-short}. The corresponding Richardson combination of the exact inverse-modified targets is
\begin{equation}\label{eq:richardson-target}
\check H_{\delta}^{\Int,(Q)}:=\sum_{s=1}^{Q}\alpha_s^\Int\check H_{s\delta}^\Int.
\end{equation}
The estimator \eqref{eq:richardson-estimator} is therefore obtained by linear post-processing of separately fitted stride estimators, while \eqref{eq:richardson-target} isolates the deterministic extrapolation effect.

The following proposition quantifies the resulting cancellation of the
inverse-target bias.

\begin{proposition}[Cancellation of the leading inverse-target corrections]
\label{prop:richardson-target-bias}
Fix $\Int\in\{\MP,\SE\}$ and $Q\ge2$, and let $\alpha_1^\Int,\ldots,\alpha_Q^\Int$ be the Richardson coefficients from Definition~\ref{def:rich-coeff}.

\begin{enumerate}[label=(\roman*),leftmargin=*]
\item \emph{(Algebraic cancellation.)}
Suppose that Assumption~\ref{ass:analytic-setting} holds and that $Q\delta<\delta_*^\Int$. Then
\begin{equation}\label{eq:richardson-target-expansion}
\check H_{\delta}^{\Int,(Q)}-H=\sum_{j\ge Q}\left(\sum_{s=1}^Q\alpha_s^\Int s^{r_\Int j}\right)\delta^{r_\Int j}\check H_{r_\Int j}^\Int.
\end{equation}
In particular, the first $Q-1$ inverse-modified corrections are
cancelled exactly.

\item \emph{(Quantitative bound.)}
Suppose, in addition, that Assumption \ref{ass:domain} holds. Then there exist $\delta_{Q,\Int}>0$ and a constant $C_{Q,\Int}>0$, independent of $\delta$, such that
\begin{equation*}
\left\|\check H_{\delta}^{\Int,(Q)}-H\right\|_{\HK}\le C_{Q,\Int}\delta^{r_\Int Q}, \qquad 0<\delta<\delta_{Q,\Int}.
\end{equation*}
\end{enumerate}
\end{proposition}

\begin{proof}
For (i), substituting the inverse-modified expansion into \eqref{eq:richardson-target} gives
\begin{align*}
\check H_{\delta}^{\Int,(Q)}-H =\sum_{s=1}^Q \alpha_s^\Int \bigl(\check H_{s\delta}^\Int-H\bigr)=\sum_{j\ge1}\left(\sum_{s=1}^Q\alpha_s^\Int s^{r_\Int j}\right)\delta^{r_\Int j}\check H_{r_\Int j}^\Int.
\end{align*}
By \eqref{eq:vandermonde}, the inner sum vanishes for
$j=1,\ldots,Q-1$, which proves
\eqref{eq:richardson-target-expansion}.

For (ii), the native-space coefficient bounds established in Appendix~\ref{sec:bae} imply that there exist $M_\Int,R_\Int>0$ such that $\|\check H_{r_\Int j}^\Int\|_{\HK} \le M_\Int R_\Int^{-r_\Int j}, $ for all $ j\ge1.$
Moreover, $\left|\sum_{s=1}^Q\alpha_s^\Int s^{r_\Int j}\right|\le\mathcal A_Q^\Int Q^{r_\Int j}.$
Hence
\[
\left\|\check H_{\delta}^{\Int,(Q)}-H\right\|_{\HK}\le\mathcal A_Q^\Int M_\Int\sum_{j\ge Q}\left(\frac{Q\delta}{R_\Int}\right)^{r_\Int j}.
\]
For $\delta>0$ sufficiently small, the series is geometric, and therefore the result follows.
\end{proof}

We now return to the learned estimator. Combining \eqref{eq:richardson-estimator} and \eqref{eq:richardson-target} gives the exact decomposition
\begin{equation}\label{eq:richardson-error-decomposition}
\widehat h_{\delta}^{\Int,(Q)}-H=\sum_{s=1}^Q\alpha_s^\Int\left(\widehat h_{s,\delta}^\Int-\check H_{s\delta}^\Int\right)+\left(\check H_{\delta}^{\Int,(Q)}-H\right).
\end{equation}
The first term contains the estimation and approximation errors of the
individual stride estimators, whereas the second is the remaining
integrator-induced error of the Richardson-combined target controlled by Proposition~\ref{prop:richardson-target-bias}.
Recall that the Richardson amplification factor is $\mathcal A_Q^\Int:=\sum_{s=1}^Q |\alpha_s^\Int|.$ Consequently,
\begin{equation}\label{eq:richardson-basic-error-bound}
\begin{aligned}
\left\|\widehat h_{\delta}^{\Int,(Q)}-H\right\|_{\HK}
\le\mathcal A_Q^\Int\max_{1\le s\le Q}\left\|\widehat h_{s,\delta}^\Int-\check H_{s\delta}^\Int\right\|_{\HK}+C_{Q,\Int}\delta^{r_\Int Q}.
\end{aligned}
\end{equation}

Thus Richardson extrapolation raises the order of the integrator-induced
error, at the price of amplifying the stride-wise learning errors by the
factor $\mathcal A_Q^\Int$. Although the stride estimators are statistically dependent because they reuse the same observed trajectories, the deterministic decomposition \eqref{eq:richardson-error-decomposition} and the bound \eqref{eq:richardson-basic-error-bound} do not require independence across strides. It remains to control the first term using the single-stride learning estimates.

\subsection{Richardson reconstruction error}

It remains to control the single-stride reconstruction errors in \eqref{eq:richardson-basic-error-bound}. Since the stride-$s$ problem has exactly the same form as the base-step problem, with effective step size $s\delta$, the estimation and approximation bounds of Section~\ref{sec:estimation-regularization} apply separately at each stride. 

For each $s=1,\ldots,Q$, define the stride-$s$ population
regularized minimizer $h_{s,\eta}^{*,\Int}:=(B_s^\Int+\eta I)^{-1}B_s^\Int\check H_{s\delta}^\Int.$
Set $M_{Q,\delta}^\Int := \max_{1\le s\le Q}\|\check H_{s\delta}^\Int\|_{\HK}$ and for a common exponent $\gamma\in(0,1]$, define
\begin{equation*}
S_{Q,\delta}^\Int:=\max_{1\le s\le Q}\inf\left\{\|\psi\|_{\HK}\mid\psi\in\HK,\quad \check H_{s\delta}^\Int=(B_s^\Int)^\gamma\psi
\right\}.
\end{equation*}

\begin{theorem}[Richardson reconstruction error]
\label{thm:rich-decomp}
Let $\Int\in\{\MP,\SE\}$ and $Q\ge2$. Suppose that Assumptions~\ref{ass:analytic-setting} and~\ref{ass:domain} hold, and let $\delta>0$ be sufficiently small so that Proposition~\ref{prop:richardson-target-bias} applies.
Assume that the source condition in Assumption~\ref{ass:source} holds at each stride with a common exponent $\gamma\in(0,1]$; that is, for each $s=1,\ldots,Q$, there exists $\psi_{s,\delta}^\Int\in\HK$ such that
\begin{equation*}
\check H_{s\delta}^\Int=(B_s^\Int)^\gamma\psi_{s,\delta}^\Int.
\end{equation*}
Then under the stride-balanced regularization \eqref{eq:stride-balanced-regularization}, for every $0<\varepsilon<1$, with probability at least $1-\varepsilon$,
\begin{equation}
\label{eq:richardson-reconstruction-bound}
\left\|\widehat h_{\delta}^{\Int,(Q)}-H\right\|_{\HK}\le C\mathcal A_Q^\Int\left(\frac{2d\kappa^2M_{Q,\delta}^\Int}{\eta}\sqrt{\frac{\log(4Q/\varepsilon)}{N}}+\eta^\gamma S_{Q,\delta}^\Int\right)+C_{Q,\Int}\delta^{r_\Int Q}.
\end{equation}
\end{theorem}

\begin{proof}
For each $s=1,\ldots,Q$, the stride-$s$ regression problem has the same form as the base-step problem, with step size $s\delta$, population covariance operator $B_s^\Int$, and inverse target $\check H_{s\delta}^\Int$. Since $\lambda_s=(s\delta)^2\eta$, the effective regularization parameter is the same value $\eta$ at every stride.

Applying Proposition~\ref{thm:estimation} to the stride-$s$ problem with failure probability $\varepsilon/Q$ gives
\[
\left\|\widehat h_{s,\delta}^\Int-h_{s,\eta}^{*,\Int}\right\|_{\HK}\le C\frac{2d\kappa^2}{\eta}\sqrt{\frac{\log(4Q/\varepsilon)}{N}}\left\|\check H_{s\delta}^\Int\right\|_{\HK}.
\]
A union bound therefore implies that these estimates hold simultaneously for all $s=1,\ldots,Q$ with probability at least $1-\varepsilon$. 

By the stride source condition and the spectral argument used in \eqref{eq:regularization-bound}, for every $\psi\in\HK$ satisfying $\check H_{s\delta}^\Int=(B_s^\Int)^\gamma\psi,$ we have $\left\|h_{s,\eta}^{*,\Int}-\check H_{s\delta}^\Int\right\|_{\HK}\le \eta^\gamma\|\psi\|_{\HK}.$
Taking the infimum over all such $\psi$ yields
\[
\left\|h_{s,\eta}^{*,\Int}-\check H_{s\delta}^\Int\right\|_{\HK}\le\eta^\gamma\inf\left\{\|\psi\|_{\HK}:\check H_{s\delta}^\Int=(B_s^\Int)^\gamma\psi\right\}\le \eta^\gamma S_{Q,\delta}^\Int.
\]
Hence, on the same event,
\begin{align*}
\left\|\widehat h_{s,\delta}^\Int-\check H_{s\delta}^\Int\right\|_{\HK}
&\le\left\|\widehat h_{s,\delta}^\Int-h_{s,\eta}^{*,\Int}\right\|_{\HK}+\left\|h_{s,\eta}^{*,\Int}-\check H_{s\delta}^\Int\right\|_{\HK}\\
&\le C\frac{2d\kappa^2}{\eta}\sqrt{\frac{\log(4Q/\varepsilon)}{N}}\left\|\check H_{s\delta}^\Int\right\|_{\HK}+\eta^\gamma S_{Q,\delta}^\Int.
\end{align*}
Taking the maximum over $s=1,\ldots,Q$ yields
\[
\max_{1\le s\le Q}\left\|\widehat h_{s,\delta}^\Int-\check H_{s\delta}^\Int\right\|_{\HK}\le C\frac{2d\kappa^2}{\eta}\sqrt{\frac{\log(4Q/\varepsilon)}{N}}M_{Q,\delta}^\Int+\eta^\gamma S_{Q,\delta}^\Int.
\]
Substituting this estimate into \eqref{eq:richardson-basic-error-bound} proves \eqref{eq:richardson-reconstruction-bound}.
\end{proof}

\begin{remark}[Integrator-error reduction versus statistical amplification]
Theorem~\ref{thm:rich-decomp} makes explicit the trade-off introduced by
Richardson extrapolation. Increasing $Q$ reduces the integrator-induced error
from $\mathcal O(\delta^{r_\Int})$ to
$\mathcal O(\delta^{r_\Int Q})$, but does not improve the statistical
reconstruction rate. Instead, the estimation and approximation errors of the
individual stride estimators are amplified by at most the Richardson factor
$\mathcal A_Q^\Int$.

Thus Richardson extrapolation is most beneficial when the integrator-induced
error is comparable to or larger than the learning error. Once the estimation
and approximation errors dominate, increasing the extrapolation level provides
no corresponding improvement and may become counterproductive because of the
growth of $\mathcal A_Q^\Int$.
\end{remark}

\begin{corollary}[Balanced joint convergence with Richardson extrapolation]
Let $\Int\in\{\MP,\SE\}$ and $Q\ge2$, and suppose  that the hypotheses of Theorem~\ref{thm:rich-decomp} hold for all sufficiently small $\delta$, with the same exponent $\gamma\in(0,1]$. 
Assume further that the source norms remain uniformly bounded as $\delta\to0$, namely, 
\begin{equation}\label{eq:uniform-richardson-bounds}
\sup_{0<\delta\le\delta_0}
S_{Q,\delta}^\Int<\infty.
\end{equation}
Let $\eta_N\asymp N^{-1/(2\gamma+2)}, $ $ \delta_N \asymp N^{-\gamma/[r_\Int Q(2\gamma+2)]},$
and choose the stride-balanced regularization $\lambda_{s,N}:=(s\delta_N)^2\eta_N,$ for $s=1,\ldots,Q.$
Then
\begin{equation}\label{eq:balanced-richardson-rate}
\left\|
\widehat h_{\delta_N}^{\Int,(Q)}-H
\right\|_{\HK}
=
\mathcal O_{\mathbb P}
\left(
\mathcal A_Q^\Int
N^{-\gamma/(2\gamma+2)}
\right).
\end{equation}
\end{corollary}

\begin{proof}
By the native-space target-bias estimate in Appendix \ref{sec:bae}, for $s=1,\ldots,Q$,
\[
\|\check H_{s\delta}^\Int\|_{\HK}\le\|H\|_{\HK}+C_\Int(s\delta)^{r_\Int}\le\|H\|_{\HK}+ C_\Int Q^{r_\Int}\delta^{r_\Int}.
\]
Hence $M_{Q,\delta}^\Int$ is uniformly bounded for sufficiently small $\delta$.
Therefore, Theorem~\ref{thm:rich-decomp} and \eqref{eq:uniform-richardson-bounds} yield
\[
\left\|\widehat h_{\delta}^{\Int,(Q)}-H\right\|_{\HK}=\mathcal O_{\mathbb P}\left(\mathcal A_Q^\Int\left(\eta^{-1}N^{-1/2}+\eta^\gamma\right)+\delta^{r_\Int Q}\right).
\]
Balancing the estimation and approximation contributions gives $\eta_N\asymp N^{-1/(2\gamma+2)},$ and $\eta_N^\gamma\asymp N^{-\gamma/(2\gamma+2)}.$
Choosing $\delta_N$ so that $\delta_N^{r_\Int Q}\asymp N^{-\gamma/(2\gamma+2)}$ yields $\delta_N \asymp N^{-\gamma/[r_\Int Q(2\gamma+2)]}.$
Substituting these choices into the preceding bound proves \eqref{eq:balanced-richardson-rate}.
\end{proof}

Compared with Corollary~\ref{cor:balanced-joint-rate}, Richardson extrapolation preserves the same statistical reconstruction rate while allowing a slower refinement of the observation step. More precisely, the required step-size scaling changes from
\[
\delta_N \asymp N^{-\gamma/[r_\Int(2\gamma+2)]} \quad \text{to}\quad \delta_N \asymp N^{-\gamma/[r_\Int Q(2\gamma+2)]}.
\]
Thus higher-level extrapolation reduces the temporal-resolution
requirement needed to retain the same overall reconstruction rate.

\section{Numerical experiments}
\label{sec:experiments}

In this section, we numerically investigate the reconstruction properties predicted by the preceding analysis. We first use the H\'enon--Heiles system as a controlled benchmark to examine the three main components of the reconstruction error: the estimation error, the approximation error, and the integrator-induced error. On the same benchmark, we then test multistride Richardson correction, its interaction with observational noise, and the effect of trajectory dependence under a fixed observation budget. We finally assess the reconstruction capability of the implicit midpoint estimator on three additional Hamiltonian systems: the double pendulum, the Frenkel--Kontorova model, and a highly non-convex Hamiltonian.

\subsection{Experimental setting}

Throughout the experiments, we use the Mat\'ern-$5/2$ kernel
\[
K(x,y)=\left(1+\frac{\sqrt{5}r}{\ell}+\frac{5r^2}{3\ell^2}\right)\exp\left(-\frac{\sqrt{5}r}{\ell}\right),\qquad r=\|x-y\|,
\]
where $\ell>0$ denotes the kernel length scale. The Mat\'ern-$5/2$ kernel satisfies the smoothness requirements of Assumption~\ref{ass:domain} as stated in Example \ref{Matern kernels}.

Given exact-flow endpoint pairs $(z_i^-,z_i^+)$ separated by a time step $\delta$, the implicit midpoint and symplectic Euler observations are constructed as
\[
x_i^{\MP}=\frac{z_i^-+z_i^+}{2},\qquad x_i^{\SE}=(q_i^-,p_i^+),\qquad v_i=\frac{z_i^+-z_i^-}{\delta}.
\]
We use the closed-form estimator derived in Section~\ref{sec:learning}, with effective regularization parameter $\eta=\lambda/\delta^2$. We consider $N$ independent trajectories with $L$ observations per trajectory, so that the total number of observations is $M=NL$. We generate reference trajectories with DOP853 using $\texttt{rtol}=10^{-12}$ and $\texttt{atol}=10^{-14}$, and construct the midpoint and symplectic Euler observations from the resulting endpoint pairs.

The main theoretical experiments are performed on the H\'enon--Heiles Hamiltonian
\[
H(q,p)=\frac12(p_1^2+p_2^2+q_1^2+q_2^2)+q_1^2q_2-\frac13q_2^3.
\]
We sample the training initial conditions uniformly from a bounded box and retain those with energy below $E_{\max}=0.14$, which lies below the escape threshold $1/6$. We evaluate the reconstruction error on the two-dimensional slice with fixed momentum $p=(0.15,-0.10)$. 

Since the Hamiltonian is identifiable only up to an additive constant, we first align the reconstruction by $c^\star=\frac{1}{N_{\rm test}}\sum_{j=1}^{N_{\rm test}}\left(H(z_j^{\rm test})-\widehat H(z_j^{\rm test})\right),$
and report the relative error
\[
S_{\operatorname{re}}(\widehat H,H)
=
\frac{
\left(\sum_j |\widehat H(z_j^{\rm test})+c^\star-H(z_j^{\rm test})|^2\right)^{1/2}
}{
\left(\sum_j |H(z_j^{\rm test})|^2\right)^{1/2}
}.
\]

\subsection{Fixed-step recovery and integrator-induced error}

We first examine how the reconstruction error varies with the time step. This experiment isolates the scheme-dependent integrator-induced error predicted by the inverse-target analysis. We take $N=1000$, $L=1$, $\eta=10^{-9}$ and $\ell=3$, and vary
$
\delta\in\{0.08,0.10,0.12,0.15,0.18,0.22,0.26,0.30\}$.
For each $\delta$, we evolve the same family of independently sampled initial conditions under the exact Hamiltonian flow, and construct midpoint and symplectic Euler observations from the resulting endpoint pairs.

Figure~\ref{fig:HH-error-decomposition} (a) shows the relative reconstruction error as a function of $\delta$. The implicit midpoint estimator exhibits the expected second-order behavior, while the symplectic Euler estimator exhibits first-order behavior. Numerically, the fitted slopes are approximately $2.03$ and $1.01$, respectively. This agrees with the inverse-target expansions
\[
\check H_\delta^{\MP}=H+\delta^2 \check H_2^{\MP}+O(\delta^4),\qquad
\check H_\delta^{\SE}=H+\delta \check H_1^{\SE}+O(\delta^2),
\]
and confirms that, once the estimation and approximation errors are sufficiently small, the reconstruction error is controlled by the integrator-induced error.

\begin{figure}[t]
    \centering
    \begin{minipage}[t]{0.48\textwidth}
        \centering
        \includegraphics[width=\textwidth]{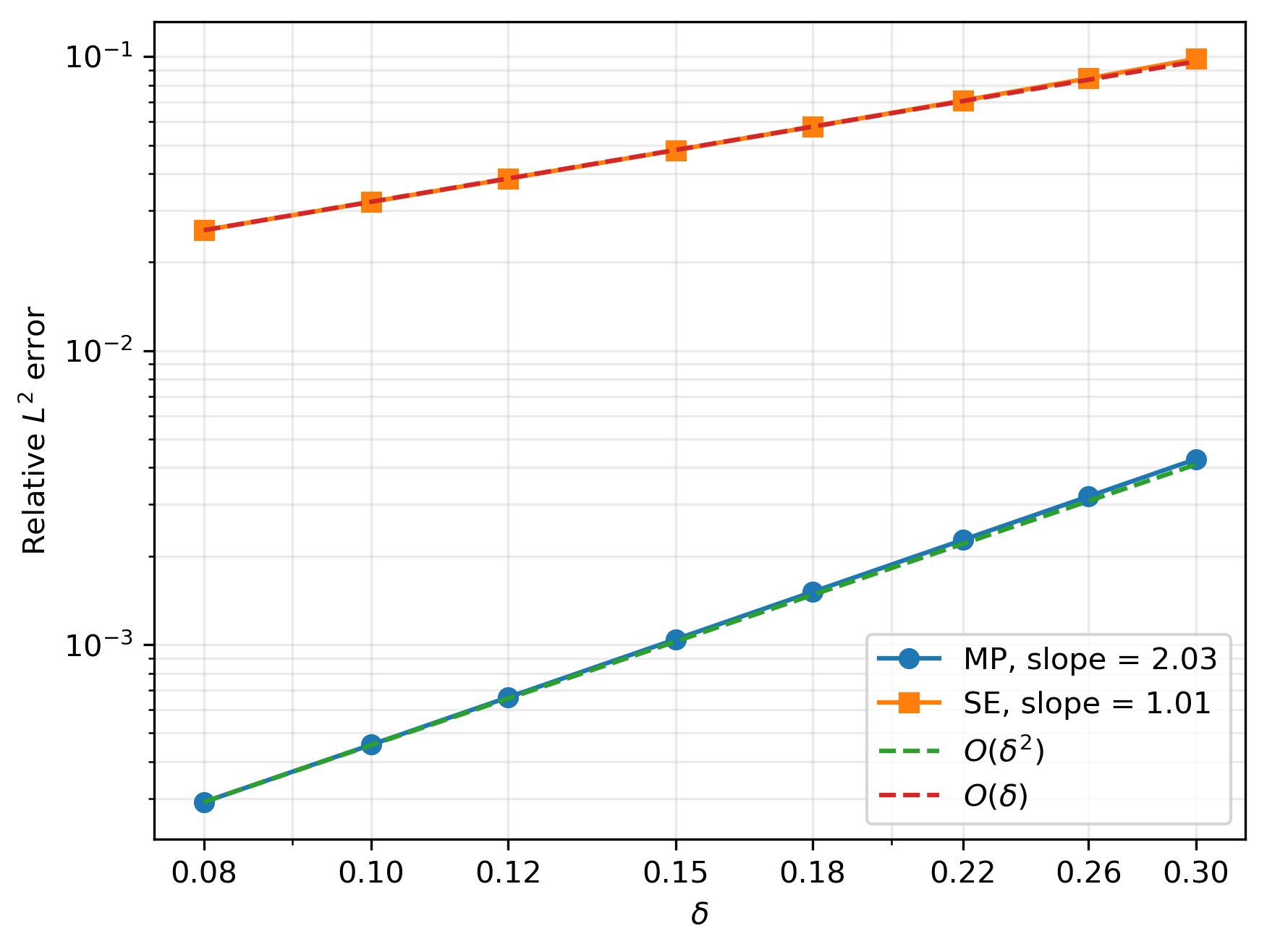}

        \small (a) Fixed $N$, varying $\delta$.
    \end{minipage}
    \hfill
    \begin{minipage}[t]{0.48\textwidth}
        \centering
        \includegraphics[width=\textwidth]{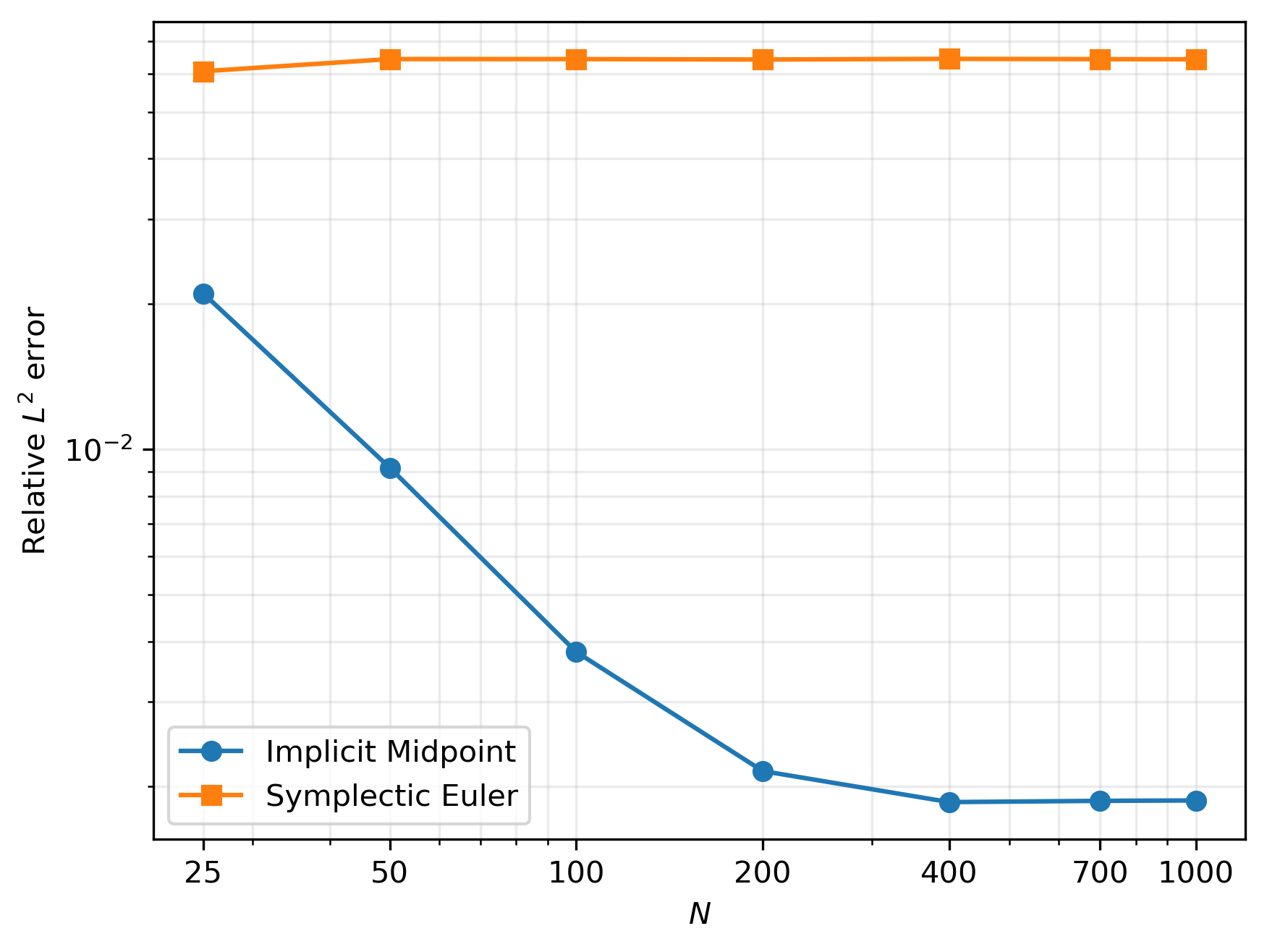}

        \small (b) Fixed $\delta$, varying $N$.
    \end{minipage}
    \caption{Error behavior for the H\'enon--Heiles system. 
    (a) The midpoint and symplectic Euler estimators exhibit the predicted second- and first-order integrator-induced errors as the time step varies. 
    (b) At fixed step size, increasing the number of independent trajectories reduces the estimation and approximation errors until the fixed-step integrator-induced error becomes dominant.}
    \label{fig:HH-error-decomposition}
\end{figure}

To separate the statistical component from the fixed-step error, we next fix $\delta=0.2, L=1$ and vary the number of independent trajectories. We use
$N\in\{25,50,100,200,400,700,1000\}$
and decrease the effective regularization parameter according to
\[
\eta_N=\eta_{\rm ref}\left(\frac{N_{\rm ref}}{N}\right)^{1/4},\qquad
\eta_{\rm ref}=10^{-9},\qquad N_{\rm ref}=1000.
\]
This choice allows the approximation contribution associated with the regularized estimator to decrease with the sample size rather than remain artificially fixed.

As shown in Figure~\ref{fig:HH-error-decomposition} (b), the midpoint error initially decreases rapidly as $N$ increases and subsequently saturates at a nonzero level. This behavior is consistent with the theoretical decomposition
\[
S_{\operatorname{re}}(\widehat H,H)
\lesssim
\text{estimation error}
+
\text{approximation error}
+
\text{integrator-induced error}.
\]
For moderate $N$, the first two terms remain visible, whereas for large $N$ the fixed-step integrator-induced error becomes dominant. The symplectic Euler curve is nearly flat over the tested range because its larger first-order integrator-induced error already dominates at relatively small sample sizes.

We further verify that the learned Hamiltonian follows the scheme-dependent inverse target itself, rather than merely exhibiting the correct asymptotic rate. For the H\'enon--Heiles potential
\[
V(q_1,q_2)=\frac12(q_1^2+q_2^2)+q_1^2q_2-\frac13q_2^3,
\]
the leading midpoint and symplectic Euler corrections are
\[
C_{\MP}(q,p)=\frac1{24}\left[p^\top\nabla^2V(q)p+\|\nabla V(q)\|^2\right],\qquad
C_{\SE}(q,p)=\frac12p^\top\nabla V(q).
\]
Hence the leading inverse-target errors are $\delta^2C_{\MP}$ and $\delta C_{\SE}$, respectively. Figure~\ref{fig:HH-target} compares these quantities directly with the learned errors $\widehat H_\delta-H$, after removing the irrelevant additive constant. The close agreement of the level sets numerically confirms that the estimator tracks the scheme-dependent inverse Hamiltonian.

\begin{figure}[t]
    \centering
    \includegraphics[width=0.82\textwidth]{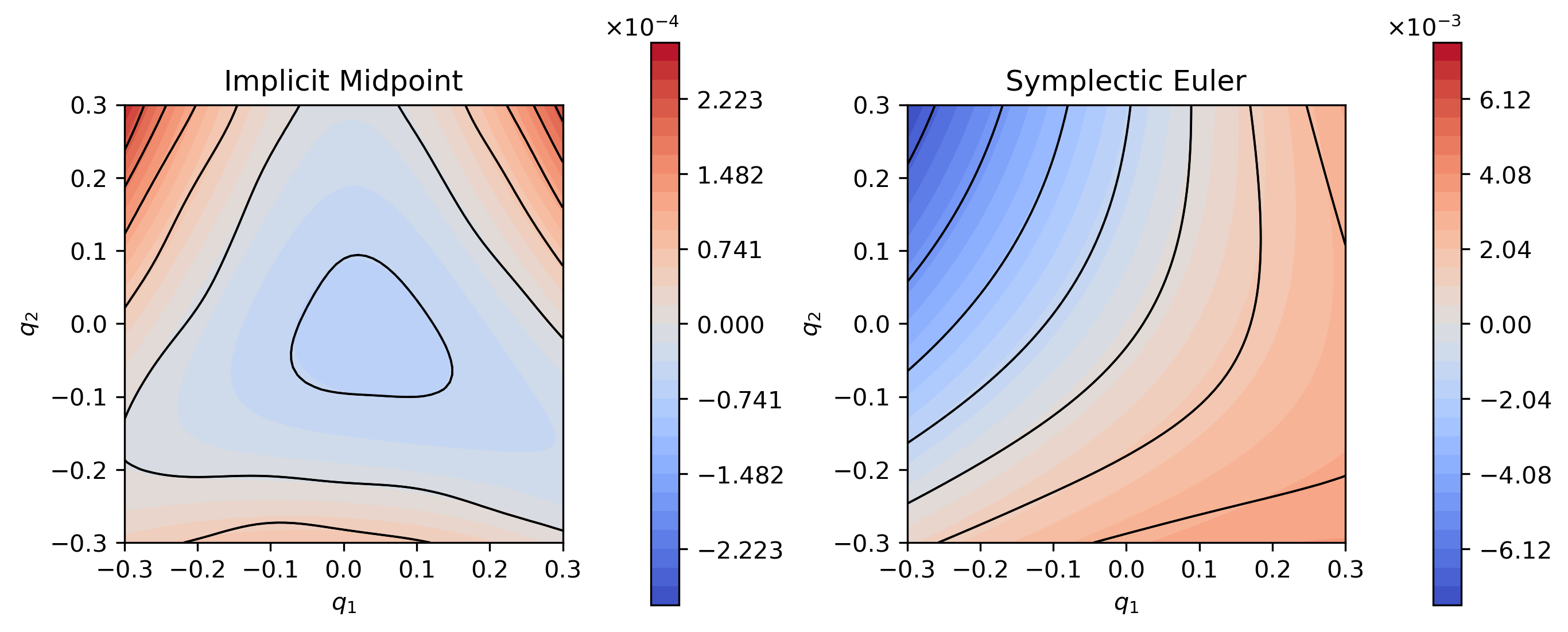}
    \caption{Leading inverse-target error and learned Hamiltonian error for the H\'enon--Heiles system. The colored fields represent the leading-order corrections $\delta^2C_{\MP}$ and $\delta C_{\SE}$, while the black contours represent the corresponding learned errors. The agreement of the level sets confirms that the learned Hamiltonian tracks the scheme-dependent inverse target.}
    \label{fig:HH-target}
\end{figure}

\subsection{Multistride Richardson correction and noise sensitivity}

We now use the H\'enon--Heiles system to investigate whether combining estimators constructed at several temporal strides can reduce scheme-dependent, integrator-induced error. For a common base step $\delta$, we use endpoint pairs separated by $s\delta$, with $s=1,2,3$, and keep the effective regularization parameter fixed across all strides $\eta=\frac{\lambda_s}{(s\delta)^2}$.
We take $N=1000$, $L=1$, $\eta=10^{-9}$, and $\ell=3$, and vary the base step over
$
\delta\in\{0.03,0.04,0.05,0.06,0.08,0.10,0.12,0.15,0.18,0.22,0.26,0.30\}$.
We use the same initial conditions for all strides and all values of $\delta$.

For midpoint, the $Q=2$ and $Q=3$ Richardson estimators are
\[
\widehat H_{\MP}^{(2)}
=
\frac43\widehat H_\delta^{\MP}
-\frac13\widehat H_{2\delta}^{\MP},
\qquad
\widehat H_{\MP}^{(3)}
=
\frac32\widehat H_\delta^{\MP}
-\frac35\widehat H_{2\delta}^{\MP}
+\frac1{10}\widehat H_{3\delta}^{\MP},
\]
whereas for symplectic Euler we use
\[
\widehat H_{\SE}^{(2)}
=
2\widehat H_\delta^{\SE}
-\widehat H_{2\delta}^{\SE},
\qquad
\widehat H_{\SE}^{(3)}
=
3\widehat H_\delta^{\SE}
-3\widehat H_{2\delta}^{\SE}
+\widehat H_{3\delta}^{\SE}.
\]
The corresponding integrator-induced error is formally reduced from $O(\delta^2)$ to $O(\delta^4)$ and $O(\delta^6)$ for midpoint, and from $O(\delta)$ to $O(\delta^2)$ and $O(\delta^3)$ for symplectic Euler.

Figure~\ref{fig:HH-richardson} compares the baseline estimator with the $Q=2$ and $Q=3$ corrections. For symplectic Euler, the improvement is substantial over a wide range of step sizes because the baseline first-order integrator-induced error is relatively large. For midpoint, the benefit of the corrections is most visible at moderate step sizes. At small $\delta$, the corrected midpoint estimators reach a common error floor, indicating that the estimation and approximation errors have become dominant. Hence increasing the Richardson order yields little additional improvement once the integrator-induced error is no longer the leading contribution.

\begin{figure}[t]
    \centering
    \includegraphics[width=0.90\textwidth]{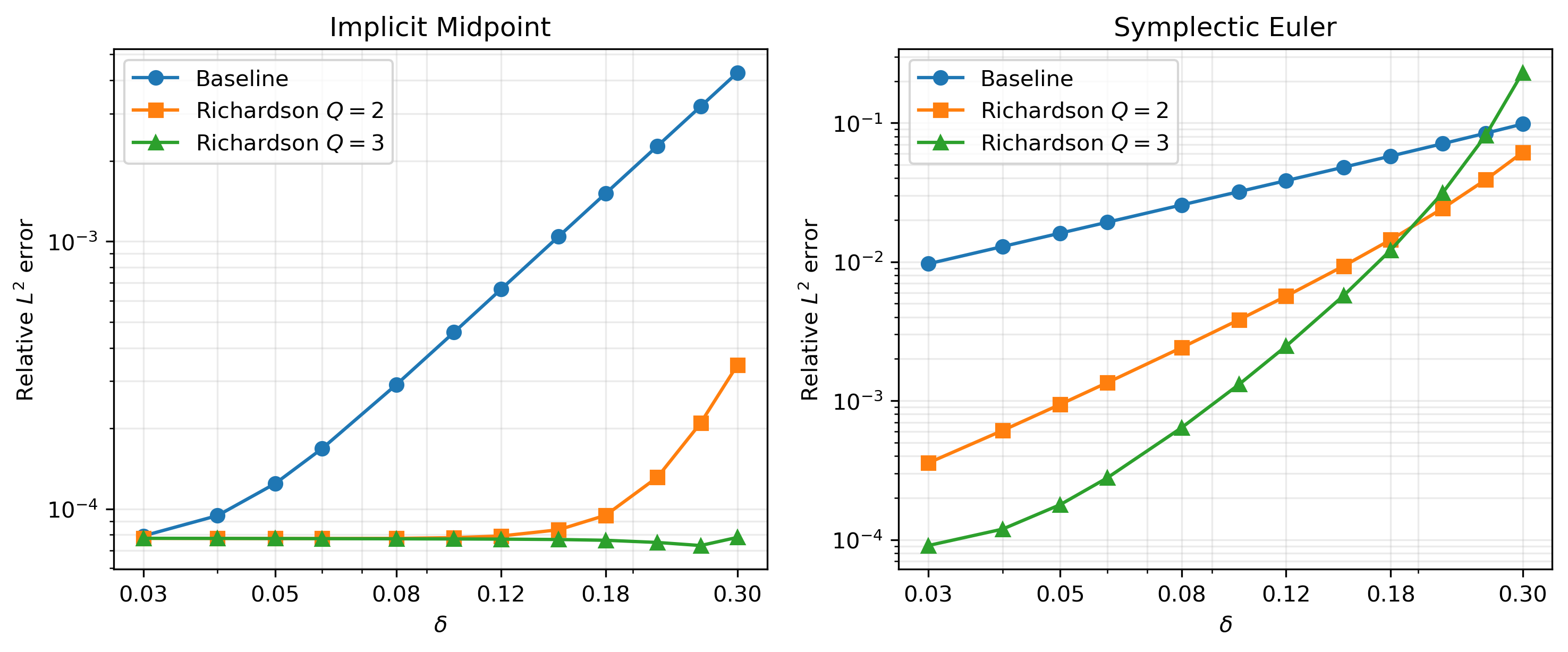}
    \caption{Multistride Richardson correction for the H\'enon--Heiles system. We compare the baseline, $Q=2$, and $Q=3$ estimators as the base step $\delta$ varies. Higher-order correction substantially reduces the symplectic Euler integrator-induced error, whereas the corrected midpoint estimators reach an estimation--approximation error floor at small step sizes.}
    \label{fig:HH-richardson}
\end{figure}

We next examine the interaction between Richardson correction and noisy trajectory observations. Following the noise model in Section~\ref{sec:deterministic-perturbations}, we perturb the observed trajectory states according to
$
\widetilde z_i=z_i+e_i$,
where the components of $e_i$ are independent Gaussian random variables. We parameterize the noise level relatively by setting
\[
e_i=a\,z_{\rm rms}\,\xi_i,\qquad
\xi_i\sim\mathcal N(0,I),
\]
where $z_{\rm rms}$ is the root-mean-square magnitude per state component over the reference endpoint states and $a$ is the relative state-noise level. We then construct the midpoint and symplectic Euler anchors, along with the corresponding discrete velocities, directly from the noisy endpoint states. Thus the perturbation enters both the regression inputs and outputs, as in the noisy scheme-adapted observation model of Section~\ref{sec:deterministic-perturbations}.

For this experiment, we take $N=1000$, $L=1$, $\delta=0.1$, $\eta=10^{-9}$, and $\ell=3$, and consider relative state-noise levels
$
a\in\{0,0.001,0.002,0.005,0.01,0.02,0.05,0.10\}$. 
We average the reconstruction errors over ten independent noise realizations. For a given stride, midpoint and symplectic Euler use the same realization of the perturbed endpoint states, while independent realizations are used across different strides.

Figure~\ref{fig:HH-noise} shows the results. In the low-noise regime, Richardson correction reduces the integrator-induced error, most noticeably for symplectic Euler. As state noise increases, however, higher-order corrections become progressively more sensitive to perturbed observations. This behavior is consistent with the larger Richardson coefficients at higher extrapolation orders, which increase the sensitivity of the combined estimator to perturbations in the individual reconstructions.

For midpoint, the uncorrected estimator becomes preferable already at relatively low noise levels because its baseline integrator-induced error is small. For symplectic Euler, the larger first-order integrator-induced error leaves a wider low-noise regime where Richardson correction remains beneficial, while the $Q=3$ estimator loses this advantage earlier than the $Q=2$ estimator as noise increases. The experiment therefore illustrates the trade-off between reducing integrator-induced error and amplifying errors caused by noisy trajectory observations.

\begin{figure}[t]
    \centering
    \includegraphics[width=0.90\textwidth]{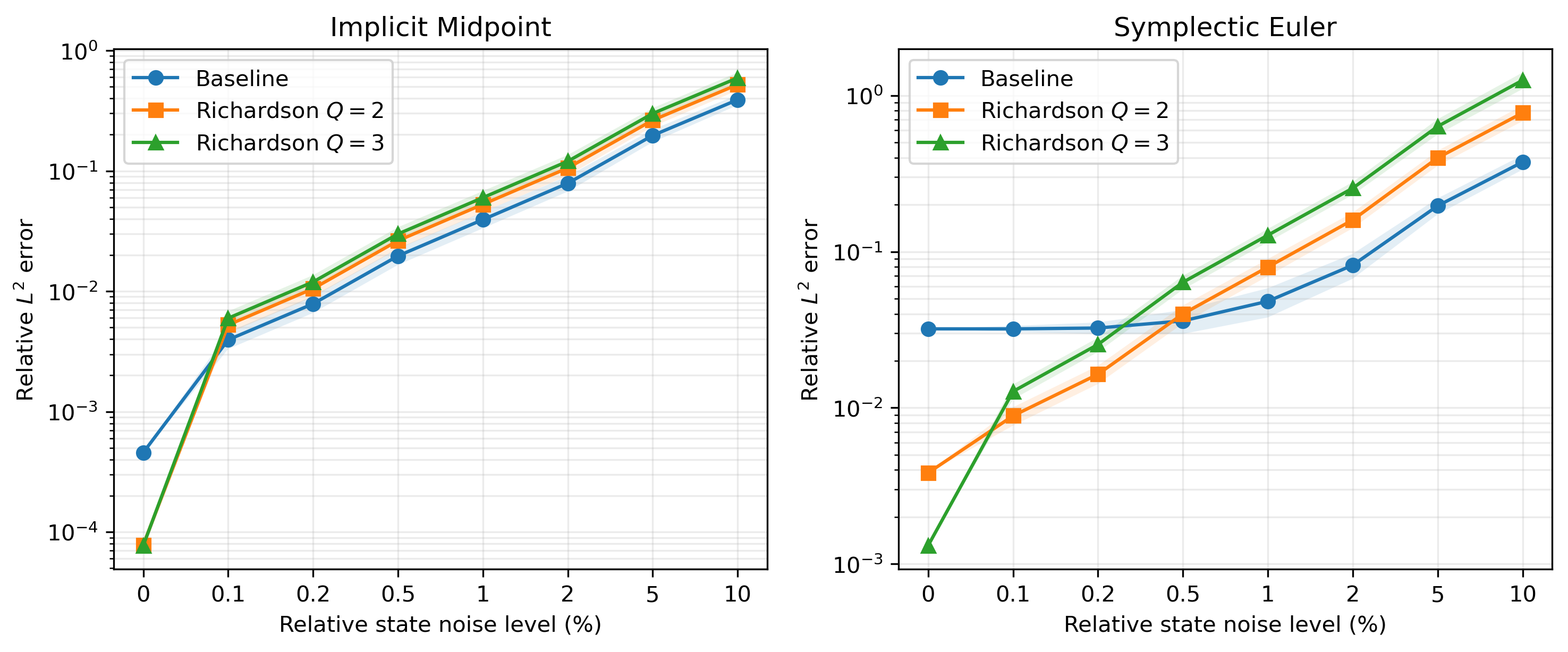}
    \caption{Effect of state observation noise on the baseline and Richardson-corrected estimators for the H\'enon--Heiles system. Shaded regions indicate one standard deviation over ten independent noise realizations. Richardson correction reduces the integrator-induced error in the low-noise regime, while higher extrapolation orders become increasingly sensitive to noisy trajectory observations.}
    \label{fig:HH-noise}
\end{figure}
\subsection{Trajectory dependence}

The previous experiments use one observation from each independent trajectory. We next investigate the distinct roles of the number $N$ of independent trajectories and the trajectory length $L$. Rather than simply increasing $L$, which also increases the total amount of data, we fix the observation budget $NL=400$ and compare $(N,L)\in\{(400,1),(200,2),(100,4),(50,8),(25,16)\}$.
For each initial condition, an exact trajectory is generated and sampled at intervals of length $\delta$, so that the $L$ observations along the same trajectory are dynamically dependent. We repeat the experiment for $\delta=0.01$, $0.05$, and $0.10$.

Figure~\ref{fig:HH-NL} shows the resulting errors. For midpoint, the reconstruction error increases systematically as $N$ decreases and $L$ increases, despite the fact that the total number of observations remains fixed. Thus, in the present regime, replacing independent trajectories by additional observations along the same trajectory does not provide an equivalent statistical benefit. The result is consistent with the interpretation that $N$ controls independent statistical averaging and global coverage of the sampling distribution, whereas increasing $L$ produces additional information along dynamically correlated orbit segments. For symplectic Euler, the allocation effect is substantially less visible because the first-order integrator-induced error dominates over the statistical variation associated with different choices of $(N,L)$.

\begin{figure}[t]
    \centering
    \includegraphics[width=\textwidth]{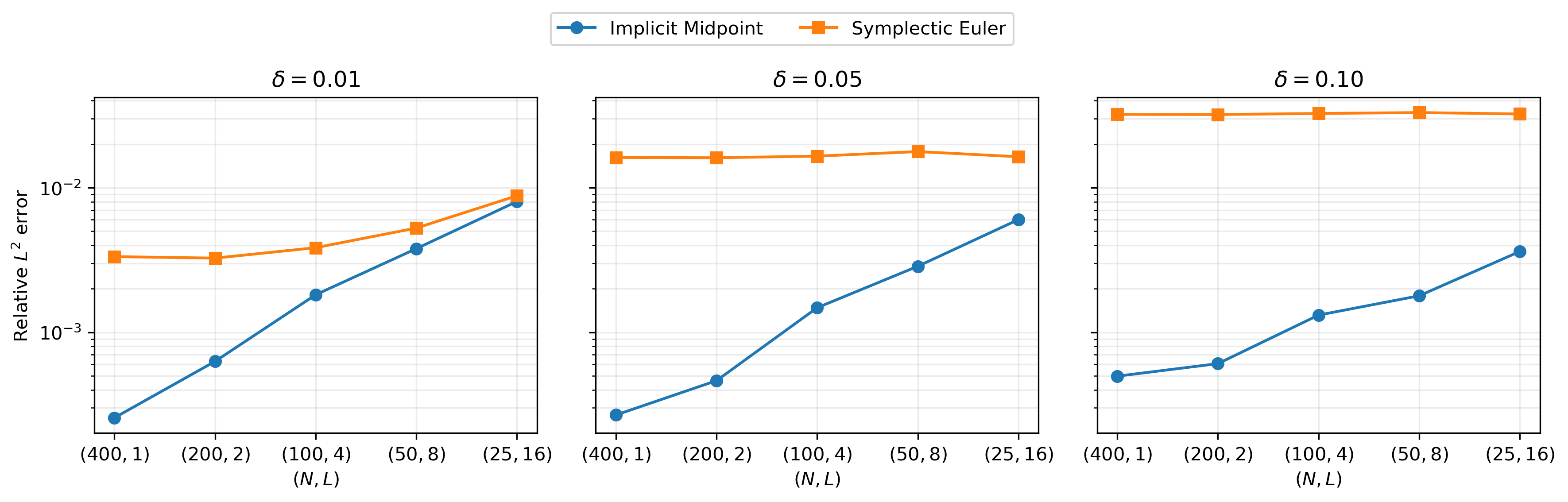}
    \caption{Effect of trajectory dependence under a fixed observation budget $NL=400$. The midpoint reconstruction benefits consistently from allocating observations across more independent trajectories. For symplectic Euler, the effect is largely masked by the dominant first-order integrator-induced error.}
    \label{fig:HH-NL}
\end{figure}

These results also emphasize that $N$ and $L$ play fundamentally different roles in the learning problem. Increasing $N$ changes the amount of independent statistical information, while increasing $L$ modifies the dynamical exploration of each trajectory and introduces stronger dependence among observations. 

\subsection{Additional Hamiltonian systems}

We finally test the trajectory-based implicit midpoint estimator on three additional Hamiltonian systems with different geometric structures. In all three experiments, we take $L=1$, so that each independent trajectory contributes a single endpoint pair. For each system, we independently sample the initial states, generate $(z_i^-,z_i^+)$ using DOP853, and construct
\[
x_i=\frac{z_i^-+z_i^+}{2},
\qquad
v_i=\frac{z_i^+-z_i^-}{\delta}.
\]
In the reconstruction plots, the learned Hamiltonian is shown without vertical alignment, while the mismatch error is computed after removing the optimal additive constant.

\paragraph{Double pendulum.}
We first consider the double mathematical pendulum with $m=l=g=1$ and Hamiltonian
\[
H(q_1,q_2,p_1,p_2)
=
\frac12
\frac{p_1^2+2p_2^2-2p_1p_2\cos(q_1-q_2)}
{1+\sin^2(q_1-q_2)}
+
4-2\cos q_1-\cos q_2.
\]
Unlike the H\'enon--Heiles system, the kinetic energy depends nonlinearly on the configuration variables, so the Hamiltonian is not of the separable form $\frac12|p|^2+V(q)$. We use $N=1000$, $\delta=0.1$, $\eta=10^{-9}$, and $\ell=3$. The initial conditions are sampled uniformly from $[-3,3]^4$, and the reconstruction is visualized on the slice $p_1=p_2=0$ over $(q_1,q_2)\in[-3,3]^2$. Figure~\ref{fig:double-pendulum} compares the true Hamiltonian, the learned Hamiltonian, and the aligned mismatch error. The reconstructed surface closely follows the nonlinear geometry of the ground truth over the displayed configuration region.

\begin{figure}[t]
    \centering
    \includegraphics[width=\textwidth]{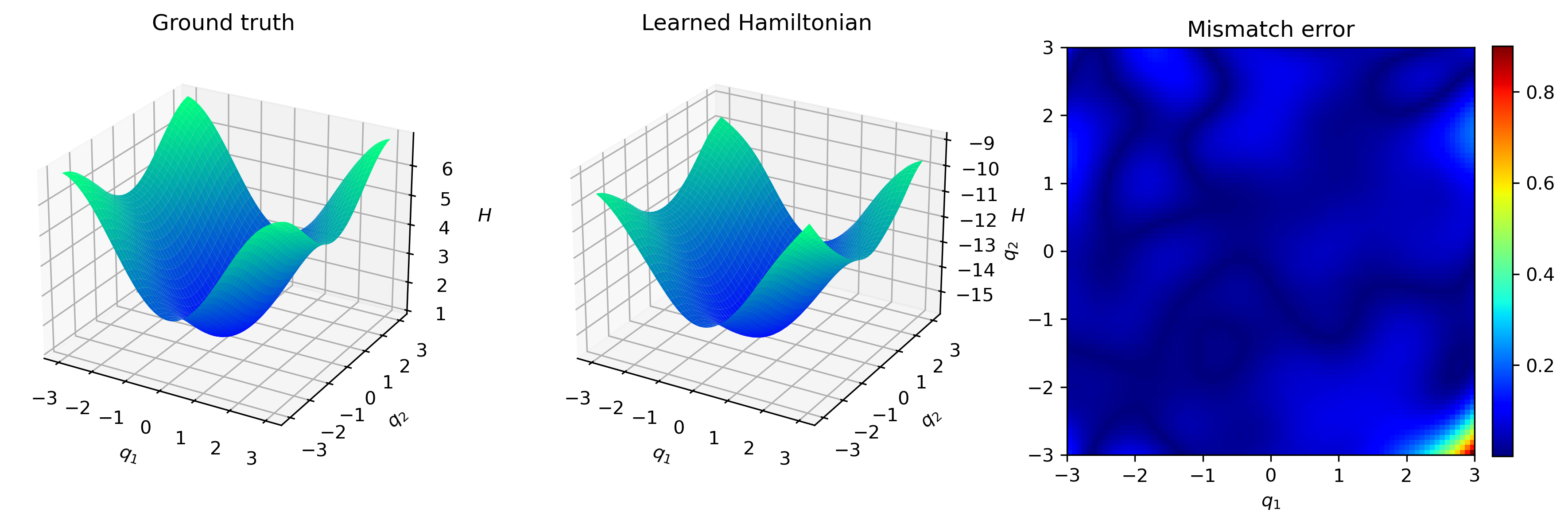}
    \caption{Double pendulum reconstruction from trajectory endpoint pairs using the implicit midpoint estimator. Left: ground-truth Hamiltonian on the slice $p_1=p_2=0$. Middle: learned Hamiltonian without vertical alignment. Right: absolute mismatch after removing the optimal additive constant.}
    \label{fig:double-pendulum}
\end{figure}

\paragraph{Frenkel--Kontorova model.}
We next consider the two-particle Frenkel--Kontorova Hamiltonian
\[
H(q_1,q_2,p_1,p_2)
=
\frac{p_1^2+p_2^2}{2}
-\cos q_1-\cos q_2
+\frac{g}{2}(q_2-q_1)^2,
\qquad g=1.
\]
This model combines a periodic on-site potential with nearest-neighbor coupling and therefore provides an energy landscape qualitatively different from the polynomial H\'enon--Heiles system. We use $N=200$, $\delta=0.1$, $\eta=10^{-9}$, and $\ell=1$. The initial conditions are sampled uniformly from $[-1,1]^4$, and the reconstruction is evaluated on the slice $p_1=p_2=0$ over $(q_1,q_2)\in[-1,1]^2$. Figure~\ref{fig:FK} shows that the learned Hamiltonian captures both the periodic structure and the deformation induced by the coupling term.

\begin{figure}[t]
    \centering
    \includegraphics[width=\textwidth]{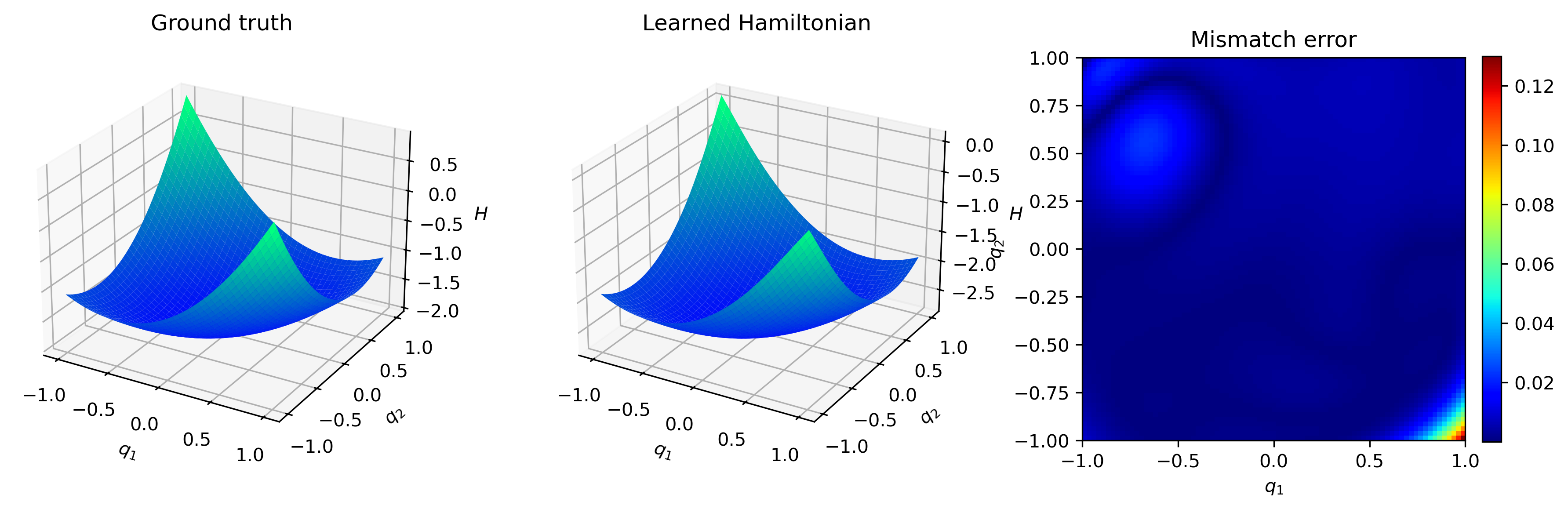}
    \caption{Frenkel--Kontorova reconstruction. Left: ground-truth Hamiltonian on the slice $p_1=p_2=0$. Middle: learned Hamiltonian without vertical alignment. Right: absolute mismatch after removing the optimal additive constant.}
    \label{fig:FK}
\end{figure}

\paragraph{Highly non-convex potential.}
Finally, we consider
\[
H(q_1,q_2,p_1,p_2)
=
\frac12(p_1^2+p_2^2)
+
\sin\left(\frac{2\pi}{3}q_1\right)
\cos\left(\frac{2\pi}{3}q_2\right)
+
\frac{\sin\left(\sqrt{q_1^2+q_2^2}\right)}
{\sqrt{q_1^2+q_2^2}},
\]
where the last term is continuously extended by the value $1$ at the origin. The corresponding potential contains several oscillatory peaks and wells and is strongly non-convex. We use $N=1500$, $\delta=0.05$, $\eta=10^{-9}$, and $\ell=1.2$. The initial conditions are sampled uniformly from $[-3,3]^4$, and the reconstruction is evaluated on the slice $p_1=p_2=0$ over $(q_1,q_2)\in[-3,3]^2$.

Figure~\ref{fig:nonconvex} shows that the learned Hamiltonian reproduces the main oscillatory and non-convex features of the true energy landscape. The raw estimator differs visibly from the ground truth by an additive vertical offset, as expected from the identifiability of the Hamiltonian. After removing this constant, the mismatch remains small over most of the interior of the domain, with the largest errors concentrated near the boundary.

\begin{figure}[t]
    \centering
    \includegraphics[width=\textwidth]{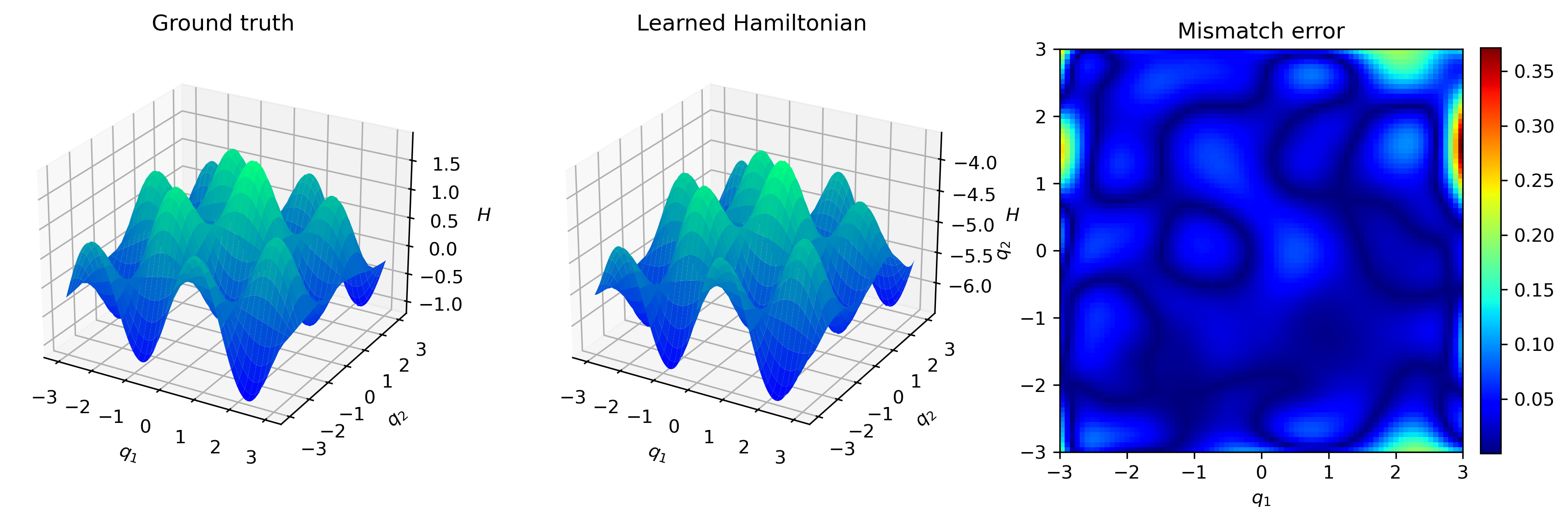}
    \caption{Reconstruction of the highly non-convex Hamiltonian. Left: ground-truth Hamiltonian on the slice $p_1=p_2=0$. Middle: learned Hamiltonian without vertical alignment. Right: absolute mismatch after removing the optimal additive constant.}
    \label{fig:nonconvex}
\end{figure}

Together with the H\'enon--Heiles experiments above, these examples show that the same trajectory-based midpoint reconstruction applies across Hamiltonian systems with nonseparable kinetic energy, periodic interactions, and strongly non-convex potential landscapes.

\section{Conclusions and future work}

We studied the reconstruction of a Hamiltonian $H:\RR^{2d}\to\RR$ from discrete trajectory data when the learning relation is defined through a symplectic one-step method $\Int\in\{\MP,\SE\}$, namely the implicit midpoint rule ($\MP$) or the $p$-implicit symplectic Euler method ($\SE$). The data consist of $N$ independent trajectories with initial conditions $z_0^{(n)}\sim\mu_0$, each contributing $L$ consecutive endpoint pairs separated by the observation step $\delta>0$. From each pair $(z^-,z^+)$ we form the scheme-adapted anchor $x^\Int$ and the discrete velocity $v_\delta=(z^+-z^-)/\delta$, and we estimate $H$ by the regularized least-squares estimator $\widehat h_{\lambda,N,L}^{\Int}$ in the RKHS $\HK$, with Tikhonov parameter $\lambda>0$ and effective regularization parameter $\eta=\lambda/\delta^2$.

Under the analytic setting of Assumption~\ref{ass:analytic-setting}, every exact-flow pair satisfies $v_\delta=J\nabla\HInt_\delta(x^\Int)$ for $|\delta|<\delta_*^\Int$, where $\HInt_\delta$ is the exact inverse modified Hamiltonian of Proposition~\ref{prop:exact-inverse-targets}. At a fixed observation step, the data therefore provide exact derivative observations of the scheme-dependent target $\HInt_\delta$, rather than of $H$. Under Assumption~\ref{ass:domain}, this discrepancy is controlled in the native-space norm: for all sufficiently small $\delta>0$,
\[
\|\HInt_\delta-H\|_{\HK}\le C_\Int\delta^{r_\Int},
\qquad r_{\MP}=2,\quad r_{\SE}=1,
\]
where $r_\Int$ is the order of the integrator-induced error and $C_\Int$ is the constant in~\eqref{eq:generic-native-bias}. The midpoint family is even in $\delta$, whereas symplectic Euler generally has a nonvanishing first-order correction. Since only $J\nabla h$ is observed, both $H$ and $\HInt_\delta$ are identifiable only up to additive constants, and reconstruction takes place on the identifiable subspace $(\mathcal H_{\mathrm{null}}^\Int)^\perp$.

The reconstruction error separates, as in~\eqref{eq:error-decomposition}, into estimation, approximation, and integrator-induced contributions. Let $\gamma\in(0,1]$ and $\psi_\delta^\Int\in\HK$ be the exponent and source element in the source condition $\HInt_\delta=(B^\Int)^\gamma\psi_\delta^\Int$, where $B^\Int$ is the population covariance operator associated with the averaged anchor law $\rho^\Int$, and let $\kappa^2=\|K\|_{C_b^2(\Omega\times\Omega)}$. Theorem~\ref{thm:global-error} gives, for every $0<\varepsilon<1$, with probability at least $1-\varepsilon$,
\[
\|\widehat h_{\lambda,N,L}^{\Int}-H\|_{\HK}
\le
C\frac{2d\kappa^2}{\eta}\sqrt{\frac{\log(4/\varepsilon)}{N}}\,\|\HInt_\delta\|_{\HK}
+\eta^\gamma\|\psi_\delta^\Int\|_{\HK}
+C_\Int\delta^{r_\Int},
\]
with $C>0$ universal. The trajectory-block concentration of Lemma~\ref{lem:trajectory-concentration} uses independence only across the $N$ trajectories and allows arbitrary dependence among the $L$ anchors within each trajectory; accordingly, the bound depends on $N$ but not on $L$. The first two terms are governed by the statistical parameters $(N,\eta)$, whereas the third depends only on $(\Int,\delta)$. At fixed $\delta$, additional data and weaker regularization drive the estimator toward $\HInt_\delta$, not toward $H$, so the integrator-induced error is not removed by increasing $N$.

Under a common source exponent $\gamma$ and uniformly bounded minimal source norms~\eqref{eq:uniform-source-norm}, the choices $\eta_N\asymp N^{-1/(2\gamma+2)}$, $\delta_N\asymp N^{-\gamma/[r_\Int(2\gamma+2)]}$, and $\lambda_N=\eta_N\delta_N^2$ balance the three contributions and yield $\|\widehat h_{\lambda_N,N,L}^{\Int}-H\|_{\HK}=\mathcal O_{\mathbb P}(N^{-\gamma/(2\gamma+2)})$ (Corollary~\ref{cor:balanced-joint-rate}). Both integrators attain the same rate; the order $r_\Int$ enters only through the required step-size refinement, $\delta_N^{\MP}\asymp N^{-\gamma/[4(\gamma+1)]}$ versus $\delta_N^{\SE}\asymp N^{-\gamma/[2(\gamma+1)]}$. If the source norms are allowed to grow as $S_\delta^\Int\lesssim\delta^{-\beta}$ for some $\beta\ge0$, the balanced rate becomes $N^{-r_\Int\gamma/[2(r_\Int(\gamma+1)+\beta)]}$, which now depends on $r_\Int$.

For noisy trajectory states $\widetilde z_l^{(n)}=z_l^{(n)}+e_l^{(n)}$, the noise enters the regression problem through two channels: an integrator-dependent anchor perturbation $U_l^{\Int,(n)}$ and a velocity perturbation $W_l^{(n)}=(e_l^{(n)}-e_{l-1}^{(n)})/\delta$. In terms of the realized pathwise noise levels $r_U$ and $r_W$ of~\eqref{eq:pathwise-noise-radii} and the constants $\kappa_1,\kappa_2$ of the embedding $\HK\hookrightarrow C^2(\overline\Omega)$ in Remark~\ref{rem:derivative-bounds}, Proposition~\ref{prop:deterministic-noisy-sensitivity} gives, for every realization of the noise satisfying its hypotheses,
\[
\|\widehat h_{\lambda,NL}^{U,W}-\widehat h_{\lambda,N,L}^{\Int}\|_{\HK}
\le
\frac{r_W}{2\sqrt\eta}+\frac{3\kappa_1\kappa_2}{\eta}\|\HInt_\delta\|_{\HK}\,r_U .
\]
Velocity noise is thus amplified at the scale $\eta^{-1/2}$ and anchor noise at the scale $\eta^{-1}$. Since $W_l^{(n)}$ carries an explicit factor $\delta^{-1}$, decreasing the observation step reduces the integrator-induced error but may amplify the effect of observation noise.

The inverse-target expansion $\check H_{s\delta}^{\Int}=H+\sum_{j\ge1}s^{r_\Int j}\delta^{r_\Int j}\check H_{r_\Int j}^{\Int}$ leads to a multistride Richardson construction that reuses the same trajectories. For an integer stride $s\in\{1,\ldots,Q\}$, where $2\le Q\le L$ is the number of strides (the extrapolation level), the pairs $(z_{l-s}^{(n)},z_l^{(n)})$ are exact-flow pairs over the step $s\delta$ and yield a stride estimator $\widehat h_{s,\delta}^{\Int}$ targeting $\check H_{s\delta}^\Int$; the stride-balanced regularization $\lambda_s=(s\delta)^2\eta$ keeps the effective parameter $\eta$ common to all strides. The Richardson coefficients $\alpha_1^\Int,\ldots,\alpha_Q^\Int$, defined by the Vandermonde system~\eqref{eq:vandermonde}, cancel the first $Q-1$ inverse-target corrections, so that the Richardson estimator $\widehat h_\delta^{\Int,(Q)}=\sum_{s=1}^Q\alpha_s^\Int\widehat h_{s,\delta}^{\Int}$ targets $\check H_\delta^{\Int,(Q)}=\sum_{s=1}^Q\alpha_s^\Int\check H_{s\delta}^\Int$ with $\|\check H_\delta^{\Int,(Q)}-H\|_{\HK}\le C_{Q,\Int}\delta^{r_\Int Q}$ (Proposition~\ref{prop:richardson-target-bias}). For $Q=2$, $(\alpha_1^{\MP},\alpha_2^{\MP})=(\frac43,-\frac13)$ and $(\alpha_1^{\SE},\alpha_2^{\SE})=(2,-1)$. The price is the Richardson amplification factor $\mathcal A_Q^\Int=\sum_{s=1}^Q|\alpha_s^\Int|$: Theorem~\ref{thm:rich-decomp} gives, with probability at least $1-\varepsilon$,
\[
\bigl\|\widehat h_\delta^{\Int,(Q)}-H\bigr\|_{\HK}
\le
C\mathcal A_Q^\Int\left(\frac{2d\kappa^2M_{Q,\delta}^\Int}{\eta}\sqrt{\frac{\log(4Q/\varepsilon)}{N}}+\eta^\gamma S_{Q,\delta}^\Int\right)
+C_{Q,\Int}\delta^{r_\Int Q},
\]
where $M_{Q,\delta}^\Int=\max_{1\le s\le Q}\|\check H_{s\delta}^\Int\|_{\HK}$ and $S_{Q,\delta}^\Int$ is the largest minimal stride source norm. No independence across strides is required. Under uniform source bounds, the Richardson estimator attains the rate $\mathcal O_{\mathbb P}(\mathcal A_Q^\Int N^{-\gamma/(2\gamma+2)})$ with the relaxed step-size refinement $\delta_N\asymp N^{-\gamma/[r_\Int Q(2\gamma+2)]}$, so that, for fixed $Q$, the same reconstruction rate is achieved with a larger observation step.

The H\'enon--Heiles experiments are consistent with this picture: they exhibit the predicted orders $r_{\MP}=2$ and $r_{\SE}=1$ together with the leading inverse-target corrections, the saturation of the error in $N$ at fixed $\delta$, and the effectiveness of Richardson extrapolation until estimation, approximation, or observation-noise effects dominate. Under a fixed observation budget $NL$, allocating observations to independent trajectories is more effective than lengthening trajectories. The midpoint estimator also gives accurate reconstructions for the double pendulum, Frenkel--Kontorova, and strongly non-convex examples.

Several questions remain. The present trajectory-block bound does not exploit the dependence structure inside each trajectory; incorporating mixing, ergodicity, or other dynamical information could sharpen the dependence on $L$, in the spirit of related learning results from dependent trajectory data \cite{lu2019nonparametric,feng2024learning}. The pathwise noise analysis could likewise be complemented by a statistical errors-in-variables treatment identifying the population target under a prescribed observation-noise law; such a treatment would also make it possible to analyze jointly the choice of $(\delta,\eta,Q)$ under noise, where the $\delta^{-1}$ amplification of velocity noise and the Richardson factor $\mathcal A_Q^\Int$ compete with the reduction of the integrator-induced error. The balanced rates rely on the unknown exponent $\gamma$ and on the uniform source bounds~\eqref{eq:uniform-source-norm} and~\eqref{eq:uniform-richardson-bounds}; verifying these bounds for concrete classes of Hamiltonians and sampling laws, and designing data-driven choices of $\eta$, $\delta$, and $Q$, are natural next steps. It would also be useful to extend the exact inverse-target construction beyond midpoint and symplectic Euler to higher-order symplectic and composition methods, building on the broader framework of geometric numerical integration \cite{HLW06,sanzserna1994numerical,leimkuhler2004simulating}, and to determine how much of the analytic construction survives under finite regularity.

\section*{Acknowledgments}
The authors acknowledge financial support from the School of Physical and Mathematical Sciences of the Nanyang Technological University. Financial support from Singapore's Ministry of Education Tier 1 grant RG100/24 entitled ``Kernel methods for global structure preserving machine learning" is also acknowledged. MV acknowledges financial support from grants PID2024-155187OB-I00, PCI2024-155047-2, and PID2022-137909NB-C21. DMdD  acknowledges financial support from the Grants PID2022-137909-NB-C21 and   CEX2023-001347-S funded by MICIU/AEI/10.13039/501100011033.

\bibliographystyle{plain}
\bibliography{references}

\clearpage
\appendix

\phantomsection

\section{Backward error analysis and inverse modified Hamiltonians}
\label{sec:bae}

This appendix constructs and analyzes the exact inverse modified Hamiltonians used for the two integrators considered in the paper. For the implicit midpoint rule, the construction locally inverts the exact-flow midpoint map; for the $p$-implicit symplectic Euler rule, it locally inverts the exact-flow mixed-endpoint map. In both cases, symplecticity implies that the one-form forced by the exact chord is closed, yielding a genuine jointly analytic Hamiltonian on a fixed neighborhood. The resulting analytic coefficient bounds yield both uniform smoothness estimates and, under the native-space assumptions introduced in Section~\ref{sec:rkhs-formulation}, corresponding bounds in $\HK$. These estimates give the method-specific target biases: quadratic and even in $\delta$ for midpoint, but generally first order with a full power-series expansion for symplectic Euler. Both constructions proceed directly from the exact flow and do not rely on a general B-series construction. Sections~\ref{sec:bae-midpoint} and~\ref{sec:bae-se} are organized in parallel: each first constructs the exact inverse target and then derives its integrator-specific coefficient and target-bias bounds.

\subsection{Implicit midpoint: exact inverse modified Hamiltonian}
\label{sec:bae-midpoint}

\subsubsection{Exact midpoint-map construction}

For $z$ near $\mathcal K$ and $\delta$ near zero, let
\[
\psi_\delta(z)
:=
\frac{z+\varphi_\delta^H(z)}{2}
\]
be the midpoint of the exact-flow pair.  If a vector field $g$ reproduces that
pair through one midpoint step, its value at $m=\psi_\delta(z)$ is forced to be
\[
g(m)=\frac{\varphi_\delta^H(z)-z}{\delta}.
\]
The next proposition shows that this pointwise prescription defines a genuine
Hamiltonian field on a fixed neighborhood.

\begin{proposition}[Exact analytic midpoint inverse]\label{prop:imh-analytic}
Under Assumption~\ref{ass:analytic-setting}, there exist $\sigma>0$ and $\delta_*>0$ such that
the following statements hold.

\begin{enumerate}[label=\textup{(\roman*)},leftmargin=2em]
\item The map $\psi_\delta$ has a jointly holomorphic local inverse
$z_\delta=\psi_\delta^{-1}$ on $\mathcal K_\sigma$ for $|\delta|<2\delta_*$.
The removable definition
\begin{equation}\label{eq:inv-mod-condition}
\widehat f_\delta(m)
:=
\begin{cases}
\dfrac{\varphi_\delta^H(z_\delta(m))-z_\delta(m)}{\delta},
&\delta\ne0,\\[7pt]
f(m),&\delta=0,
\end{cases}
\end{equation}
is jointly holomorphic in $(m,\delta)$.

\item There is a jointly holomorphic Hamiltonian
$\check H_\delta$, real for real $(m,\delta)$ and unique up to a
$\delta$-dependent additive constant, such that
\[
\widehat f_\delta=J\nabla\check H_\delta.
\]
After fixing $z_\star\in\mathcal K$ and imposing
$\check H_\delta(z_\star)=H(z_\star)$, one has the exact identity
\begin{equation}\label{eq:exact-inverse-midpoint}
I_\delta^{\check H_\delta}(z)=\varphi_\delta^H(z)
\end{equation}
for every $z\in\mathcal K$ and $|\delta|<\delta_*$.

\item The family is even:
$\check H_{-\delta}=\check H_\delta$.  Consequently,
\begin{equation}\label{eq:inv-mod-recursion}
\check H_\delta
=
H+\sum_{k\ge1}\delta^k\check H_k
=
H+\sum_{j\ge1}\delta^{2j}\check H_{2j},
\qquad
\widehat f_\delta
=
f+\sum_{k\ge1}\delta^k
\hypertarget{sym-fhatk}{\widehat f_k},
\end{equation}
where both series converge uniformly, together with every fixed number of
spatial derivatives, on compact subsets of $\mathcal K_\sigma$ for
$|\delta|<2\delta_*$.  These Taylor coefficients are the formal inverse
modified coefficients: substituting the series into
$I_\delta^{\check H_\delta}=\varphi_\delta^H$ and matching powers of $\delta$
recovers the usual formal inverse modified equation.
\end{enumerate}
\end{proposition}

\begin{proof}
\emph{Step 1: invert the midpoint coordinate.}
The holomorphic ODE theorem
(Proposition~\ref{prop:appendix-holomorphic-ode}) gives joint holomorphy of
$(\delta,z)\mapsto\varphi_\delta^H(z)$ on a neighborhood of
$\{0\}\times\mathcal K$.  Hence $\psi_\delta$ is jointly holomorphic.  Consider
\[
G(\delta,m;z):=\psi_\delta(z)-m.
\]
Since $\psi_0=\operatorname{id}$, for every $m\in\mathcal K$,
\[
G(0,m;m)=0,
\qquad
D_zG(0,m;m)=I.
\]
The analytic implicit function theorem
(Proposition~\ref{prop:appendix-analytic-ift}) therefore gives, locally, a unique jointly
holomorphic solution $z=z_\delta(m)$.  Apply this construction on the closure of a sufficiently
small preliminary tube about $\mathcal K$.  Compactness gives a common complex step-size disk and,
after passing to a smaller tube $\mathcal K_\sigma$, uniqueness makes the local solutions agree on
overlaps.  Thus $z_\delta$ is a jointly holomorphic local inverse and, after shrinking the disk once
more, its two inverse identities give
\[
\psi_\delta(z_\delta(m))=m
\quad(m\in\mathcal K_\sigma),
\qquad
z_\delta(\psi_\delta(z))=z
\quad(z\in\mathcal K).
\]
Here the disk is chosen so that every exact midpoint issued from $\mathcal K$ lies in
$\mathcal K_\sigma$.  In~\eqref{eq:inv-mod-condition} the apparent quotient at
$\delta=0$ is removable because
\[
\frac{\varphi_\delta^H(z)-z}{\delta}
=
\int_0^1 f(\varphi_{t\delta}^H(z))\,dt.
\]
Thus $\widehat f_\delta$ is jointly holomorphic and
$\widehat f_0=f$.

\medskip\noindent
\emph{Step 2: use symplecticity to obtain a Hamiltonian.}
For $\delta\ne0$ and $m$ in the image of $\psi_\delta$, write the two endpoints as
$x_\delta(m):=z_\delta(m)$ and
$y_\delta(m):=\varphi_\delta^H(z_\delta(m))$. By the definitions of the midpoint
and the chord field,
\[
m=\frac{x_\delta(m)+y_\delta(m)}{2},
\qquad
\widehat f_\delta(m)=\frac{y_\delta(m)-x_\delta(m)}{\delta}.
\]
Solving these two equations for the endpoints gives
\[
x_\delta(m)=m-\frac{\delta}{2}\widehat f_\delta(m),
\qquad
y_\delta(m)=m+\frac{\delta}{2}\widehat f_\delta(m).
\]
The same endpoint formulas hold at $\delta=0$ by continuity, since then
$x_0(m)=y_0(m)=m$. In particular, no additional factor occurs. By construction,
$y_\delta=\varphi_\delta^H\circ x_\delta$.

Fix $m$. Differentiating the endpoint maps with respect to $m$ gives
\[
Dx_\delta(m)=I-\frac{\delta}{2}D\widehat f_\delta(m),
\qquad
Dy_\delta(m)=I+\frac{\delta}{2}D\widehat f_\delta(m).
\]
Since the complexified exact Hamiltonian flow is symplectic,
\[
y_\delta^*\omega
=x_\delta^*(\varphi_\delta^H)^*\omega
=x_\delta^*\omega.
\]
In canonical coordinates the pullback of $\omega$ by a differentiable map $G$
is represented by $(DG)^\top J(DG)$. Hence the preceding equality is
\[
\left(I-\frac{\delta}{2}D\widehat f_\delta(m)\right)^\top
J\left(I-\frac{\delta}{2}D\widehat f_\delta(m)\right)
=
\left(I+\frac{\delta}{2}D\widehat f_\delta(m)\right)^\top
J\left(I+\frac{\delta}{2}D\widehat f_\delta(m)\right),
\]
whose two sides expand, respectively, as
\begin{align*}
J
&-\frac{\delta}{2}J D\widehat f_\delta(m)
-\frac{\delta}{2}(D\widehat f_\delta(m))^\top J
+\frac{\delta^2}{4}(D\widehat f_\delta(m))^\top
  J D\widehat f_\delta(m),\\
J
&+\frac{\delta}{2}J D\widehat f_\delta(m)
+\frac{\delta}{2}(D\widehat f_\delta(m))^\top J
+\frac{\delta^2}{4}(D\widehat f_\delta(m))^\top
  J D\widehat f_\delta(m).
\end{align*}
Canceling the common zeroth- and second-order terms gives
\[
-\frac{\delta}{2}
\left[
(D\widehat f_\delta(m))^\top J+J D\widehat f_\delta(m)
\right]
=
\frac{\delta}{2}
\left[
(D\widehat f_\delta(m))^\top J+J D\widehat f_\delta(m)
\right].
\]
For $\delta\ne0$, rearranging and dividing by $\delta$ therefore gives
\begin{equation}\label{eq:fhat-hamiltonian-jacobian}
(D\widehat f_\delta(m))^\top J
+J\,D\widehat f_\delta(m)=0.
\end{equation}
At $\delta=0$, Step~1 gives $\widehat f_0=f=J\nabla H$, and hence
$D\widehat f_0=J\Hess H$. Using $J^\top=-J$, $J^2=-I$, and symmetry of
$\Hess H$ gives
\[
(D\widehat f_0)^\top J+J D\widehat f_0
=(J\Hess H)^\top J+J(J\Hess H)
=\Hess H-\Hess H=0.
\]
Thus~\eqref{eq:fhat-hamiltonian-jacobian} holds for every $|\delta|<2\delta_*$,
including $\delta=0$.

Now set $g_\delta:=-J\widehat f_\delta$. Its Jacobian is
$Dg_\delta=-J D\widehat f_\delta$, and
\eqref{eq:fhat-hamiltonian-jacobian} implies
\[
(Dg_\delta)^\top
=(D\widehat f_\delta)^\top J
=-J D\widehat f_\delta
=Dg_\delta.
\]
Therefore, the one-form $g_\delta(m)^\top dm$ is closed. Because
$\mathcal K_\sigma$ may be chosen convex, it is exact. With the normalization at
$z_\star$, define
\begin{equation}\label{eq:checkH-line-integral}
\check H_\delta(m)
:=
H(z_\star)
+
\int_0^1
(m-z_\star)^\top
\bigl[-J\widehat f_\delta(z_\star+t(m-z_\star))\bigr]\,dt.
\end{equation}
For completeness, let $m_t=z_\star+t(m-z_\star)$ and differentiate this formula in
a direction $h$. Symmetry of $Dg_\delta$ gives
\begin{align*}
D\check H_\delta(m)[h]
&=\int_0^1\left[
h^\top g_\delta(m_t)
+t(m-z_\star)^\top Dg_\delta(m_t)h
\right]dt\\
&=\int_0^1\frac{d}{dt}
\left[t\,h^\top g_\delta(m_t)\right]dt
=h^\top g_\delta(m).
\end{align*}
Consequently $\nabla\check H_\delta=g_\delta=-J\widehat f_\delta$, equivalently
$\widehat f_\delta=J\nabla\check H_\delta$. Formula
\eqref{eq:checkH-line-integral} also gives joint holomorphy and reality on real
arguments.  Finally, if $m=\psi_\delta(z)$, then
\[
z=m-\frac{\delta}{2}\widehat f_\delta(m),
\qquad
\varphi_\delta^H(z)
=m+\frac{\delta}{2}\widehat f_\delta(m),
\]
which proves~\eqref{eq:exact-inverse-midpoint}.  Local uniqueness of the
midpoint branch follows from Proposition~\ref{prop:appendix-analytic-ift} applied to
$m-z-(\delta/2)\widehat f_\delta(m)=0$: its derivative with respect to $m$ is $I$ at
$\delta=0$, and compactness of $\mathcal K$ makes the choice uniform after decreasing
$\delta_*$.

\medskip\noindent
\emph{Step 3: exploit symmetry.}
Let $m=\psi_\delta(z)$ and $w=\varphi_\delta^H(z)$.  The flow-group identity
gives
\[
\psi_{-\delta}(w)
=
\frac{w+\varphi_{-\delta}^H(w)}{2}
=
\frac{w+z}{2}
=
m.
\]
Using~\eqref{eq:inv-mod-condition} at $\delta$ and $-\delta$ therefore yields
\[
\widehat f_{-\delta}(m)
=
\frac{\varphi_{-\delta}^H(w)-w}{-\delta}
=
\frac{z-w}{-\delta}
=
\widehat f_\delta(m).
\]
The normalization in~\eqref{eq:checkH-line-integral} then gives
$\check H_{-\delta}=\check H_\delta$.  Since
$\check H_0=H$, the convergent Taylor expansions in~\eqref{eq:inv-mod-recursion}
follow, and all odd coefficients vanish.
\end{proof}

\subsubsection{Quadratic target bias}

\begin{corollary}[Coefficient, native-space, and target-bias bounds]
\label{cor:imh-native-bound}
Under Assumption~\ref{ass:analytic-setting}, for every fixed integer $m\ge0$ there is $C_m>0$ such that
\begin{equation}\label{eq:imh-Cm-bound}
 \max_{|\alpha|\le m}\sup_{z\in\mathcal K}
 |\partial^\alpha\check H_k(z)|
 \le C_m\delta_*^{-k},
 \qquad k\ge1.
\end{equation}
If, in addition, Assumption~\ref{ass:domain} holds and an integer $m>s$ is chosen, then
\begin{equation}\label{eq:imh-native-coeff-bound}
 \|\check H_k\|_{\HK}
 \le C_\Omega\delta_*^{-k},
 \qquad k\ge1.
\end{equation}
For $|\delta|\le\delta_*/2$, the normalized family satisfies
\begin{align}
 \|\nabla\check H_\delta-\nabla H\|_{C^0(\mathcal K)}
 &\le C_{\mathcal K}\delta^2,
 \label{eq:target-bias-grad}\\
 \|\check H_\delta-H\|_{C^0(\mathcal K)}
 &\le C_{\mathcal K}\delta^2,
 \label{eq:target-bias-fun}\\
 \|\check H_\delta-H\|_{\HK}
 &\le C_\Omega\delta^2.
 \label{eq:target-bias-native}
\end{align}
The last estimate uses Assumption~\ref{ass:domain}.
\end{corollary}

\begin{proof}
Choose $0<r<\sigma$ and set
\[
M_r:=\sup_{\substack{|\zeta|=\delta_*\\z\in\mathcal K_r}}
|\check H_\zeta(z)|<\infty.
\]
For every multi-index $\alpha$ with $|\alpha|\le m$, the spatial Cauchy estimate of
Proposition~\ref{prop:appendix-cauchy-estimates} gives, uniformly on $|\zeta|=\delta_*$,
\[
 \sup_{z\in\mathcal K}|\partial_z^\alpha\check H_\zeta(z)|
 \le C_{m,r}M_r.
\]
Cauchy's coefficient formula in $\delta$ (Proposition~\ref{prop:appendix-cauchy-integral}) then gives
\[
 \partial_z^\alpha\check H_k(z)
 =\frac{1}{2\pi i}\int_{|\zeta|=\delta_*}
 \frac{\partial_z^\alpha\check H_\zeta(z)}{\zeta^{k+1}}\,d\zeta,
 \qquad
 \|\partial_z^\alpha\check H_k\|_{C^0(\mathcal K)}
 \le C_{m,r}M_r\delta_*^{-k},
\]
which proves~\eqref{eq:imh-Cm-bound}.  Restricting these bounds to $\overline\Omega$, choosing
an integer $m>s$, and using
\(
 H^m(\Omega)\hookrightarrow H^s(\Omega)\cong\HK
\)
proves~\eqref{eq:imh-native-coeff-bound}.  Finally, evenness gives $\check H_k=0$ for odd $k$.
Summing the convergent coefficient series from $k=2$, using
\(
 \sum_{k\ge2}(|\delta|/\delta_*)^k\le2\delta^2/\delta_*^2
\), proves~\eqref{eq:target-bias-grad} and~\eqref{eq:target-bias-fun} from
\eqref{eq:imh-Cm-bound} with $m=1$ and $m=0$, respectively, and proves
\eqref{eq:target-bias-native} from~\eqref{eq:imh-native-coeff-bound}.
\end{proof}

\begin{remark}[Inverse versus forward modified Hamiltonians for the midpoint rule]
The result above is different from the classical forward modified Hamiltonian.
The forward construction asks for a Hamiltonian whose \emph{exact flow}
represents a numerical step and is generally interpreted as an asymptotic
series.  Here the direction is reversed: the exact flow
$\varphi_\delta^H$ is given, and we seek the Hamiltonian whose \emph{midpoint
step} reproduces that flow.  The single midpoint coordinate
$\psi_\delta(z)$ makes this inverse problem pointwise.  Its local analytic
inverse produces the genuine analytic family $\check H_\delta$, and the
convergent series is simply its Taylor series in $\delta$.
\end{remark}

\subsection{Symplectic Euler: exact inverse modified Hamiltonian}
\label{sec:bae-se}

\subsubsection{Exact mixed-endpoint construction}

Let
\(
 M:=\mathbb R_q^d\times\mathbb R_p^d
\)
be the original phase space, with generic canonical coordinates \((q,p)\), and let
\(
 M_{\mathrm{mix}}:=\mathbb R_q^d\times\mathbb R_P^d
\)
be a notationally distinct copy used for mixed endpoints, with coordinates \((q,P)\).  Write the
exact flow in blocks as
\[
 \varphi_\delta^H(q,p)
 =
 \bigl(Q_\delta(q,p),P_\delta(q,p)\bigr).
\]
Define the mixed map from the initial-state copy to the mixed-endpoint copy by
\begin{equation*}
 \chi_\delta(q,p)
 :=
 \bigl(q,P_\delta(q,p)\bigr).
\end{equation*}
Under the canonical identification of the two Euclidean copies, \(\chi_0=\operatorname{Id}\), so
this map is locally invertible for small \(\delta\).  Write its inverse as
\[
 z_\delta(q,P)
 =
 \chi_\delta^{-1}(q,P)
 =
 \bigl(q,\pi_\delta(q,P)\bigr)
\]
and set
\[
 \mathcal Q_\delta(q,P)
 :=
 Q_\delta\bigl(q,\pi_\delta(q,P)\bigr).
\]
The chord prescription is naturally indexed by \((q,P)\in M_{\mathrm{mix}}\).  We identify
\(M_{\mathrm{mix}}\) with \(M\) by \((q,P)\mapsto(q,p=P)\), use the same symbols for the
resulting pullbacks, and define the inverse modified vector field on the original phase space by
\begin{equation}\label{eq:inv-mod-condition-se}
 \widehat f_\delta(q,p)
 :=
 \begin{cases}
 \displaystyle
 \frac1\delta
 \binom{\mathcal Q_\delta(q,p)-q}
       {p-\pi_\delta(q,p)},&\delta\ne0,\\[10pt]
 J\nabla H(q,p),&\delta=0.
 \end{cases}
\end{equation}
Thus \(\widehat f_\delta\colon M\to\mathbb R^{2d}\) lives in the same vector-field space as
\(f=J\nabla H\).  Before the identification, lowercase \(p\) denotes the initial momentum in the
flow map, whereas capital \(P\) denotes the mixed/final momentum.  After the identification, the
second argument in \(\widehat f_\delta(q,p)\) is a generic phase-space coordinate; for a particular
step it takes the value \(P\).  Thus, if
\(
 \varphi_\delta^H(q,p)=(Q,P)
\), then
\[
 \widehat f_\delta(q,P)
 =
 \frac1\delta\binom{Q-q}{P-p}.
\]

\begin{proposition}[Exact analytic symplectic Euler inverse]\label{prop:imh-analytic-se}
Under Assumption~\ref{ass:analytic-setting}, there exist \(\sigma>0\) and \(\delta_*>0\) such that:

\begin{enumerate}[label=\textup{(\roman*)},leftmargin=2em]
\item The mixed map \(\chi_\delta\) has a jointly holomorphic inverse \(z_\delta\) on
\(\mathcal K_\sigma\) for \(|\delta|<2\delta_*\), and
\(\widehat f_\delta\) in~\eqref{eq:inv-mod-condition-se} is jointly holomorphic.

\item There is a jointly holomorphic Hamiltonian \(\check H_\delta\), real on real arguments and
unique up to a \(\delta\)-dependent constant, such that
\[
 \widehat f_\delta=J\nabla\check H_\delta .
\]
After fixing \(z_\star\in\mathcal K\) and imposing
\(\check H_\delta(z_\star)=H(z_\star)\), it satisfies
\begin{equation*}
 I_\delta^{\check H_\delta}(z)=\varphi_\delta^H(z),
 \qquad z\in\mathcal K,\quad |\delta|<\delta_*.
\end{equation*}

\item The normalized target has the convergent expansion
\begin{equation}\label{eq:inv-mod-recursion-se}
 \check H_\delta
 =
 H+\sum_{k\ge1}\delta^k\check H_k ,
 \qquad
 \widehat f_\delta
 =
 J\nabla H+\sum_{k\ge1}\delta^k\widehat f_k .
\end{equation}
The series converge with every fixed number of spatial derivatives on compact subsets of
\(\mathcal K_\sigma\).  Their coefficients are the formal inverse modified coefficients.

\item The first coefficient is
\begin{equation}\label{eq:se-first-coefficient}
 \check H_1(z)
 =
 \frac12\langle H_q(z),H_p(z)\rangle
 -
 \frac12\langle H_q(z_\star),H_p(z_\star)\rangle .
\end{equation}
In particular, the family is generally not even in \(\delta\).
\end{enumerate}
\end{proposition}

\begin{proof}
\emph{Step 1: invert the mixed map.}
By Proposition~\ref{prop:appendix-holomorphic-ode},
\((\delta,z)\mapsto\varphi_\delta^H(z)\) is jointly holomorphic on a neighborhood of
\(\{0\}\times\mathcal K\).  In particular, its second block
\((\delta,q,\eta)\mapsto P_\delta(q,\eta)\) is jointly holomorphic.  Consider, on the
mixed-endpoint copy,
\[
 G(\delta,q,P;\eta):=P_\delta(q,\eta)-P.
\]
Since \(P_0(q,\eta)=\eta\), for every \((q,P)\in\mathcal K\),
\[
 G(0,q,P;P)=0,
 \qquad
 D_\eta G(0,q,P;P)=I.
\]
The analytic implicit function theorem
(Proposition~\ref{prop:appendix-analytic-ift}) therefore gives, locally, a unique jointly
holomorphic solution \(\eta=\pi_\delta(q,P)\).  Compactness of \(\mathcal K\) lets us choose a
common complex step-size disk and a fixed tube \(\mathcal K_\sigma\); after shrinking both if
necessary, uniqueness makes the local solutions agree on overlaps.  Thus
\(z_\delta(q,P)=(q,\pi_\delta(q,P))\) is jointly holomorphic and satisfies
\[
 \chi_\delta\bigl(z_\delta(q,P)\bigr)=(q,P),
 \qquad (q,P)\in\mathcal K_\sigma.
\]
Shrinking \(\delta_*\) once more ensures that every mixed endpoint issued from
\(\mathcal K\) lies in \(\mathcal K_\sigma\).  The local uniqueness in the implicit function
theorem then also gives
\(z_\delta(\chi_\delta(z))=z\) for \(z\in\mathcal K\).

The quotient in~\eqref{eq:inv-mod-condition-se} is removable.  Indeed, after the canonical
identification above, for \(\delta\ne0\) one has
\[
 \widehat f_\delta(q,p)
 =
 \frac{\varphi_\delta^H\bigl(z_\delta(q,p)\bigr)-z_\delta(q,p)}{\delta},
\]
because
\(
 \varphi_\delta^H(z_\delta(q,p))=(\mathcal Q_\delta(q,p),p)
\).  Moreover,
\[
 \frac{\varphi_\delta^H(z)-z}{\delta}
 =
 \int_0^1 f\bigl(\varphi_{t\delta}^H(z)\bigr)\,dt,
\]
and composition with \(z_\delta\) preserves joint holomorphy.  At \(\delta=0\), its value is
\(f=J\nabla H\).

\medskip\noindent
\emph{Step 2: prove closedness directly from symplecticity.}
Use the generic phase-space coordinates \(x=(q,p)\) from
\eqref{eq:inv-mod-condition-se} and abbreviate
\(\pi=\pi_\delta(q,p)\), \(\mathcal Q=\mathcal Q_\delta(q,p)\).  The one-form forced by the chord is
\begin{equation*}
 \beta_\delta
 :=
 (-J\widehat f_\delta)^\top dx
 =
 \frac{\pi-p}{\delta}\cdot dq
 +
 \frac{\mathcal Q-q}{\delta}\cdot dp .
\end{equation*}
For \(\delta\ne0\),
\begin{align*}
 \delta\,d\beta_\delta
 &=
 d\pi\wedge dq-dp\wedge dq
 +d\mathcal Q\wedge dp-dq\wedge dp\\
 &=
 -dq\wedge d\pi+d\mathcal Q\wedge dp.
\end{align*}
Here vector wedge notation means the sum over coordinate pairs.  The exact flow is symplectic, so
after expressing its graph in the mixed coordinates,
\[
 d\mathcal Q\wedge dp=dq\wedge d\pi.
\]
Therefore \(d\beta_\delta=0\).  Joint holomorphy extends this identity to \(\delta=0\), where
\(\beta_0=dH\).

The tube can be chosen convex.  Thus the closed form is exact, and the normalized potential is
\begin{equation}\label{eq:checkH-line-integral-se}
 \check H_\delta(x)
 :=
 H(z_\star)
 +
 \int_0^1
 \beta_\delta\bigl(z_\star+t(x-z_\star)\bigr)[x-z_\star]\,dt .
\end{equation}
Differentiating the line integral, using closedness, gives
\(d\check H_\delta=\beta_\delta\), or
\(\widehat f_\delta=J\nabla\check H_\delta\).  Formula
\eqref{eq:checkH-line-integral-se} also gives joint holomorphy and reality on real arguments.

Now take an initial state \(z=(q,p)\in\mathcal K\), write
\(\varphi_\delta^H(z)=(Q,P)\), and put
\(x=\chi_\delta(z)=(q,P)\).  By the definitions above,
\[
 J\nabla\check H_\delta(q,P)
 =
 \frac1\delta\binom{Q-q}{P-p}.
\]
These are exactly the two equations of the \(p\)-implicit symplectic Euler rule.  The branch
continuous at \(\delta=0\) is locally unique after reducing \(\delta_*\), proving the exact
one-step identity.  Any other scalar with the same gradient differs by a constant.

\medskip\noindent
\emph{Step 3: expand and identify the first coefficient.}
Joint holomorphy gives the Taylor series~\eqref{eq:inv-mod-recursion-se}; substituting it into the exact
one-step identity recovers the formal inverse modified series coefficient by coefficient.

For completeness, keep the generic phase-space argument \(x=(q,p)\) fixed.  The exact-flow Taylor
expansion and series reversion give, with all derivatives evaluated at \((q,p)\),
\begin{align*}
 \pi_\delta(q,p)
 &=
 p+\delta H_q
 +\frac{\delta^2}{2}\bigl(H_{qq}H_p+H_{qp}H_q\bigr)
 +\mathcal O(\delta^3),\\
 \mathcal Q_\delta(q,p)
 &=
 q+\delta H_p
 +\frac{\delta^2}{2}\bigl(H_{pp}H_q+H_{pq}H_p\bigr)
 +\mathcal O(\delta^3).
\end{align*}
Consequently,
\[
 \widehat f_\delta
 =
 J\nabla H
 +
 \frac{\delta}{2}
 J\nabla\langle H_q,H_p\rangle
 +\mathcal O(\delta^2).
\]
Integrating and applying the normalization at \(z_\star\) yields
\eqref{eq:se-first-coefficient}.
\end{proof}

\begin{remark}[The adjoint Euler variant]
For the \(q\)-implicit adjoint rule, the mixed coordinate is \((Q,p)\). Repeating the entire
construction with that coordinate gives a normalized analytic family
\(\check H_\delta^{\mathrm{adj}}\) whose \(q\)-implicit step reproduces the exact flow. Its first
coefficient is
\(\check H_1^{\mathrm{adj}}=-\frac12\langle H_q,H_p\rangle\) modulo constants. This sign change is
also forced by adjointness.
\end{remark}

\subsubsection{First-order target bias}

\begin{corollary}[Coefficient, native-space, and target-bias bounds]
\label{cor:imh-native-bound-se}
Under Assumption~\ref{ass:analytic-setting}, for every fixed integer \(m\ge0\) there is \(C_m>0\) such that
\begin{equation}\label{eq:imh-Cm-bound-se}
 \max_{|\alpha|\le m}\sup_{z\in\mathcal K}
 |\partial^\alpha\check H_k(z)|
 \le C_m\delta_*^{-k},
 \qquad k\ge1.
\end{equation}
If, in addition, Assumption~\ref{ass:domain} holds and an integer \(m>s\) is chosen, then
\begin{equation}\label{eq:imh-native-coeff-bound-se}
 \|\check H_k\|_{\HK}
 \le C_\Omega\delta_*^{-k},
 \qquad k\ge1.
\end{equation}
For \(|\delta|\le\delta_*/2\), the normalized family satisfies
\begin{align}
 \|\nabla\check H_\delta-\nabla H\|_{C^0(\mathcal K)}
 &\le C_{\mathcal K}|\delta|,
 \label{eq:target-bias-grad-se}\\
 \|\check H_\delta-H\|_{C^0(\mathcal K)}
 &\le C_{\mathcal K}|\delta|,
 \label{eq:target-bias-fun-se}\\
 \|\check H_\delta-H\|_{\HK}
 &\le C_\Omega|\delta|.
 \label{eq:target-bias-native-se}
\end{align}
The last estimate uses Assumption~\ref{ass:domain}.  If
\(\nabla\langle H_q,H_p\rangle\) is not identically zero on \(\mathcal K\), the gradient order is
generically sharp.
\end{corollary}

\begin{proof}
Choose $0<r<\sigma$ and set
\[
 M_r:=\sup_{\substack{|\zeta|=\delta_*\\z\in\mathcal K_r}}
 |\check H_\zeta(z)|<\infty.
\]
Spatial Cauchy estimates on \(\mathcal K_r\), uniformly for \(|\zeta|=\delta_*\), followed by
Cauchy's coefficient formula in \(\delta\), give~\eqref{eq:imh-Cm-bound-se}.  Restricting these
bounds to \(\overline\Omega\), choosing an integer \(m>s\), and using
\(
 H^m(\Omega)\hookrightarrow H^s(\Omega)\cong\HK
\)
give~\eqref{eq:imh-native-coeff-bound-se}.  Summing the convergent coefficient series from
\(k=1\), using
\(
 \sum_{k\ge1}(|\delta|/\delta_*)^k\le2|\delta|/\delta_*
\), proves~\eqref{eq:target-bias-grad-se} and~\eqref{eq:target-bias-fun-se} from
\eqref{eq:imh-Cm-bound-se} with \(m=1\) and \(m=0\), respectively, and proves
\eqref{eq:target-bias-native-se} from~\eqref{eq:imh-native-coeff-bound-se}.  The sharpness statement
follows from~\eqref{eq:se-first-coefficient}.
\end{proof}

\begin{remark}[Inverse versus forward modified Hamiltonians for symplectic Euler]
The classical forward construction seeks a Hamiltonian whose exact flow represents a numerical
step and is commonly interpreted through an asymptotic series.  Here the direction is reversed:
the exact flow is given, and the mixed map determines a Hamiltonian whose symplectic Euler step
reproduces it.  The analytic family is an exact local object, and its convergent inverse modified
series is its Taylor series in \(\delta\).
\end{remark}

\section{Complex-analytic tools}

This appendix records the standard complex-analysis results used in the backward error analysis.
They are classical; see, for example, Conway~\cite{conway1978complex}. Collecting them here makes
the analytic inputs of Appendix~\ref{sec:bae} explicit.

\begin{proposition}[Cauchy integral formula on disks]\label{prop:appendix-cauchy-integral}
Let $F:\{z\in\CC:|z-z_0|<R\}\to X$ be holomorphic with values in a complex Banach space $X$. Then for every $0<r<R$ and every $|z-z_0|<r$,
\[
F(z)=\frac{1}{2\pi i}\int_{|\zeta-z_0|=r}\frac{F(\zeta)}{\zeta-z}\,d\zeta.
\]
In particular,
\[
\frac{F^{(m)}(z_0)}{m!}
=
\frac{1}{2\pi i}\int_{|\zeta-z_0|=r}\frac{F(\zeta)}{(\zeta-z_0)^{m+1}}\,d\zeta.
\]
\end{proposition}

\begin{proposition}[Cauchy estimates]\label{prop:appendix-cauchy-estimates}
Let $F:\{z\in\CC:|z-z_0|<R\}\to X$ be holomorphic, where $X$ is a complex Banach space. Then for every $0<r<R$ and every integer $m\ge 0$,
\[
\|F^{(m)}(z_0)\|
\le
\frac{m!}{r^m}\sup_{|\zeta-z_0|=r}\|F(\zeta)\|.
\]
More generally, if $F$ is holomorphic on a polydisk in $\CC^n$, then each mixed derivative is bounded by the corresponding product of reciprocal radii times the supremum norm of $F$ on the boundary polydisk.
\end{proposition}

\begin{proposition}[Analytic implicit function theorem]\label{prop:appendix-analytic-ift}
Let $G:\Omega\times U\to X$ be holomorphic, where $\Omega\subset\CC^m$, $U\subset X$ is open, and $X$ is a complex Banach space. Suppose
\[
G(\lambda_0,u_0)=0
\qquad\text{and}\qquad
D_u G(\lambda_0,u_0):X\to X
\]
is invertible. Then there exist neighborhoods $\Omega_0\ni\lambda_0$ and $U_0\ni u_0$ and a unique holomorphic map $u:\Omega_0\to U_0$ such that
\[
G(\lambda,u(\lambda))=0
\qquad\text{for all }\lambda\in\Omega_0.
\]
\end{proposition}

\begin{proposition}[Holomorphic dependence for analytic ODEs]\label{prop:appendix-holomorphic-ode}
Let $f:V\to\CC^n$ be holomorphic on an open set $V\subset\CC^n$, and consider
\[
\dot z(t)=f(z(t)),\qquad z(0)=z_0\in V.
\]
Then, for every compact set $K\Subset V$, there exists $\tau>0$ such that whenever $z_0\in K$, the solution exists for $|t|<\tau$ and depends holomorphically on both $t$ and $z_0$. In particular, the exact flow map $(t,z_0)\mapsto \varphi_t(z_0)$ admits a holomorphic extension to a complex neighborhood of $\{0\}\times K$.
\end{proposition}

\section{Glossary of symbols}

\noindent
\textbf{Hamiltonians and data.}

\medskip
\nopagebreak[4]
\noindent\begin{tabularx}{\textwidth}{@{}>{\raggedright\arraybackslash}p{0.30\textwidth}
>{\raggedright\arraybackslash}p{0.22\textwidth}>{\raggedright\arraybackslash}X@{}}
\toprule
\textbf{Symbol} & \textbf{Name} & \textbf{Meaning} \\
\midrule
$H$ & Physical Hamiltonian & $H:\RR^{2d}\to\RR$ \\
$\varphi_\delta^H$ & Exact flow & Time-$\delta$ flow of $\dot z=J\nabla H$ \\
$I_\delta^{h,\Int}$ & Integrator step & One step of method $\Int$ applied to $h$ \\
$\check H_\delta^\Int$ & Exact inverse target & Hamiltonian satisfying $I_\delta^{\check H_\delta^\Int,\Int}=\varphi_\delta^H$ \\
$x^\Int$ & Scheme-adapted anchor & $(z^-+z^+)/2$ for MP; $(q^-,p^+)$ for $p$-implicit SE \\
$v_\delta$ & Discrete velocity & $(z^+-z^-)/\delta$ \\
\bottomrule
\end{tabularx}

\bigskip\noindent
\textbf{Estimators.}

\medskip
\nopagebreak[4]
\noindent\begin{tabularx}{\textwidth}{@{}>{\raggedright\arraybackslash}p{0.30\textwidth}
>{\raggedright\arraybackslash}p{0.22\textwidth}>{\raggedright\arraybackslash}X@{}}
\toprule
\textbf{Symbol} & \textbf{Name} & \textbf{Meaning} \\
\midrule
$\widehat h_{\lambda,N,L}^{\Int}$ & Empirical estimator & Minimizer of the regularized empirical loss \\
$h_{\lambda,L}^{*,\Int}$ & Population minimizer & Minimizer of the corresponding population problem \\
$\widehat h_{\lambda,NL}^{U,W}$ & Perturbed estimator & Estimator formed from noisy anchors and discrete velocities \\
$\widehat h_{s,\delta}^{\Int}$ & Stride-$s$ estimator & Estimator constructed from endpoint pairs separated by $s\delta$ \\
$\widehat h_\delta^{\Int,(Q)}$ & Richardson estimator & $\sum_{s=1}^Q\alpha_s^\Int\widehat h_{s,\delta}^{\Int}$ \\
\bottomrule
\end{tabularx}

\bigskip\noindent
\textbf{Operators and distributions.}

\medskip
\nopagebreak[4]
\noindent\begin{tabularx}{\textwidth}{@{}>{\raggedright\arraybackslash}p{0.30\textwidth}
>{\raggedright\arraybackslash}p{0.22\textwidth}>{\raggedright\arraybackslash}X@{}}
\toprule
\textbf{Symbol} & \textbf{Name} & \textbf{Meaning} \\
\midrule
$A^\Int$ & Observation operator & $A^\Int h=J\nabla h$ as a map into $L^2(\rho^\Int;\RR^{2d})$ \\
$B^\Int=(A^\Int)^*A^\Int$ & Population covariance & Positive trace-class covariance operator \\
$A_{NL}^\Int$ & Empirical observation operator & Sampling at the $NL$ scheme-adapted anchors \\
$B_{NL}^\Int=(A_{NL}^\Int)^*A_{NL}^\Int$ & Empirical covariance & Empirical covariance operator \\
$\rho^\Int$ & Population input law & Averaged distribution of the scheme-adapted anchors \\
\bottomrule
\end{tabularx}

\bigskip\noindent
\textbf{Parameters and error quantities.}

\medskip
\nopagebreak[4]
\noindent\begin{tabularx}{\textwidth}{@{}>{\raggedright\arraybackslash}p{0.30\textwidth}
>{\raggedright\arraybackslash}p{0.22\textwidth}>{\raggedright\arraybackslash}X@{}}
\toprule
\textbf{Symbol} & \textbf{Name} & \textbf{Meaning} \\
\midrule
$J$ & Symplectic matrix & $\bigl[\begin{smallmatrix}0&I_d\\-I_d&0\end{smallmatrix}\bigr]$ \\
$\Int\in\{\MP,\SE\}$ & Integrator index & Implicit midpoint or $p$-implicit symplectic Euler \\
$N$ & Number of trajectories & Number of independent initial conditions \\
$L$ & Observations per trajectory & Number of consecutive endpoint pairs contributed by each trajectory \\
$\delta$ & Observation step & Time interval between consecutive observed states \\
$\lambda$ & Regularization parameter & Tikhonov parameter in the empirical loss \\
$\eta=\lambda/\delta^2$ & Effective regularization & Regularization parameter after rescaling by $\delta^2$ \\
$\gamma$ & Source exponent & Exponent in $\check H_\delta^\Int=(B^\Int)^\gamma\psi_\delta^\Int$ \\
$r_\Int$ & Integrator-error order & $r_{\MP}=2$ and $r_{\SE}=1$ \\
$\alpha_s^\Int$ & Richardson coefficient & Weight assigned to the stride-$s$ estimator \\
$\mathcal A_Q^\Int$ & Richardson amplification & $\sum_{s=1}^Q|\alpha_s^\Int|$ \\
$r_U,r_W$ & Pathwise noise levels & Anchor- and velocity-perturbation levels in~\eqref{eq:pathwise-noise-radii} \\
\bottomrule
\end{tabularx}
\end{document}